\documentclass[11pt]{article}
\usepackage{amsmath, amssymb, amsthm, amsfonts, graphicx,color} 
\usepackage{cite}
\usepackage[normalem]{ulem}
\usepackage{geometry} 
\usepackage{graphicx}
\usepackage{amscd}

\include{defs}

\numberwithin{equation}{section}
\newcommand{\eps}{\varepsilon} \newcommand{\TT}{{\mathbb T^2_\ell}}
\newcommand{\TTT}{{\mathbb T_{\ell^{\eps}}^2}}
\def\({\left(}
\def\){\right)}
\def\1{\mathbf{1}}
\def\Xint#1{\mathchoice
{\XXint\displaystyle\textstyle{#1}}%
{\XXint\textstyle\scriptstyle{#1}}%
{\XXint\scriptstyle\scriptscriptstyle{#1}}%
{\XXint\scriptscriptstyle\scriptscriptstyle{#1}}%
\!\int}
\def\XXint#1#2#3{{\setbox0=\hbox{$#1{#2#3}{\int}$ }
\vcenter{\hbox{$#2#3$ }}\kern-.6\wd0}}

\def\dashint{\Xint-}

\def\a{\alpha}

\def\D{\displaystyle}

\def\B{{\mathcal{B}}}

\def\curl{{\rm curl\,}}

\def\diam{\mathrm{diam}\ }
\def\dist{\text{dist}\ }
\def\div{\mathrm{div} \ }

\def\dt0{{{\frac{d}{dt}}_{|t=0}}}

\def\D{\displaystyle}

\def\ep{\varepsilon}

\def\F{{\mathcal{F}}}

\def\hal{\frac{1}{2}}

\def\h{{\eta}}

\def\indic{\mathbf{1}}
\def\indic{\mathbf{1}}

\def\lep{{|\, \mathrm{ln}\, \ep|}}

\def\loc{{\text{\rm loc}}}

\def\l|{\left|}

\def\mn{\mathbb{N}}
\def\mn{\mathbb{N}}
\def\mr{\mathbb{R}}

\def\mz{\mathbb{Z}}
\def\nab{\nabla}

\def\np{\nab^{\perp}}

\def\p{\partial}

\def\ro{\rho}
\def\r|{\right|}

\def\sm{\setminus}

\def\URC{{K_{R+C}}}
\def\URc{{K_{R-C}}}

\def\UR{{ K_R}}

\def\vp{\varphi}

\def\h{h_{\eps}'}
\newtheorem{theorem}{Theorem}
\newtheorem{coro}{Corollary}[section] 
\newtheorem{lem}[coro]{Lemma}

\newtheorem{pro}[coro]{Proposition}
\newtheorem{definition}[coro]{Definition}
\newtheorem{remark}[coro]{Remark}
\newtheorem{proposition}[coro]{Proposition}

\begin{document}

\title{\bf The $\mathbf \Gamma$-limit of the two-dimensional
  Ohta-Kawasaki energy. II. Droplet arrangement at the sharp interface
  level via the renormalized energy.}

\author{Dorian Goldman, Cyrill B. Muratov and Sylvia Serfaty}

\maketitle

\begin{abstract}
  This is the second in a series of papers in which we derive a
  $\Gamma$-expansion for the two-dimensional non-local Ginzburg-Landau
  energy with Coulomb repulsion known as the Ohta-Kawasaki model in
  connection with diblock copolymer systems.  In this model, two
  phases appear, which interact via a nonlocal Coulomb type
  energy. Here we focus on the sharp interface version of this energy
  in the regime where one of the phases has very small volume
  fraction, thus creating small ``droplets'' of the minority phase in
  a ``sea'' of the majority phase.  In our previous paper, we computed
  the $\Gamma$-limit of the leading order energy, which yields the
  averaged behavior for almost minimizers, namely that the density of
  droplets should be uniform. Here we go to the next order and derive
  a next order $\Gamma$-limit energy, which is exactly the Coulombian
  renormalized energy obtained by Sandier and Serfaty as a limiting
  interaction energy for vortices in the magnetic Ginzburg-Landau
  model. The derivation is based on the abstract scheme of
  Sandier-Serfaty that serves to obtain lower bounds for 2-scale
  energies and express them through some probabilities on patterns via
  the multiparameter ergodic theorem.  Without thus appealing to the
  Euler-Lagrange equation, we establish for all configurations which
  have ``almost minimal energy'' the asymptotic roundness and radius
  of the droplets, and the fact that they asymptotically shrink to
  points whose arrangement minimizes the renormalized energy in some
  averaged sense. Via a kind of $\Gamma$-equivalence, the obtained
  results also yield an expansion of the minimal energy for the
  original Ohta-Kawasaki energy.  This leads to expecting to see
  triangular lattices of droplets as energy minimizers.
\end{abstract}
  
\section{Introduction}
\label{sec:introduction}

This is our second paper devoted to the $\Gamma$-convergence study of
the two-dimensional Ohta-Kawasaki energy functional\cite{ohta86} in
two space dimensions in the regime near the onset of non-trivial
minimizers. The energy functional has the following form:
\begin{align}
  \label{EE}
 \mathcal{E}[u] = \int_\Omega \( \frac{\eps^2}{2} |\nabla u|^2 +
  V(u)\) dx + \frac{1}{2} \int_\Omega \int_\Omega (u(x)- \bar u)
  G_0(x, y) (u(y)- \bar u) \, dx \, dy,
\end{align}
where $\Omega$ is the domain occupied by the material, $u: \Omega \to
\mathbb R$ is the scalar order parameter, $V(u)$ is a symmetric
double-well potential with minima at $u = \pm 1$, such as the usual
Ginzburg-Landau potential $V(u) = \tfrac{9}{32} (1 - u^2)^2$ (for
simplicity, the overall coefficient in $V$ is chosen to make the
associated surface tension constant to be equal to $\eps$, i.e., we
have $\int_{-1}^1 \sqrt{2 V(u)} \,du = 1$), $\eps > 0$ is a parameter
characterizing interfacial thickness, $\bar u \in (-1, 1)$ is the
background charge density, and $G_0$ is the Neumann Green's function
of the Laplacian, i.e., $G_0$ solves
\begin{align}
  \label{G0}
  -\Delta G_0(x, y) = \delta(x - y) - {1 \over |\Omega|}, \qquad
  \int_\Omega G_0(x, y) \, dx = 0,
\end{align}
where $\Delta$ is the Laplacian in $x$ and $\delta(x)$ is the Dirac
delta-function, with Neumann boundary conditions. Note that $u$ is
also assumed to satisfy the ``charge neutrality'' condition
\begin{align}
  \label{neutr}
  {1 \over |\Omega|} \int_\Omega u \, dx = \bar u.
\end{align}
For a discussion of the motivation and the main quantitative features
of this model, see our first paper \cite{SCD}, as well as
\cite{m:cmp10,m:pre02}. For specific applications to physical systems,
we refer the reader to
\cite{degennes79,stillinger83,ohta86,nyrkova94,glotzer95,lundqvist,m:pre02,m:phd}.

In our first paper \cite{SCD}, we established the leading order term
in the $\Gamma$-expansion of the energy in \eqref{EE} in the scaling
regime corresponding to the threshold between trivial and non-trivial
minimizers. More precisely, we studied the behavior of the energy as
$\eps \to 0$ when
\begin{align}
  \label{ubeps}
  \bar u^\eps := -1 + \eps^{2/3} |\ln \eps|^{1/3} \bar\delta,
\end{align}
for some fixed $\bar \delta > 0$ and when $\Omega$ is a flat
two-dimensional torus of side length $\ell$, i.e., when $\Omega =
\mathbb T^2_\ell = [0, \ell)^2$, with periodic boundary conditions. As
follows from \cite[Theorem 2]{SCD} and the arguments in the proof of
\cite[Theorem 3]{SCD}, in this regime minimizers of $\mathcal E$
consist of many small ``droplets'' (regions where $u > 0$) and their
number blows up as $\ep \to 0$. We showed that, after a suitable
rescaling the energy functional in \eqref{EE} $\Gamma$-converges in
the sense of convergence of the (suitably normalized) droplet
densities, to the limit functional $E^0[\mu]$ defined for all
densities $\mu \in \mathcal M(\TT) \cap H^{-1}(\TT)$ by:
\begin{align}
  \label{E0mu}
  E^0[\mu] = \frac{\bar \delta^2 \ell^2}{2\kappa^2} + \left(3^{2/3} -
    \frac{2 \bar \delta}{\kappa^2} \right) \int_\TT d\mu + 2
  \iint_{\TT \times \TT} G(x - y) d \mu(x) d \mu(y),
\end{align}
where $G(x)$ is the {\em screened} Green's function of the Laplacian,
i.e., it solves the periodic problem for the equation
\begin{align}
  \label{G}
  -\Delta G + \kappa^2 G = \delta(x) \quad \text{in} \quad \TT,
\end{align}
and $\kappa = 1/\sqrt{V''(1)} = \frac23$. Here we noted that the
double integral in \eqref{E0mu} is well defined in the sense
$\iint_{\TT \times \TT} G(x - y) d \mu(x) d \mu(y) := \int_\TT v d
\mu$, where the latter is interpreted as the Hahn-Banach extension of
the corresponding linear functional defined by the integral on smooth
test functions (see also \cite[Sec. 7.3.1]{SS3} and \cite{brezis79}
for further discussion). Indeed, $v := G * d \mu$ is the convolution
understood distributionally, i.e., $\langle G*d \mu, f \rangle :=
\langle G*f, d \mu \rangle = \int_\TT \left( \int_\TT G(x - y) f(y) dy
\right) d \mu(x)$ for every $f \in C^\infty(\TT)$ and, hence, by
elliptic regularity $\| v \|_{H^1(\TT)} \leq C \| f \|_{H^{-1}(\TT)}$
for some $C > 0$, so $v \in H^1(\TT)$.

In particular, for $\bar \delta > \bar \delta_c$, where
\begin{align}
  \label{deltac}
  \bar\delta_c := \frac12 3^{2/3} \kappa^2,
\end{align}
the limit energy $E^0[\mu]$ is minimized by $d\mu(x) = \bar\mu \, dx$, where 
\begin{align}
  \label{mubar}
  \bar\mu = \tfrac12 (\bar\delta - \bar \delta_c) \qquad \text{and}
  \qquad E^0[\bar\mu] = \tfrac{\bar\delta_c}{2 \kappa^2} (2 \bar\delta
  - \bar\delta_c).
\end{align}
When $\bar \delta \leq \bar \delta_c$, the limit energy is minimized
by $\mu = 0$, with $E^0[0] = \bar \delta^2 / ( 2 \kappa^2)$. The value
of $\bar \delta = \bar \delta_c$ thus serves as the threshold
separating the trivial and the non-trivial minimizers of the energy in
\eqref{EE} together with \eqref{ubeps} for sufficiently small
$\eps$. Above that threshold, the droplet density of energy-minimizers
converges to the uniform density $\bar{\mu}$.
 
The key point that enables the analysis above is a kind of
$\Gamma$-equivalence between the energy functional in \eqref{EE} and
its screened sharp interface analog (for general notions of
$\Gamma$-equivalence or variational equivalence, see
\cite{braides08,ACP}):
\begin{equation}
  \label{E2}
  E^\eps[u] = \frac{\varepsilon}{2} \int_\TT  |\nabla u| \, dx +
  \frac{1}{2}  \int_\TT  \int_\TT  (u(x)-\bar u^\eps) G(x - y) (u(y)
  -\bar u^\eps) \, dx \, dy. 
\end{equation}
Here, $G$ is the screened potential as in \eqref{G}, and $u \in
\mathcal A$, where
\begin{align}
  \label{A}
  \mathcal A := BV(\TT; \{-1, 1\}), 
\end{align}
and we note that on the level of $E^\eps$ the neutrality condition in
\eqref{neutr} has been removed. As we showed in \cite{SCD},
following the approach of \cite{m:cmp10}, for $\mathcal E^\eps$ given
by \eqref{EE} in which $\bar u = \bar u^\eps$ and $\bar u^\eps$ is
defined in \eqref{ubeps}, we have
\begin{align}
  \label{EEEepsal}
  \min \mathcal E^\eps = \min E^\eps + O(\eps^\alpha \min E^\eps),
\end{align}
for some $\alpha > 0$. Therefore, in order to understand the leading
order asymptotic expansion of the minimal energy $\min \mathcal
E^\eps$ in terms of $|\ln \eps|^{-1}$, it is sufficient to obtain such
an expansion for $\min E^\eps$. This is precisely what we will do in
the present paper. 

In view of the discussion above, in this paper we concentrate our
efforts on the analysis of the sharp interface energy $E^\eps$ in
\eqref{E2}. An extension of our results to the original diffuse
interface energy $\mathcal E^\eps$ would lead to further technical
complications that lie beyond the scope of the present paper and will
be treated elsewhere. Here we wish to extract the next order
non-trivial term in the $\Gamma$-expansion of the sharp interface
energy $E^\eps$ after \eqref{E0mu}. In contrast to \cite{m:cmp10}, we
will not use the Euler-Lagrange equation associated to \eqref{E2}, so
our results about minimizers will also be valid for ``almost
minimizers'' (cf.  Theorem \ref{th2}).

We recall that for $\eps \ll 1$ the energy minimizers for $E^\eps$ and
$\bar \delta > \bar \delta_c$ consist of $O(|\ln \eps|)$ nearly
circular droplets of radius $r \simeq 3^{1/3} \eps^{1/3} |\ln
\eps|^{-1/3}$ uniformly distributed throughout $\TT$ \cite[Theorem
2.2]{m:cmp10}.  This is in contrast with the study of
\cite{choksi10,choksi11} for a closely related energy, where the
number of droplets remains bounded as $\ep \to 0$, and the authors
extract a limiting interaction energy for a finite number of points.

By $\Gamma$-convergence, we obtained in \cite[Theorem 1]{SCD} the
convergence of the droplet density of almost minimizers $(u^\eps)$ of
$E^\ep$:
\begin{align}
  \label{mueps1}
  \mu^\eps(x) := \frac{1}{2} \eps^{-2/3} |\ln
  \eps|^{-1/3}(1+u^\eps(x)),
\end{align}
to the uniform density $\bar \mu$ defined in \eqref{mubar}. However,
this result does not say anything about the microscopic placement of
droplets in the limit $\eps \to 0$. In order to understand the
asymptotic arrangement of droplets in an energy minimizer, our plan is
to blow-up the coordinates by a factor of $\sqrt{\lep}$, which is the
inverse of the scale of the typical inter-droplet distance, and to
extract the next order term in the $\Gamma$-expansion of the energy in
terms of the limits as $\ep \to 0$ of the blown-up configurations
(which will consist of an infinite number of point charges in the
plane with identical charge).

We will show that the arrangement of the limit point configurations is
governed by the Coulombic renormalized energy $W$, which was
introduced in \cite{SS3}. That energy $W$ was already derived as a
next order $\Gamma$-limit for the magnetic Ginzburg-Landau model of
superconductivity \cite{SS2,SS3}, and also for two-dimensional Coulomb
gases \cite{SS4}. Our results here follow the same method of
\cite{SS2}, and yield almost identical conclusions.

The ``Coulombic renormalized energy'' is a way of computing a total
Coulomb interaction between an infinite number of point charges in the
plane, neutralized by a uniform background charge (for more details
see Section \ref{sec2}).  It is shown in \cite{SS2} that its minimum
is achieved. It is also shown there that the minimum among simple
lattice patterns (of fixed volume) is uniquely achieved by the
triangular lattice (for a closely related result, see \cite{CO}), and
it is conjectured that the triangular lattice is also a global
minimizer. This triangular lattice is called ``Abrikosov lattice'' in
the context of superconductivity and is observed in experiments in
superconductors \cite{tinkham}.

The next order limit of $E^\ep$ that we shall derive below is in fact
the average of the energy $W$ over all limits of blown-up
configurations (i.e. average with respect to the blow up center).  Our
result says that limits of blow-ups of (almost) minimizers should
minimize this average of $W$. This permits one to distinguish between
different patterns at the microscopic scale and it leads, in view of
the conjecture above, to expecting to see triangular lattices of
droplets (in the limit $\ep \to 0$), around almost every blow-up
center (possibly with defects).  Note that the selection of triangular
lattices was also considered in the context of the Ohta-Kawasaki
energy by Chen and Oshita \cite{CO}, but there they were only obtained
as minimizers among simple lattice configurations consisting of
non-overlapping ideally circular droplets.

It is somewhat expected that minimizers of the Ohta-Kawasaki energy in
the macroscopic setting are periodic patterns in all space dimensions
(in fact in the original paper \cite{ohta86} only periodic patterns
are considered as candidates for minimizers). This fact has never been
proved rigorously, except in one dimension by M\"uller \cite{mu} (see
also \cite{ren00trusk,yip06}), and at the moment seems very difficult.
For higher-dimensional problems, some recent results in this direction
were obtained in \cite{ACO,spadaro09,m:cmp10} establishing
equidistribution of energy in various versions of the Ohta-Kawasaki
model on macroscopically large domains. Several other results
\cite{choksi10,choksi11,sternberg11,Cicalese} were also obtained to
characterize the geometry of minimizers on smaller domains.  The
results we obtain here, in the regime of small volume fraction and in
dimension two, provide more quantitative and qualitative information
(since we are able to distinguish between the cost of various
patterns, and have an idea of what the minimizers should be like) and
a first setting where periodicity can be expected to be proved.
  
The Ohta-Kawasaki setting differs from that of the magnetic
Ginzburg-Landau model in the fact that the droplet ``charges'' (i.e.,
their volume) are all positive, in contrast with the vortex degrees in
Ginzburg-Landau, which play an analogous role and can be both positive
and negative integers.  It also differs in the fact that the droplet
volumes are not quantized, contrary to the degrees in the
Ginzburg-Landau model. This creates difficulties and the major
difference in the proofs. In particular we have to account for the
possibility of many very small droplets, and we have to show that the
isoperimetric terms in the energy suffice to force (almost) all the
droplets to be round and of fixed volume. This has to be done at the
same time as the lower bound for the other term in the energy, for
example an adapted ``ball construction'' for non-quantized quantities
has to be re-implemented, and the interplay between these two effects
turns out to be delicate.

\medskip

Our paper is organized as follows. In Section \ref{sec2} we formulate
the problem and state our main results concerning the $\Gamma$-limit
of the next order term in the energy \eqref{E2} after the zeroth order
energy derived in \cite{SCD} is subtracted off. In Section
\ref{setup}, we derive a lower bound on this next order energy via an
energy expansion as done in \cite{SCD} however isolating lower order
terms obtained via the process. We then proceed via a ball
construction as in \cite{jerrardball,sandierball,SS3} to obtain lower
bounds on this energy in Section \ref{sec4} and consequently obtain an
energy density bounded from below with almost the same energy via
energy displacement as in \cite{SS2} in Section \ref{sec5}. In Section
\ref{sec6} we obtain explicit lower bounds on this density on bounded
sets in the plane in terms of the renormalized energy for a finite
number of points. We are then in the appropriate setting to apply the
multiparameter ergodic theorem as in \cite{SS2} to extend the lower
bounds obtained to global bounds, which we present at the end of
Section \ref{sec6}. Finally the corresponding upper bound (cf. Part
(ii) of Theorem \ref{main}) is presented in Section \ref{sec11}.

\paragraph{Some notations.} We use the notation $(u^\eps) \in \mathcal
A$ to denote sequences of functions $u^\eps \in \mathcal A$ as $\eps =
\eps_n \to 0$, where $\mathcal A$ is an admissible class. We also use
the notation $\mu \in \mathcal M(\Omega)$ to denote a positive finite
Radon measure $d \mu$ on the domain $\Omega$. With a slight abuse of
notation, we will often speak of $\mu$ as the ``density'' on $\Omega$
and set $d \mu(x) = \mu(x) dx$ whenever $\mu \in L^1(\Omega)$. With
some more abuse of notation, for a measurable set $E$ we use $|E|$ to
denote its Lebesgue measure, $|\partial E|$ to denote its perimeter
(in the sense of De Giorgi), and $\mu(E)$ to denote $\int_E d
\mu$. The symbols $H^1(\Omega)$, $BV(\Omega)$, $C^k(\Omega)$ and
$H^{-1}(\Omega)$ denote the usual Sobolev space, the space of
functions of bounded variation, the space of $k$-times continuously
differentiable functions, and the dual of $H^1(\Omega)$,
respectively. The symbol $o_\eps(1)$ stands for the quantities that
tend to zero as $\eps \to 0$ with the rate of convergence depending
only on $\ell$, $\bar \delta$ and $\kappa$.

\bigskip
\section{Problem formulation and main results}
\label{sec2}

In the following, we fix the parameters $\kappa > 0$, $\bar\delta > 0$
and $\ell > 0$, and work with the energy $E^\eps$ in \eqref{E2}, which
can be equivalently rewritten in terms of the connected components
$\Omega_i^\eps$ of the family of sets of finite perimeter $\Omega^\eps
:= \{u^\eps = +1\}$, where $(u^\eps) \in \mathcal A$ are almost
minimizers of $E^\eps$, for sufficiently small $\eps$ (cf. the
discussion at the beginning of Sec. 2 in \cite{SCD}). The sets
$\Omega^\eps$ can be decomposed into countable unions of connected
disjoint sets, i.e., $\Omega^\eps = \bigcup_i \Omega_i^\eps$, whose
boundaries $\partial \Omega_i^\eps$ are rectifiable and can be
decomposed (up to negligible sets) into countable unions of disjoint
simple closed curves.  Then the density $\mu^\eps$ in \eqref{mueps1}
can be rewritten as
\begin{align}
  \label{mu}
  \mu^\eps(x) := \eps^{-2/3} |\ln \eps|^{-1/3} \sum_i
  \chi_{\Omega_i^\eps}(x),
\end{align}
where $\chi_{\Omega_i^\eps}$ are the characteristic functions of
$\Omega_i^\eps$.  Motivated by the scaling analysis in the discussion
preceding equation \eqref{mueps1}, we define the rescaled areas and
perimeters of the droplets: 
\begin{align}
  \label{AP}
  A_i^\eps := \eps^{-2/3} |\ln \eps|^{2/3} |\Omega_i^\eps|, \qquad
  P_i^\eps := \eps^{-1/3} |\ln \eps|^{1/3} |\partial \Omega_i^\eps|.
\end{align}
Using these definitions, we obtain (see \cite{SCD,m:cmp10}) the
following equivalent definition of the energy of the family
$(u^\eps)$:
\begin{align}
  E^\eps[u^\eps] = \eps^{4/3} |\ln \eps|^{2/3} \left( {\bar \delta^2
      \ell^2 \over 2 \kappa^2} + \bar E^\eps[u^\eps]
  \right), \label{EEE}
\end{align}
where
\begin{align}
  \bar E^\eps[u^\eps] := {1 \over |\ln \eps|} \sum_i \left( P_i^\eps -
    {2 \bar \delta \over \kappa^2} A_i^\eps \right) + 2 \iint_{\TT
    \times \TT} G(x - y) d \mu^\eps(x) d \mu^\eps(y). \label{Ebar}
\end{align}
Also note the relation
\begin{align}
  \label{Amu}
  \mu^\eps(\TT) = {1 \over |\ln \eps|} \sum_i A_i^\eps.
\end{align}

As was shown in \cite{m:cmp10,SCD}, in the limit
$\eps \to 0$ the minimizers of $E^\eps$ are non-trivial if and only if
$\bar \delta > \bar \delta_c$, and we have asymptotically
\begin{align}
  \label{minE}
  \min E^\eps \simeq {\bar\delta_c \over 2 \kappa^2} (2 \bar \delta -
  \bar \delta_c) \eps^{4/3} |\ln \eps|^{2/3} \ell^{2} \qquad \text{as
  } \eps \to 0.
\end{align}
Furthermore, if $\mu^\eps$ is as in \eqref{mu} and we let $v^\eps$ be
the unique solution of
\begin{equation}\label{vdef}
  -\Delta v^{\varepsilon} + \kappa^2 v^{\varepsilon} = \mu^\eps
  \qquad \text{in} \quad W^{2,p}(\TT), 
\end{equation}
for any $p < \infty$, then we have 
\begin{align}
  \label{vepslim}
  v^\eps \rightharpoonup \bar v := {1 \over 2 \kappa^2} (\bar \delta -
  \bar \delta_c) \qquad \text{in} \quad H^1(\TT).
\end{align}
To extract the next order terms in the $\Gamma$-expansion of $E^\eps$
we, therefore, subtract this contribution from $E^\eps$ to define a
new rescaled energy $F^\eps$ (per unit area):
\begin{equation}
  \label{F}
  F^\eps[u] := \eps^{-4/3} |\ln \eps|^{1/3}  \ell^{-2} E^\eps[u]
  -  \lep    {\bar\delta_c \over 2 \kappa^2} (2 \bar \delta - \bar \delta_c)
  +  {1 \over 4 \cdot 3^{1/3}} (\bar \delta - \bar \delta_c) (\ln 
  |\ln \eps| + \ln 9).
\end{equation}
Note that we also added the third term into the bracket in the
right-hand side of \eqref{F} to subtract the next-to-leading order
contribution of the droplet self-energy, and we have scaled $F^\eps$
in a way that allows to extract a non-trivial $O(1)$ contribution to
the minimal energy (see details in Section \ref{setup}).  The main
result of this paper in fact is to establish $\Gamma$-convergence of
$F^\eps$ to the renormalized energy $W$ which we now define.

In \cite{SS2}, the renormalized energy $W$ was introduced and defined
in terms of the superconducting current $j$, which is particularly
convenient for the studies of the magnetic Ginzburg-Landau model of
superconductivity. Here, instead, we give an equivalent definition,
which is expressed in terms of the limiting electrostatic potential of
the charged droplets, after blow-up, which is the limit of some proper
rescaling of $v^\ep$ (see below).  However, this limiting
electrostatic potential will only be known up to additive constants,
due to the fact that we will take limits over larger and larger tori.
This issue can be dealt with in a natural way by considering {\em
  equivalence classes} of potentials, whereby two potentials differing
by a constant are not distinguished:
\begin{align}
  \label{vequiv}
  [\varphi] := \{ \varphi + c \ | \ c \in \mathbb R \}.
\end{align}
This definition turns the homogeneous spaces $\dot{W}^{1,p}(\mathbb
R^d)$ into Banach spaces of equivalence classes of functions in
$W^{1,p}_{loc}(\mathbb R^d)$ defined in \eqref{vequiv} (see, e.g.,
\cite{ortner12}). Here we similarly define the local analog of the
homogeneous Sobolev spaces as
\begin{align}
  \label{H1dot}
  \dot{W}^{1,p}_{loc}(\mathbb R^2) := \left\{ [\varphi] \ | \ \varphi
    \in W^{1,p}_{loc}(\mathbb R^2) \right\},
\end{align}
with the notion of convergence to be that of the $L^p_{loc}$
convergence of gradients.  In the following, we will omit the brackets
in $[ \cdot ]$ to simplify the notation and will write $\varphi \in
\dot{W}^{1,p}_{loc}(\mathbb R^2)$ to imply that $\varphi$ is any member of
the equivalence class in \eqref{vequiv}.

We define the admissible class of the renormalized energy as follows :
\begin{definition}\label{defam}
  For given $m > 0$ and $p \in (1,2)$, we say that $\varphi$ belongs
  to the admissible class $\mathcal A_m$, if $\varphi \in
  \dot{W}^{1,p}_{loc}(\mathbb R^2)$ and $\varphi$ solves
  distributionally 
  \begin{align}
    \label{vm}
    -\Delta \varphi = 2 \pi \sum_{a \in \Lambda} \delta_a - m,
  \end{align}
  where $\Lambda \subset \mathbb R^2$ is a discrete set and
  \begin{align}
    \label{Lamm}
    \lim_{R \to \infty} {2 \over R^2} \int_{B_R(0)} \sum_{a \in
      \Lambda} \delta_a(x) dx = m.
  \end{align}
\end{definition}

\begin{remark}
Observe that if $\varphi \in \mathcal A_m$, then for every $x \in B_R(0)$
we have 
\begin{align}
  \label{varphi}
  \varphi(x) = \sum_{a \in \Lambda_R} \ln |x - a|^{-1} + \varphi_R(x),
\end{align}
where $\Lambda_R \ := \Lambda \cap \bar B_R(0)$ is a finite set of
distinct points and $\varphi_R \in C^\infty(\mathbb R^2)$ is analytic
in $B_R(0)$. In particular, the definition of $\mathcal A_m$ is
independent of $p$.
\end{remark}

We next define the renormalized energy.

\begin{definition}\label{Wvdef}
  For a given $\varphi \in \displaystyle\bigcup_{m > 0} \mathcal A_m$,
  the renormalized energy $W$ of $\varphi$ is defined as
  \begin{align}
    \label{Wv}
    W(\varphi) := \limsup_{R \to \infty} \lim_{\eta \to 0} {1 \over
      |K_R| } \left( \int_{\mathbb R^2 \backslash \cup_{a \in \Lambda}
        B_\eta(a)} \frac12 |\nabla \varphi|^2 \chi_R dx + \pi \ln \eta
      \sum_{a \in \Lambda} \chi_R(a) \right),
  \end{align}
  where $K_R = [-R,R]^2$, $\chi_R$ is a smooth cutoff function with
  the properties that $0 < \chi_R < 1$, in $K_R \sm (\partial K_R \cup
  K_{R-1})$, $\chi_R(x) = 1$ for all $x \in K_{R-1}$, $\chi_R(x) = 0$
  for all $x \in \mathbb R^2 \backslash K_R$, and $|\nabla \chi_R|
  \leq C$ for some $C > 0$ independent of $R$.
\end{definition}

Various properties of $W$ are established in \cite{SS2}, we refer the
reader to that paper. The most relevant to us here are
\begin{enumerate}
\item $\min_{\mathcal{A}_m}W$ is achieved for each $m > 0$.

\item If $\vp \in \mathcal{A}_m$ and $\vp' (x) :=
  \vp(\frac{x}{\sqrt{m}})$, then $\varphi' \in \mathcal{A}_1$ and
\begin{equation}\label{scalingW}
  W(\vp) = m \(  W(\vp')- \frac14 \log m\),
\end{equation}
hence  
$$ 
\min_{\mathcal{A}_m}W = m \( \min_{\mathcal{A}_1} W- \frac14 \log m\).
$$

\item $W$ is minimized over potentials in $\mathcal A_1$ generated by
  charge configurations $\Lambda$ consisting of simple lattices by the
  potential of a triangular lattice, i.e. \cite[Theorem 2 and Remark
  1.5]{SS2},
$$
\min_{\stackrel{\varphi \in \mathcal A_1}{\Lambda \text{ simple
      lattice}}} W(\varphi) = W(\varphi^\triangle) = - \frac12 \ln
(\sqrt{2 \pi b} \, |\eta(\tau)|^2) \simeq -0.2011,
$$
where $\tau = a + i b$, $\eta(\tau) = q^{1/24} \prod_{n \geq 1} (1 -
q^n)$ is the Dedekind eta function, $q = e^{2 \pi i \tau}$, $a$ and
$b$ are real numbers such that $\Lambda^*_\triangle = {1 \over \sqrt{2
    \pi b}} \Big( (1,0) \mathbb Z \oplus (a, b) \mathbb Z \Big)$ is
the dual lattice to a triangular lattice $\Lambda^\triangle$ whose
unit cell has area $2 \pi$, and $\varphi^\triangle$ solves \eqref{vm}
with $\Lambda = \Lambda^\triangle$.
\end{enumerate}
In particular, from property 2 above it is easy to see that the role
of $m$ in the definition of $W$ is inconsequential.

We are now ready to state our main result.  Let $\ell^\eps := |\ln
\eps|^{1/2} \ell$. For a given $u^\eps \in \mathcal A$, we then
introduce the potential (recall that $\varphi^\eps$ is a
representative in the equivalence class defined in \eqref{vequiv})
\begin{align}
  \label{phiepsdef}
  \varphi^\eps(x) := 2 \cdot 3^{-2/3} |\ln \eps| \ \tilde v^\eps(x |\ln
  \eps|^{-1/2}),
\end{align} 
where $\tilde v^\eps$ is a periodic extension of $v^\eps$ from
$\mathbb T^2_{\ell^\eps}$ to the whole of $\mathbb R^2$.  We also
define $\mathcal{P}$ to be the family of translation-invariant
probability measures on $\dot{W}^{1,p}_{loc}(\mathbb R^2)$
concentrated on $\mathcal A_m$ with $m = 3^{-2/3} (\bar \delta - \bar
\delta_c)$.

\begin{theorem} \textbf{\emph{($\Gamma$-convergence of
      $F^{\varepsilon}$)}}
  \label{main}
  \noindent Fix $\kappa > 0$, $\bar \delta > \bar \delta_c$, $p \in
  (1,2)$ and $\ell > 0$, and let $F^{\varepsilon}$ be defined by
  (\ref{F}). Then, as $\eps \to 0$ we have
  \begin{align}
    \label{F0}
    F^{\varepsilon} \stackrel{\Gamma}\to F^0[P] := 3^{4/3} \int
    W(\varphi) dP(\varphi) + \frac{3^{2/3}(\bar \delta - \bar \delta_c)}{8},
  \end{align}
  where $P \in \mathcal{P}$. More precisely:
  \begin{itemize}
  \item[i)] (Lower Bound) Let $(u^\eps) \in \mathcal A$ be such that
    \begin{align}
      \label{ZO.1}
      \limsup_{\varepsilon \to 0} F^{\varepsilon}[u^{\varepsilon}] <
      +\infty,
    \end{align}
    and let $P^{\varepsilon}$ be the probability measure on
    $\dot{W}^{1,p}_{loc}(\mathbb{R}^2)$ which is the pushforward of
    the normalized uniform measure on
    $\mathbb{T}_{\ell^\varepsilon}^2$ by the map $x \mapsto
    \varphi^{\varepsilon}(x+\cdot)$, where $\varphi^\eps$ is as in
    \eqref{phiepsdef}.  Then, upon extraction of a subsequence,
    $(P^\eps)$ converges weakly to some $P\in \mathcal{P}$, in the
    sense of measures on $\dot{W}^{1,p}_{loc}(\mathbb R^2)$ and
    \begin{align}
      \label{Lbound1}
      \liminf_{\eps \to 0} F^\eps[u^\eps] \geq F^0[P].
    \end{align}

  \item[ii)] (Upper Bound) Conversely, for any probability measure
    $P\in \mathcal{P}$, letting $Q$ be its push-forward under
    $-\Delta$, there exists $(u^{\varepsilon}) \in \mathcal A$ such
    that letting $Q^{\varepsilon}$ be the pushforward of the
    normalized Lebesgue measure on $\mathbb{T}_{\ell^\varepsilon}^2$
    by $x \mapsto -\Delta \varphi^{\varepsilon} \left(x +
      \cdot\right)$, where $\varphi^\eps$ is as in \eqref{phiepsdef},
    we have $Q^{\varepsilon} \rightharpoonup Q$, in the sense of
    measures on $W^{-1,p}_{loc}(\mathbb R^2)$, and
    \begin{equation}\label{Ubound1}
      \limsup_{\eps \to 0} F^{\varepsilon}[u^{\varepsilon}] \leq F^0[P],
    \end{equation}
    as $\varepsilon \to 0$. 
  \end{itemize}
\end{theorem}
\noindent
We will prove that the minimum of $F^0$ is achieved. Moreover, it is
achieved for any $P \in \mathcal{P}$ which is concentrated on
minimizers of $\mathcal A_m$ with $m = 3^{-2/3} (\bar \delta - \bar
\delta_c)$.

\begin{remark}
  The phrasing of the theorem does not exactly fit the framework of
  $\Gamma$-convergence, since the lower bound result and the upper
  bound result are not expressed with the same notion of
  convergence. However, since weak convergence of $P_\ep$ to $P$
  implies weak convergence of $Q_\ep$ to $Q$, the theorem implies a
  result of $\Gamma$-convergence where the sense of convergence from
  $P_\ep$ to $P$ is taken to be the weak convergence of their
  push-forwards $Q_\ep$ to the corresponding $Q$.
\end{remark}

The next theorem expresses the consequence of Theorem
\ref{main} for almost minimizers:
\begin{theorem}\label{th2}
  Let $m = 3^{-2/3} (\bar \delta - \bar \delta_c)$ and let $(u^{\eps})
  \in \mathcal{A}$ be a family of almost minimizers of $F^0$, i.e.,
  let
$$
\lim_{\ep \to 0} F^{\varepsilon}[u^{\varepsilon}] = \min_{\mathcal{P}}
F^0.
$$ 
Then, if $P$ is the limit measure from Theorem \ref{main}, $P$-almost
every $\varphi$ minimizes $W$ over $\mathcal{A}_m$. In addition
\begin{align}
  \label{F0min}
  \min_{\mathcal{P}} F^0 = 3^{4/3} \min_{\mathcal A_m} W +
  \frac{3^{2/3} (\bar \delta - \bar \delta_c)}{8}.
\end{align}
\end{theorem}
\noindent Note that the formula in \eqref{F0min} is not totally
obvious, since the probability measure concentrated on a single
minimizer $\varphi \in \mathcal A_m$ of $W$ does not belong to
$\mathcal P$.

The result in Theorem \ref{th2} allows us to establish the expansion
of the minimal value of the original energy $\mathcal E^\eps$ by
combining it with \eqref{F} and \eqref{EEEepsal}.

\begin{theorem} \textbf{\emph{(Asymptotic expansion of $\min \mathcal
      E^\eps$)}}
  \label{t:minEE}
  \noindent Let $V = \frac{9}{32} (1-u^2)^2$, $\kappa = \frac23$ and
  $m = 3^{-2/3} (\bar \delta - \bar \delta_c)$. Fix $\bar \delta >
  \bar \delta_c$ and $\ell > 0$, and let $\mathcal E^\eps$ be defined
  by (\ref{EE}) with $\bar u = \bar u^\eps$ from (\ref{ubeps}). Then,
  as $\eps \to 0$ we have
  \begin{align}
    \label{EEmin}
    \ell^{-2} \min \mathcal E^\eps = {\bar\delta_c \over 2 \kappa^2}
    (2 \bar \delta - \bar \delta_c) \eps^{4/3} |\ln \eps|^{2/3} - {1
      \over 4 \cdot 3^{1/3}} (\bar \delta - \bar \delta_c) \eps^{4/3}
    |\ln \eps|^{-1/3} (\ln |\ln \eps| + \ln 9) \nonumber \\+
    \eps^{4/3} |\ln \eps|^{-1/3}\( 3^{4/3} \ \min_{\mathcal A_m} \ W +
    \frac{3^{2/3} (\bar \delta - \bar \delta_c)}{8} \)+ o(\eps^{4/3}
    |\ln \eps|^{-1/3}).
  \end{align}
\end{theorem}

As mentioned above, the $\Gamma$-limit in Theorem \ref{main} cannot be
expressed in terms of a single limiting function $\vp$, but rather it
effectively averages $W$ over all the blown-up limits of $\vp^\ep$,
with respect to all the possible blow-up centers.  Consequently, for
almost minimizers of the energy, we cannot guarantee that each
blown-up potential $\vp^\ep$ converges to a minimizer of $W$, but only
that this is true after blow-up except around points that belong to a
set with asymptotically vanishing volume fraction.  Indeed, one could
easily imagine a configuration with some small regions where the
configuration does not ressemble any minimizer of $W$, and this would
not contradict the fact of being an almost minimizer since these
regions would contribute only a negligible fraction to the energy.
Near all the good blow-up centers, we will know some more about the
droplets: it will be shown in Theorem \ref{Wlbound} that they are
asymptotically round and of optimal radii.
    
We finish this section with a short sketch of the proof. Most of the
proof consists in proving the lower bound, i.e. Part (i) of Theorem
\ref{main}.  The first step, accomplished in Section \ref{setup} is,
following the ideas of \cite{m:cmp10}, to extract from $F^\ep$ some
positive terms involving the sizes and shapes of the droplets and
which are minimized by round droplets of fixed appropriate
radius. These positive terms, gathered in what will be called $M_\ep$,
can be put aside and will serve to control the discrepancy between the
droplets and the ideal round droplets of optimal sizes.  We then
consider what remains when this $M_\ep$ is subtracted off from $F^\ep$
and express it in blown-up coordinates $x'= x \sqrt{\lep}$. It is then
an energy functional, expressed in terms of some rescaling of
$\vp^\ep$ which has no sign and which ressembles that studied in
\cite{SS2}.  Thus we apply to it the strategy of \cite{SS2}. The main
point is to show that, even though the energy density is not bounded
below, it can be transformed into one that is by absorbing the
negative terms into positive terms in the energy in the sense of
energy displacement \cite{SS2}, while making only a small error. In
order to prove that this is possible, we first need to establish sharp
lower bounds for the energy carried by the droplets (with an error
$o(1)$ per droplet).  These lower bounds contain possible errors which
will later be controlled via the $M_\ep$ term.  This is done in
Section \ref{sec4} via a ball construction as in
\cite{jerrardball,sandierball,SS3}.  In Section \ref{sec5} we use
these lower bounds to perform the energy displacement as in
\cite{SS2}. Once the energy density has been replaced this way by an
essentially equivalent energy density which is bounded below, we can
apply the abstract scheme of \cite{SS2} that serves to obtain lower
bounds for ``two-scale energies'' which $\Gamma$-converge at the
microscopic scale, via the multiparameter ergodic theorem. This is
achieved is Section \ref{sec6}.  Prior to this we obtain explicit
lower bounds at the microscopic scale in terms of the renormalized
energy for a finite number of points. It is then these lower bounds
that get integrated out, or averaged out at the macroscopic scale to
provide a global lower bound.

Finally, there remains to obtain the corresponding upper bound. This
is done via an explicit construction of a periodic test-configuration,
following again the method of \cite{SS2}.

\section{Derivation of the leading order energy}\label{setup}

In preparation for the proof of Theorem \ref{main}, we define
\begin{align}
  \rho_{\eps} := 3^{1/3}\eps^{1/3} |\ln \eps|^{1/6} \;\;\textrm{ and
  }\;\;\; \bar r_{\varepsilon} := \(\frac{|\ln \varepsilon|}{|\ln
    \rho_{\eps}|}\)^{1/3} .\label{newRdef1}
\end{align}
Recall that to leading order the droplets are expected to be circular
with radius $3^{1/3} \ep^{1/3} \lep^{-1/3}$. Thus $\ro_{\eps}$ is the
expected radius, once we have blown up coordinates by the factor of
$\sqrt{|\ln \eps|}$, which will be done below. Also, we know that the
expected normalized area $A_i$ is $3^{2/3} \pi$, but this is only true
up to lower order terms which were negligible in \cite{SCD}; as we
show below, a more precise estimate is $A_i \simeq \pi \bar
r_{\varepsilon}^2$, so $\bar{r}_\ep$ above can be viewed as a
``corrected'' normalized droplet radius.  Since our estimates must be
accurate up to $o_{\eps}(1)$ per droplet and the self-energy of a
droplet is of order $A_i^2 \ln \ro_\ep $, we can no longer ignore
these corrections.

%00The second term in the above equality explains the origin of the third term appearing in $F$ (cf. equation \eqref{F}).
% \noindent It is shown in \cite{Muratov} that \eqref{E2} can be written as
%\begin{equation}\label{Erescale} E^{\eps}(u) = \varepsilon^{4/3} |\ln \varepsilon|^{2/3}\( \bar E^{\eps}(\Omega) + \frac{\bar \delta^2}{2\kappa^2}\),\end{equation}
%where $\Omega := \{u > 0\}$ and $\bar E^{\eps}$ is to be explicited below. By the results of \cite{ambrosio2}, $\Omega$ can be decomposed into a countable union of connected, rectifiable sets, i.e. $\Omega = \bigcup_i \Omega_i$, whose
%boundaries $\partial \Omega_i$ are countable unions of disjoint simple curves. As in \cite{Muratov}, we write $\bar E^{\eps}$ for any set $\Omega$ as
%\begin{equation}\label{Edef2}
%\bar E^{\eps}(\Omega) :=\frac{1}{|\ln \varepsilon|} \sum_{i}\( P_i^{\eps} - \frac{2 \bar \delta}{\kappa^2} A_i^{\eps}\) + 2 \int_{\mathbb{T}_{\ell}^2} |\nabla v|^2 + \kappa^2 |v|^2 dx,
%\end{equation}
%where $A_i^{\eps}$, $P_i^{\eps}$ are rescaled perimeter and area defined as
%\begin{align}
%A_i^{\eps} &= \varepsilon^{-2/3}|\ln \varepsilon|^{2/3}|\Omega_i|,\\
%P_i^{\eps} &= \varepsilon^{-1/3} |\ln \varepsilon|^{1/3} |\partial\Omega_i|, \label{isorescale}
%\end{align}
%and $v$ is the unique solution to the screened Poisson equation:
%\begin{equation}\label{vdef2}
% -\Delta v + \kappa^2 v= \frac{1}{2} \varepsilon^{-2/3} |\ln \eps|^{-1/3} (1+ u) =\frac{1}{|\ln \varepsilon|}\sum_{i} A_i \frac{\chi_{\Omega_i}}{|\Omega_i|} =: \mu.
%\end{equation}

The goal of the next subsection is to obtain an explicit lower bound
for $F_{\eps}$ defined by \eqref{F} in terms of the droplet areas and
perimeters, which will then be studied in Sections \ref{sec4} and
onward. We follow the analysis of \cite{SCD}, but isolate higher order
terms.
%\section{First order energy extraction}\label{sec3}

\subsection{Energy extraction}

We begin with the original energy $\bar E^{\eps}$ (cf. \eqref{Ebar})
while adding and subtracting the \emph{truncated} self interaction:
first we define, for $\gamma \in (0,1)$, truncated droplet volumes by
\begin{equation} \label{P1.N34}  \tilde A_i^{\ep} :=
  \begin{cases}
    A_i^{\ep} & \mathrm{if} \  A_i^{\ep} <  3^{2/3} \pi \gamma^{-1}, \\
    (3^{2/3} \pi \gamma^{-1} |A_i^{\ep}|)^{1/2} & \mathrm{if} \
    A_i^{\ep} \geq 3^{2/3} \pi \gamma^{-1},
  \end{cases}
\end{equation}
as in \cite{SCD}. The motivation for this truncation will become clear
in the proof of Proposition~\ref{wspread}, when we obtain lower bounds
on the energy on annuli.  In \cite{SCD} the self-interaction energy of
each droplet extracted from $\bar E^{\eps}$ was $\frac{|\tilde
  A_i^{\ep}|^2}{3\pi|\ln \varepsilon|}$, yielding in the end the
leading order energy $E^0[\mu]$ in \eqref{E0mu}. A more precise
calculation of the self-interaction energy corrects the coefficient of
$|\tilde A_i^\ep|^2$ by an $O(\ln |\ln \eps| / |\ln \eps|)$ term,
yielding the following corrected leading order energy for $E^\ep$:
\begin{equation}\label{RZorder}
  E_{\eps}^0[\mu] := \frac{\bar \delta^2 \ell^2}{2\kappa^2} +
  \(\frac{3}{\bar r_{\eps}} - \frac{2\bar 
    \delta}{\kappa^2}\)\int_{\mathbb{T}_{\ell}^2} d\mu + 2
  \iint_{\mathbb{T}_{\ell}^2 \times \mathbb{T}_{\ell}^2} G(x-y)
  d\mu(x)d\mu(y).
\end{equation}
The energy in \eqref{RZorder} is explicitly minimized by $d\mu(x) =
\bar\mu_\eps \, dx$ (again a correction to the previously known
$\bar{\mu}$ from \eqref{mubar}) where
\begin{align}\label{mueps}
  \bar\mu_\eps := \frac12 \left( \bar \delta - {3 \kappa^2 \over 2
      \bar r_\eps} \right) \qquad \text{for} \qquad \bar \delta > {3
    \kappa^2 \over 2 \bar r_\eps},
\end{align}
and 
\begin{align} \label{3.6}
  \min E^0_\eps = {\bar \delta_c \ell^2 \over 2 \kappa^2 } \left\{ 2 \bar
    \delta \left( { 3 \over \bar r_\eps^3} \right)^{1/3} - \bar
    \delta_c \left( { 3 \over \bar r_\eps^3} \right)^{2/3} \right\}.
\end{align}
Observing that $\bar r_\ep \to 3^{1/3}$ we immediately check that 
\begin{equation}\label{convermubar}
  \bar \mu_\ep\to \bar \mu \quad \text{as} \ \ep \to 0,
\end{equation}
and in addition that \eqref{3.6} converges to the second expression in
\eqref{mubar}. To obtain the next order term, we Taylor-expand the
obtained formulas upon substituting the definition of
$\bar{r}_\ep$. After some algebra, we obtain
\begin{align}
  \ell^{-2} \min E^0_\eps = {\bar \delta_c \over 2 \kappa^2} \left( 2
    \bar \delta - \bar \delta_c \right) - {1 \over 4 \cdot 3^{1/3}} (
  \bar \delta - \bar \delta_c) {\ln |\ln \eps| + {\ln 9} \over |\ln
    \eps|} + O \left( {(\ln |\ln \eps|)^2 \over |\ln \eps|^2} \right).
\end{align}
Recalling once again the definition of $F^{\eps}$ from \eqref{F}, we
then find
$$ F^\ep[u^\ep]= \lep \( \ep^{-4/3} \lep^{-2/3} \ell^{-2}
E^\ep[u^\eps] - \ell^{-2} \min E_\ep^0 \) + O\( {(\ln |\ln \eps|)^2
  \over |\ln \eps|} \right),
$$
and in view of the definition of $\bar E^{\eps}$ from \eqref{EEE}, we
thus may write
\begin{align}\label{Eexpan}
  F^\ep[u^\ep]= |\ln \eps|\ell^{-2} \( \bar E^{\eps}[u^{\eps}] +
  \frac{ \bar \delta^2\ell^2}{2\kappa^2} -\min E^0_\eps\) + O \left(
    {(\ln |\ln \eps|)^2 \over |\ln \eps|} \right).
\end{align}
Thus obtaining a lower bound for the first term in the right-hand side
of \eqref{Eexpan} implies, up to $o_{\eps}(1)$, a lower bound for
$F^{\eps}$. This is how we proceed to prove Lemma \ref{lem1} below.

With this in mind, we begin by setting
\begin{equation}\label{hdeffirst}
  v^{\eps} = \bar v^{\ep} + \frac{h_{\eps}}{|\ln \eps|}, \qquad \bar v^{\ep} =
  \frac{1}{2\kappa^2} \(\bar \delta - \frac{3 \kappa^2}{2\bar
    r_{\eps}}\),
\end{equation}
where $\bar v^{\ep}$ is the solution to \eqref{vdef} with right side
equal to $\bar \mu_\eps$ in \eqref{mueps}.

\subsection{Blowup of coordinates}

We now rescale the domain $\mathbb{T}_{\ell}^2$ by making the change
of variables
\begin{align} \nonumber x' &= x \sqrt{|\ln \varepsilon|}, \\
  \label{231} \h(x') &= h_{\varepsilon}  (x), \\
  \nonumber \Omega_{i,\eps}' &={\Omega}_i^\eps \sqrt{\lep}, \\
  \nonumber \ell^{\eps} &= \ell \sqrt{\lep}.
\end{align}
%We also denote $e_{\ep}'$ as the energy density $e_{\ep}$ in blown up coordinates:
%\begin{equation}\label{hdensityblownup}
%e_{\ep}'(x') =  2\(|\nabla \h (x')|^2
%+ \frac{\kappa^2}{|\ln \varepsilon|} (\h)^2\).
%\end{equation}
Observe that
\begin{equation}\label{phih}
  \varphi^{\eps}(x') = 2 \cdot 3^{-2/3} \h (x') \qquad
  \forall x' \in \mathbb{T}_{\ell^\ep}^2,
\end{equation}
where $\varphi^{\eps}$ is defined by \eqref{phiepsdef}. It turns out
to be more convenient to work with $\h$ and rescale 
only at the end back to $\varphi^{\eps}$.

\subsection{Main result}

We are now ready to state the main result of this section, which
provides an explicit lower bound on $F^{\eps}$. The strategy, in
particular for dealing with droplets that are too small or too large
is the same as \cite{SCD}, except that we need to go to higher order
terms.

\begin{pro} \label{lem1} There exist universal constants $\gamma \in
  (0, \frac16)$, $ c_1>0, c_2>0$, $c_3>0$ and $\eps_0 > 0$ such that
  if $\bar\delta >\bar\delta_c$ and $(u^{\eps}) \in \mathcal{A}$ with
  $\Omega^{\eps} := \{u^{\eps} > 0\}$, then for all $\ep < \eps_0$
\begin{align}\label{minoF}
  \ell^2 F^{\eps}[u^{\eps}] \geq M_\ep + \frac{2}{|\ln \varepsilon|}
  \int_{\mathbb{T}_{\ell^\ep}^2}\( |\nabla h_{\varepsilon}'|^2 +\frac{
    \kappa^2}{\lep} |h_{\varepsilon}'|^2 \)dx' - \frac{1}{\pi \bar
    r_{\eps}^3 } \sum_{A_i^{\ep} \geq 3^{2/3} \pi \gamma} |\tilde
  A_i^{\ep}|^2 +o_\ep(1),
\end{align} 
where $M_\ep \ \geq 0$ is defined by
\begin{multline}\label{defM}
  M_{\varepsilon} := \sum_{i} \left( P_i^\ep - \sqrt{4 \pi A_i^{\ep}}
  \right) + c_1 \sum_{A_i^{\ep} > 3^{2/3} \pi \gamma^{-1} } A_i^{\ep}
 %\label{voldef1a} \\
 %\label{voldef2aM}
  \\ + c_2\sum_{3^{2/3} \pi \gamma \leq A_i^{\ep} \leq 3^{2/3} \pi
    \gamma^{-1} }(A_i^{\ep} - \pi \bar r_{\eps}^2)^2+
%\label{voldef3aM}
  c_3 \sum_{A_i^{\ep} < 3^{2/3} \pi \gamma} A_i^{\ep}.
\end{multline}
\end{pro}

\begin{remark}
  \label{rem-Meps}
  Defining $\beta := 3^{2/3} \pi \gamma$, by isoperimetric inequality
  applied to each connected component of $\Omega^{\eps}$ separately every
  term in the first sum in the definition of $M_\eps$ in \eqref{defM}
  is non-negative. In particular, $M_\eps$ measures the discrepancy
  between the droplets $\Omega_i^{\eps}$ with $A_i^\eps \geq \beta$
  and disks of radius $\bar r_{\eps}$.
\end{remark}

The proposition will be proved below, but before let us examine some
of its further consequences.  The result of the proposition implies
that our a priori assumption $\limsup_{\eps \to 0} F^\eps[u^{\eps}] <
+\infty$ translates into
$$
M_\ep + \frac{2}{\lep} \int_\TTT \( |\nabla h_{\varepsilon}'|^2 +
\frac{\kappa^2 }{\lep} |h_{\varepsilon}'|^2 \) dx' - \frac{1}{\pi
  \bar{r}_\ep^3} \sum_{A_i^\ep \geq \beta} |\tilde A_i^{\ep}|^2 \le C,
$$
for some $C > 0$ independent of $\eps \ll 1$, which, in view of
\eqref{newRdef1} is also
\begin{align}\label{aprioribound}
  M_\ep + \frac{2}{\lep}\( \int_\TTT \( |\nabla h_{\varepsilon}'|^2 +
  \frac{\kappa^2 }{\lep} |h_{\varepsilon}'|^2 \) dx' - \frac{1}{2
    \pi}|\ln \rho_{\eps}| \sum_{A_i^\ep \geq \beta} |\tilde
  A_i^{\ep}|^2\) \leq C.
\end{align}
A major goal of the next sections is to obtain the following estimate
\begin{equation}\label{aprouver} 
  \frac{1}{\lep}\(  \int_\TTT \( |\nabla h_{\varepsilon}'|^2 +
  \frac{\kappa^2 }{\lep} |h_{\varepsilon}'|^2 \) dx' - \frac{1}{2
    \pi}|\ln \rho_{\eps}| \sum_{A_i^\ep \geq 
    \beta} |\tilde A_i^{\ep}|^2 \) \ge -C \ln^2 (M_\eps +2), 
\end{equation} 
for some $C > 0$ independent of $\eps \ll 1$, so that the a priori
bound (\ref{aprioribound}) in fact implies that $M_{\varepsilon}$ is
uniformly bounded independently of $\varepsilon$ for small
$\eps$. This will be used crucially in Section
\ref{LboundW}. %In particular
%with this assumption, the presence of the second sum in $M_\ep$ implies the existence of $C>0$ depending on $\bar \delta$ and $\gamma$ such that
%\begin{equation}\label{dropletUbound} \limsup_{\varepsilon \to 0} A_i^{\ep} \leq C.\end{equation}

We note that $h_\ep'(x')$ satisfies the equation
\begin{align}\label{heqn1} 
  -\Delta \h + \frac{\kappa^2}{|\ln \varepsilon|} \h =
  \mu_{\varepsilon}' - \bar \mu^{\eps}\quad \text{in} \ W^{2,p}(\TTT)
\end{align}
where we define in $\TTT$
\begin{align}
  \label{rescaledmu} \mu_{\varepsilon}' (x') := \sum_{i} A_i^{\ep}
  \tilde \delta_i^\eps(x'), 
\end{align}
and
\begin{align}
  \label{fakedirac} \tilde \delta_i^\eps (x') := \frac{
    \chi_{\Omega_{i,\eps}'} (x')}{|\Omega_{i,\eps}'|},
\end{align}
which will be used in what follows.  Notice that each $\tilde
\delta_i^\eps(x')$ approximates the Dirac delta concentrated on some
point in the support of $\Omega_{i,\eps}'$ and, hence, $\mu_\eps'(x')
dx'$ approximates the measure associated with the collection of point
charges with magnitude $A_i^\eps$. In particular, the measure
$d\mu_\eps'$ evaluated over the whole torus equals the total charge:
$\mu_\eps'(\TTT) = \sum_i A_i^\eps$.

\subsection{Proof of Proposition \ref{lem1}}

{- \it Step 1:} We are first going to show that for universally small
$\eps>0$ and all $\gamma \in (0,\frac16)$ we have
\begin{equation}\label{lemma1}
\ell^2 F^{\eps}[u^{\eps}] \geq T_1 + T_2+ T_3+ T_4 + T_5 +o_\ep(1),
\end{equation}
where 
\begin{align}
  T_1 & = \sum_{i} \left( P_i^{\ep} - \sqrt{4 \pi A_i^\ep} \right),
\label{vo0}
\\
T_2 & = \frac{\gamma^{7/2}}{4\pi}\sum_{3^{2/3} \pi \gamma \leq
  A_i^{\ep} \leq 3^{2/3} \pi \gamma^{-1} } (A_i^{\ep}-\pi \bar
r_{\eps}^2)^2, \label{voldef1} \\
T_3 &=
\label{voldef2} \frac{\gamma^{-5/2}}{4\pi^2 \cdot
  3^{2/3}}\sum_{A_i^{\ep} < 3^{2/3} \pi \gamma} A_i^{\ep} (A_i^{\ep} -
\pi \bar r_{\eps}^2)^2,\\
T_4 & = \label{volterm3} \sum_{A_i^{\ep} > 3^{2/3} \pi \gamma^{-1}}
\(6^{-1} \gamma^{-1} - 1 \)A_i^{\ep}, \\
T_5 & = \label{hterm1} \frac{2}{|\ln \varepsilon|}
\int_{\mathbb{T}_{\ell}^2} \(|\nabla h_{\varepsilon}|^2 + \kappa^2
|h_{\varepsilon}|^2\)dx - \frac{1}{\pi \bar r_{\eps}^3 } \sum_{i}
|\tilde A_i^{\ep}|^2.
\end{align}
To bound $F^\ep [u^\eps]$ from below, we start from \eqref{Eexpan}.
In particular, in view of \eqref{vdef} we may rewrite \eqref{Ebar} as
\begin{align}
  \bar E^\eps[u^\eps] & = {{1 \over |\ln \eps|} \sum_i \left( P_i^\eps
      - {2 \bar \delta \over \kappa^2} A_i^\eps \right) + 2 \int_\TT
    \left( |\nabla v^\eps|^2 + \kappa^2 |v^\eps|^2 \right)
    dx} \notag \\
  \nonumber & = \frac{1}{|\ln \varepsilon|} \sum_{i}
  \left( P_i^{\ep} - \sqrt{4 \pi A_i^{\ep}} \right) \\
  &+\label{E1} \frac{1}{|\ln \varepsilon|} \sum_{i}\( \sqrt{4 \pi
    A_i^{\ep}} - \frac{2\bar \delta}{\kappa^2} A_i^{\ep} +
  \frac{1}{\pi \bar r_{\eps}^3} |\tilde
  A_i^{\ep}|^2 \)  \\
  &+\label{E2a} 2 \int_{\mathbb{T}_{\ell}^2} \left( |\nabla
    v^{\varepsilon}|^2 + \kappa^2 |v^{\varepsilon}|^2 \right) dx -
  \frac{1}{\pi \bar r_{\eps}^3 |\ln \eps|} \sum_{i} |\tilde
  A_i^{\ep}|^2.
\end{align}
We start by focusing on \eqref{E1}.  First, in the case $A_i^{\ep} >
3^{2/3} \pi \gamma^{-1}$ we have $|\tilde A_i^{\ep}|^2 = 3^{2/3} \pi
\gamma^{-1} A_i^{\ep}$ and hence, recalling that $\bar{r}_\ep =
3^{1/3}+o_\ep(1)$, where $o_\eps(1)$ depends only on $\eps$, we have
for $\ep$ universally small and $\gamma < \frac16$:
\begin{align}\nonumber 
  \frac{|\tilde A_i^{\eps}|^2}{\pi \bar r_{\eps}^3} =
  \frac{A_i^{\eps}}{\pi \bar r_{\eps}^3} \(3^{2/3} \pi \gamma^{-1} -
  3\pi \bar r_{\eps}^2 + 3 \pi \bar r_{\eps}^2\)&= A_i^{\eps}
  \(\frac{3}{\bar r_{\eps}} + \frac{3^{2/3}}{\bar r_{\eps}^3} \(
  \gamma^{-1} - 3 \left( {\bar r_{\eps} \over
      3^{1/3}} \right)^2\)\)\\
 \label{I1.2}
 &\geq A_i^{\eps}\(\frac{3}{\bar r_{\eps}} + \frac{1}{6}
 \(\gamma^{-1}- 6\)\).
\end{align}
We conclude that for $A_i^\ep > 3^{2/3} \pi \gamma^{-1}$, we have
\begin{align}\label{conclu1}
  \( \sqrt{4 \pi A_i^{\ep}} + \frac{|\tilde A_i^{\ep}|^2}{\pi \bar
    r_{\eps}^3} - \frac{2\bar \delta}{\kappa^2} A_i^{\ep}\) \ge
  \left( \frac{3}{\bar r_{\ep}} - \frac{2 \bar \delta}{\kappa^2} 
  + \frac{1}{6} \( \gamma^{-1} - 6 \) \right) A_i^{\ep}. 
\end{align}

On the other hand, when $A_i^{\ep} \leq 3^{2/3} \pi \gamma^{-1}$ we
have $\tilde A_i^{\ep} =A_i^{\ep}$ and we proceed as follows.  Let us
begin by defining, similarly to \cite{SCD}, the function
\[
f(x) = \frac{2 \sqrt{\pi}}{\sqrt{x}} + \frac{x}{\pi \bar r_{\eps}^3}
\]
for $x \in (0,+\infty)$ and observe that $f$ is convex and attains its
minimum of $\frac{3}{\bar r_{\eps}}$ at $x= \pi \bar
r_{\varepsilon}^2$, with
\begin{equation*} 
  f''(x) = \frac{3 \sqrt{\pi}}{2 x^{5/2}} >  0.
\end{equation*} 
By a second order Taylor expansion of $f$ around $ \pi \bar
r_{\varepsilon}^2$, using the fact that $f''$ is decreasing on $(0, +
\infty)$, we then have for all $x \leq x_0$
\begin{equation}\label{LBcase1a}
  \sqrt{4 \pi  x}
  + \frac{x^2}{\pi \bar r_{\eps}^3 }  = x f(x)
  \geq  x \( \frac{3}{\bar r_{\eps}} + {3 \sqrt{\pi} \over 4
    x_0^{5/2}}  \(x - \pi \bar
  r_{\varepsilon}^2  \)^2 \).
\end{equation} 
We, hence, conclude that when $3^{2/3} \pi \gamma \leq A_i^\ep \leq
3^{2/3} \pi \gamma^{-1}$, we have
\begin{align}\label{conclu2}
  \sqrt{4 \pi A_i^{\ep}} + \frac{|\tilde A_i^{\ep}|^2}{\pi \bar
    r_{\eps}^3} - \frac{2\bar \delta}{\kappa^2} A_i^{\ep} \ge
  \(\frac{3}{\bar r_{\ep}} - \frac{2 \bar \delta}{\kappa^2}\)
  A_i^{\ep} + \frac{\gamma^{5/2}}{4\pi^2 \cdot 3^{2/3}} A_i^{\ep}
  (A_i^{\ep} - \pi \bar r_{\varepsilon}^2)^2,
\end{align}
and when $A_i^\eps < 3^{2/3} \pi
\gamma$, we have
\begin{align}
  \label{conclu3}
  \sqrt{4 \pi A_i^{\ep}} + \frac{|\tilde A_i^{\ep}|^2}{\pi \bar
    r_{\eps}^3} - \frac{2\bar \delta}{\kappa^2} A_i^{\ep} \ge
  \(\frac{3}{\bar r_{\ep}} - \frac{2 \bar \delta}{\kappa^2}\)
  A_i^{\ep} + \frac{\gamma^{-5/2}}{4\pi^2 \cdot 3^{2/3}} A_i^{\ep}
  (A_i^{\ep} - \pi \bar r_{\varepsilon}^2)^2,
\end{align}

Combining \eqref{conclu1}, \eqref{conclu2} and \eqref{conclu3},
summing over all $i$, and distinguishing the different cases, we can
now bound (\ref{E1}) from below as follows:
\begin{align}
  \sum_{i} \( \sqrt{4 \pi A_i^{\ep}} - \frac{2\bar \delta}{\kappa^2}
  A_i^{\ep} + \frac{1}{\pi \bar r_{\eps}^3} |\tilde A_i^{\ep}|^2 \) &
  \geq \(\frac{3}{\bar r_{\ep}} - \frac{2 \bar \delta}{\kappa^2}\)
  \sum_i
  A_i^{\ep} \notag \\
  &+ \frac{\gamma^{7/2}}{4\pi}\sum_{ 3^{2/3} \pi \gamma \leq A_i^{\ep}
    \leq 3^{2/3} \pi \gamma^{-1} } (A_i^{\ep}-\pi \bar
  r_{\varepsilon}^2)^2 \nonumber \\ &+ \frac{\gamma^{-5/2}}{4\pi^2
    \cdot 3^{2/3}}\sum_{A_i^{\ep} < 3^{2/3} \pi \gamma} A_i^{\ep}
  (A_i^{\ep}
  - \pi \bar r_{\varepsilon}^2)^2 \nonumber \\
  &+ \label{volterm31} \sum_{A_i^{\ep} > 3^{2/3} \pi \gamma^{-1}}
  \(6^{-1} \gamma^{-1} - 1\)A_i^{\ep} .
\end{align}
%\begin{align}
% \label{E3} \frac{1}{|\ln \varepsilon|} \(\bar r_{\varepsilon}^2 - \frac{2 \bar \delta}{\kappa^2}\)\sum_i A_i^{\ep} &+ \sum_{A_i^{\ep} > \frac{12 \pi \bar \delta}{\kappa^2}}2\sqrt{\pi}|A_i^{\ep}|^{1/2} + \frac{\bar \delta }{\kappa^2} A_i^{\ep} \\
%&+ \label{E4a} \frac{1}{2} \frac{3^{-2/3}}{2\pi^2} \sum_{A_i^{\ep} < \pi \bar r_{\varepsilon}^2} A_i^{\ep} (A_i^{\ep} - \pi \bar r_{\varepsilon}^2)^2\\
%&+ \label{E5a} 3\frac{\sqrt{\pi}}{2} \frac{1}{|\ln \varepsilon|}\sum_{ \frac{12 \bar \delta \pi}{\kappa^2} \geq A_i^{\ep} > \pi \bar r_{\varepsilon}^2} |A_i^{\ep}|^{-3/2}(A_i^{\ep}-\pi \bar r_{\varepsilon}^2)^2.
%\end{align}

We now focus on the term in (\ref{E2a}). Using \eqref{hdeffirst}, we
can write the integral in (\ref{E2a}) as:
\begin{align} 
  \label{325} 
  2 & \int_\TT \left( \nabla v^\eps|^2 + \kappa^2 |v^\eps|^2 \right)
  dx \nonumber \\
  & = \frac{2}{|\ln \varepsilon|^2} \int_{\TT} \left( |\nabla
    h_{\varepsilon}|^2 + \kappa^2 h_{\varepsilon}^2 \right) dx \ +
  \frac{4 \kappa^2 \bar v^{\ep}}{|\ln \varepsilon|} \int_{\TT}
  h_{\varepsilon} dx + 2 \kappa^2 |\bar v^{\ep}|^2 \ell^2.
\end{align}
Integrating \eqref{vdef} over $\TT$ and recalling the definition of
$h_{\varepsilon}$ in \eqref{hdeffirst}, as well as \eqref{Amu}, leads to
\begin{equation}
  \label{E4n} \frac{ 4\kappa^2 \bar v^{\ep}}{|\ln \varepsilon|}
  \int_{\TT} h_{\varepsilon} dx = \frac{4 \bar
    v^{\varepsilon}}{|\ln \varepsilon|} \sum_i
  A_i^{\ep} - 4 \kappa^2 |\bar v^{\varepsilon}|^2 \ell^2.
\end{equation}
Combining \eqref{325} and \eqref{E4n}, we then find
\begin{align}
  2 \int_{\TT} \left( |\nab v^\ep|^2 + \kappa^2|v^\ep|^2 \right) dx &
  = \frac{2}{|\ln \varepsilon|^2} \int_{\TT} \left( |\nabla
    h_{\varepsilon}|^2 + \kappa^2 h_{\varepsilon}^2 \right) dx \notag
  \\
  & - \frac{1}{|\ln \varepsilon|} \(\frac{3}{\bar r_{\eps}} -
  \frac{2\bar \delta}{\kappa^2}\) \sum_i A_i^{\ep} - 2 \kappa^2 |\bar
  v^{\ep}|^2 \ell^2 . \label{327}
\end{align}
Also, by direct computation using \eqref{3.6} and \eqref{hdeffirst} we
have
\begin{align}
  \label{mk2v2l2}
  2 \kappa^2 |\bar v^\eps|^2 \ell^2 = {\bar\delta^2 \ell^2 \over 2
    \kappa^2} - \min E^0_\eps.
\end{align}
Therefore, combining this with \eqref{Eexpan}, \eqref{volterm31} and
\eqref{327}, after passing to the rescaled coordinates and performing
the cancellations we find that
\begin{align}
  \ell^2 F^\ep[u^\ep] \ge & \ T_1+ T_2+T_3+T_4 + \frac{2}{|\ln \eps|}
  \int_{\TTT} \left( |\nabla h_{\varepsilon}'(x')|^2 +
    \frac{\kappa^2}{\lep}
    |h_{\varepsilon}'(x')|^2 \right) dx'\notag \\
  & - \frac{1}{\pi \bar r_{\eps}^3} \sum_{i} |\tilde A_i^{\ep}|^2 +
  o_\eps(1), \label{328}
\end{align}
which is nothing but \eqref{lemma1}.

% On the other hand, using \eqref{RZorder} and the fact that
% $\bar\mu^\ep $ minimizes $E^0_{\eps}$, we have \begin{align}\nonumber
%   &\lep \(\bar E^{\eps}(u^{\eps}) + \frac{\ell^2 \bar
%     \delta^2}{2\kappa^2} - \(\(\frac{3}{\bar r_{\eps}} - \frac{2 \bar
%     \delta}{\kappa^2}\)\int d\bar \mu^{\eps} + 2\kappa^2\int_{\TT}
%   |\bar v^{\eps}|^2 dx + \frac{\ell^2 \bar \delta^2}{2\kappa^2}
%   \)\)\\ \label{rrF} &= \lep \(\bar E^{\eps}(u^{\eps}) + \frac{ \ell^2
%     \bar \delta^2}{2\kappa^2} - E_{\eps}^0(\bar \mu^{\eps})\)=
%   \lep\(\bar E^{\eps}(u^{\eps}) + \frac{ \ell^2 \bar
%     \delta^2}{2\kappa^2} - \min E_{\eps}^0\).
% \end{align}
% In view of \eqref{Eexpan}, \eqref{lemma1} follows from \eqref{328} and
% \eqref{rrF}.
% \\

\noindent
{- \it Step 2: }
We proceed to absorbing the contributions of the small droplets in
\eqref{hterm1} by \eqref{voldef1}.  To that effect, we observe that,
for the function
\begin{equation}
  \Phi_\ep(x): = \frac{\gamma^{-5/2}}{4\pi^2 \cdot 3^{2/3}} x (x-\pi
  \bar r_{\eps}^2)^2  - \frac{1}{\bar r_{\eps}^3 } x^2 \geq
  \frac{\gamma^{-5/2} x}{4\pi^2 \cdot 3^{2/3}}  
  \left\{ \pi^2 \bar r_\eps^4 - \left( 2 \pi \bar r_\eps^2 +
      {\gamma^{5/2} \over \bar r_\eps^3} \right) x \right\}, 
\end{equation}
there exists a universal $\gamma \in (0,\frac16)$ such that $\Phi_\ep
(x) \geq x$ whenever $0 \leq x < 3^{2/3}\pi \gamma$ and $\eps$ is
universally small. Using this observation, we may absorb all the terms
with $A_i^{\ep} < 3^{2/3}\pi \gamma$ appearing in the second term in
(\ref{hterm1}) into \eqref{voldef2} by suitably reducing the
coefficient in front of the latter. This proves the result. \qed

\section{Ball construction}\label{sec4}

The goal of this section is to show \eqref{aprouver} using the
abstract framework of Theorem 3 in \cite{SS2}.  The difficulty in
doing this, as in the case of the Ginzburg-Landau model treated in
\cite{SS2}, is that the energy density $e_\ep' - \frac{1}{\pi} |\ln
\rho_{\eps}| \sum_{A_i^{\ep} \geq \beta} |\tilde A_i^{\ep}|^2
\tilde{\delta_i^\eps}$ is not positive (or bounded below independently
of $(u^\eps)$).  The next two subsections are meant to go around this
difficulty by showing that this energy density can be modified, by
displacing a part of the energy from the regions where the energy
density is positive into regions where the energy density is negative
in order to bound the modified energy density from below while making
only a small enough error. This is achieved by obtaining sharp lower
bounds on the energy of the droplets. Since their volumes and shapes
are a priori unknown, the terms in $M_\ep$ are used to control in a
quantitative way the deviations from the droplets being balls of fixed
volume.

In this section we perform a ball construction which follows the
procedure of \cite{SS2}.  The goal is to cover the droplets
$\{\Omega_{i,\eps}'\}$ whose volumes are bounded from below by a given
$\beta > 0$ with a finite collection of disjoint closed balls whose
radii are smaller than 1, on which we have a good lower bound for the
energy in the left-hand side of \eqref{aprouver}. This is possible for
sufficiently small $\eps$ in view of the fact that $\ell^\eps \to
\infty$ and that the leading order asymptotic behavior of the energy
from \eqref{minE} yields control on the perimeter and, therefore, the
essential diameter of each of $\Omega_{i,\eps}'$. The precise
statements are given below.  We will also need the following basic
result, which holds for sufficiently small $\eps$ ensuring that the
droplets are smaller than the sidelength of the torus (see the
discussion at the beginning of Sec. 2 in \cite{SCD}).
\begin{lem}
  \label{l-diamP}
  There exists $\eps_0 > 0$ depending only on $\ell$, $\kappa$,
  $\bar\delta$ and $\sup_{\eps > 0} F^\eps[u^\eps]$ such that for all
  $\eps \leq \eps_0$ we have
  \begin{align}
    \label{diamP}
    \mathrm{ess \, diam} (\Omega_{i,\eps}') \leq c | \partial
    \Omega_{i,\eps}' |,
  \end{align}
  for some universal $c > 0$.
\end{lem}

>From now on and for the rest of the paper we fix $\gamma$ to be the
constant given in Proposition \ref{lem1} and, as in the previous
section, we define $\beta = 3^{2/3} \pi \gamma$.  We also introduce
the following notation which will be used repeatedly below. To index
the droplets, we will use the following definitions:
\begin{align}
  I_{\beta} := \{ i \in \mathbb N : A_i^{\ep} \ \geq \ \beta\}, \quad
  I_{E} := \{i \in \mathbb N : |\Omega_{i,\eps}' \cap (\TTT \backslash
  E)| = 0 \}, \quad I_{\beta,E} := I_\beta \cap I_E,
\end{align}
where $E \subset \TTT$.  For a collection of balls $\mathcal{B}$, the
number $r(\mathcal{B})$ (also called the total radius of the
collection) denotes the sum of the radii of the balls in
$\mathcal{B}$. For simplicity, we will say that a ball $B$ {\em
  covers} $\Omega_{i,\eps}'$, if $i \in I_B$.

The principle of the ball construction introduced by Jerrard
\cite{jerrardball} and Sandier \cite{sandierball} and adapted to the
present situation is to start from an initial set, here $\bigcup_{i
  \in I_{\beta,U}} \Omega_{i,\eps}'$ for a given $U \subseteq \TTT$
and cover it by a union of finitely many closed balls with
sufficiently small radii. This collection can then be transformed into
a collection of {\em disjoint} closed balls by the procedure, whereby
every pair of intersecting balls is replaced by a larger ball whose
radius equals the sum of the radii of the smaller balls and which
contains the smaller balls. This process is repeated until all the
balls are disjoint. The obtained collection will be denoted $\B_0$,
its total radius is $r(\B_0)$. Then each ball is dilated by the same
factor with respect to its corresponding center. As the dilation
factor increases, some balls may touch. If that happens, the above
procedure of ball merging is applied again to obtain a new collection
of disjoint balls of the same total radius.  The construction can be
stopped when any desired total radius $r$ is reached, provided that
$r$ is universally small compared to $\ell^\eps$. This yields a
collection $\B_r$ covering the initial set and containing a
logarithmic energy \cite{jerrardball,sandierball}.

We now give the statement of our result concerning the ball
construction and the associated lower bounds. Throughout the rest of
the paper we use the notation $f^+ := \max(f, 0)$ and $f^- := -\min(f,
0)$.

\begin{pro}\label{boulesren}  
  Let $U \subseteq \TTT$ be an open set such that $I_{\beta,U} \not=
  \varnothing$, and assume that \eqref{ZO.1} holds.
  \begin{enumerate}
  \item[-] There exists $\eps_0 > 0$, $r_0 \in (0, 1)$ and $C > 0$
    depending only on $\ell$, $\kappa$, $\bar\delta$ and $\sup_{\eps >
      0} F^\eps[u^\eps]$ such that for all $\eps < \eps_0$ there
    exists a collection of finitely many disjoint closed balls $\B_0$
    whose union covers $\bigcup_{i \in I_{\beta,U}} \Omega_{i,\eps}'$
    and such that
    \begin{align}
      \label{rB00}
      r(\B_0)\le c \eps^{1/3} |\ln \eps|^{1/6} \sum_{i \in
        I_{\beta,U}} P_i^\eps < r_0,
    \end{align}
    for some universal $c > 0$. Furthermore, for every $r\in
    [r(\B_0),r_0]$ there is a family of disjoint closed balls $\B_r$
    of total radius $r$ covering $\B_0$.

  \item[-] For every $B \in \B_r$ such that $B \subset U$ we have
    $$ \int_{B}  \left( |\nabla \h|^2   \, dx' + \frac{\kappa^2}{4 |\ln
        \varepsilon|} |\h|^2 \right) dx' \ge \frac{1}{2\pi} \( \ln
    \frac{r}{r(\B_0)}- cr \)^+ \sum_{i \in I_{\beta,B}} |\tilde
    A_i^{\ep}|^2,$$ for some $c>0$ depending only on $\kappa$ and
    $\bar\delta$.

  \item[-] If $B\in\B_r$, for any non-negative Lipschitz function
    $\chi$ with support in $U$, we have
    \begin{multline*} 
      \int_{B} \chi \left( |\nabla \h|^2 \,dx' + \frac{\kappa^2}{4
          |\ln \varepsilon|} |\h|^2\right) dx' - \frac{1}{2 \pi
      }\(\ln\frac r{r(\B_0)} - cr\)^+
      \sum_{i\in I_{\beta,B}} \chi_i  |\tilde A_i^{\ep}|^2 \nonumber \\
      \ge - C \|\nab\chi\|_\infty \sum_{i \in I_{\beta,B}} |\tilde
      A_i^{\ep}|^2,
    \end{multline*} 
    where $\chi_i := \int_U \chi \tilde \delta_i^{\eps} dx' $, with
    $\tilde \delta_i^\eps(x')$ defined in \eqref{fakedirac}, for some
    $c > 0$ depending only on $\kappa$ and $\bar\delta$, and a
    universal $C > 0$.
\end{enumerate}
\end{pro}
\begin{remark}
The explanation for the factor of $\frac{1}{4}$ in front of $\frac{\kappa^2}{|\ln \varepsilon|} |\h|^2$ is that we must
`save' a fraction of this term for the mass displacement argument in Section \ref{sec5} and in the convergence 
result in Section \ref{sec6}.
\end{remark}
\begin{proof}[Proof of the first item]
  Choose an arbitrary $r_0 \in (0, 1)$.  As in \cite{SCD}, from the
  basic lower bound on $\bar E^\eps$ (see \cite[Equations (2.12) and
  (2.15)]{SCD}):
  \begin{align}
    \label{Efour}
    \bar E^\eps[u^\eps] \geq \frac{1}{|\ln \eps|} \sum_i P_i^\eps - {2
      \bar \delta \over \kappa^2 |\ln \eps|} \sum_i A_i^\eps + {2
      \over \kappa^2 \ell^2 |\ln \eps|^2} \left( \sum_i A_i^\eps
    \right)^2,
  \end{align}
  where $A_i^\eps$ and $P_i^\eps$ are defined in \eqref{AP}, we obtain
  with the help of \eqref{ZO.1} that
  \begin{align}
    \limsup_{\eps \to 0}\frac{1}{|\ln \eps|} \sum_i A_i^\eps \leq C,
    \qquad \limsup_{\eps \to 0}\frac{1}{|\ln \eps|} \sum_i P_i^\eps
    \leq C, \label{Pibound}
  \end{align}
  for some $C > 0$ depending only on $\ell$, $\kappa$, $\bar\delta$
  and $\sup_{\eps > 0} F^\eps[u^\eps]$. 

  As is well-known, the essential diameter of a connected component of
  a set of finite perimeter on a torus can be bounded by its
  perimeter, provided that the latter is universally small compared to
  the size of the torus (see, e.g., \cite{ambrosio2}). Therefore, in
  view of the the definition of $P_i^\eps$ in \eqref{AP} and the
  second of \eqref{Pibound}, for sufficiently small $\eps$ it is
  possible to cover each $\Omega_{i,\eps}'$ with $i \in I_{\beta,U}$
  by a closed ball $B_i$, so that the collection $\tilde {\mathcal
    B}_0$ consisting of all $B_i$'s (possibly intersecting) has total
  radius
  \begin{align}
    \label{rB0}
    r_0(\tilde{\B}_0) \leq C \eps^{1/3} |\ln \eps|^{1/6} \sum_{i \in
      I_{\beta,U}} P_i^\eps,
  \end{align}
  for some universal $C > 0$. Furthermore, by the first inequality in
  \eqref{Pibound} and the fact that $A_i^\eps \geq \beta$ for all $i \in
  I_{\beta,U}$ the collection $\tilde{\B}_0$ consists of only finitely
  many balls. Therefore, we can apply the construction \`a la Jerrard
  and Sandier outlined at the beginning of this section to obtain the
  desired family of balls $\B_0$ and $\B_r$, with $r(\B_0) = r({\tilde
    \B}_0)$. The estimate on the radii follows by combining the second
  of \eqref{Pibound} and \eqref{rB0} and the fact that $\ell^\eps \to
  \infty$ with the rate depending only on $\ell$, for sufficiently
  small $\eps$ depending on $\ell$, $\kappa$, $\bar\delta$,
  $\sup_{\eps > 0} F^\eps[u^\eps]$ and $r_0$.
\end{proof}

\begin{proof}[Proof of the second item]

  Let $B \subset U$ be a ball in the collection
  $\mathcal{B}_r$. Denote the radius of $B$ by $r_B$ and set
  \[ 
  X_{\varepsilon} := \frac{\kappa^2}{|\ln \varepsilon|} \int_B \h dx'.
  \] 
  Integrating \eqref{heqn1} over $B$ and applying the divergence
  theorem, we have 
  \begin{align}
    \label{greens1}
    \int_{\partial B} \frac{\partial \h}{\partial \nu} \ d \mathcal
    H^1(x') = {m}_{B,\varepsilon} - X_{\varepsilon},
  \end{align}
  where
  $$
  {m}_{B,\varepsilon} := \int_B (\mu'_\eps(x') - \bar
  \mu^{\eps}) dx' = \sum_{i \in I_{B}} A_i^{\ep} + \sum_{i \not\in
    I_B} \theta_i A_i^\eps - \bar \mu^{\eps}|B|,
  $$ 
  for some $\theta_i \in [0, 1)$ representing the volume fraction in
  $B$ of those droplets that are not covered completely by $B$, and
  $\nu$ is the inward normal to $\partial B$.  Using the
  Cauchy-Schwarz inequality, we then deduce from \eqref{greens1} that
  \begin{equation}\label{lecercle}
    \int_{\partial B}|\nabla \h|^2 \, d
    \mathcal H^1(x') \geq \frac{1}{2\pi r_B} ({m}_{B,
      \varepsilon} - X_{\varepsilon})^2 \geq \frac{{m}_{ B, \ep}^2 -
      2 {m}_{B, \ep} X_\ep}{2\pi r_B} .
  \end{equation} 
  By another application of the Cauchy-Schwarz inequality, we may
  write
  \begin{equation}\label{lecercle2}
    \frac{\kappa^2}{4 |\ln \varepsilon|} \int_B |\h|^2 \, dx' \geq
    \frac{X_{\varepsilon}^2}{4 \pi r_B^2}\frac{|\ln
      \varepsilon|}{\kappa^2}. 
  \end{equation}
 
  We now add \eqref{lecercle} and \eqref{lecercle2} and optimize the
  right-hand side over $X_{\eps}$. We obtain
  \begin{equation}\label{lecerclelast}
    \int_{\partial B}|\nabla \h|^2 \, d
    \mathcal H^1(x')  + \frac{\kappa^2}{4|\ln \varepsilon|} \int_B
    |\h|^2 \, dx' \geq \frac{{m}_{B,\eps}^2}{2\pi r_B}\(1 -
    \frac{Cr_B}{|\ln \eps|}\), 
  \end{equation}
  for $C =  \kappa^4$. Recalling that $r_B \leq r \leq r_0 <
  1$, we can choose $\eps$ sufficiently small depending only on
  $\kappa$ so that the term in parentheses above is positive.

  Inserting the definition of ${m}_{B, \ep}$ into \eqref{lecerclelast}
  and discarding some positive terms yields
  \begin{align}
    \int_{\partial B}|\nabla \h|^2 & d \mathcal H^1(x') +
    \frac{\kappa^2}{4|\ln \varepsilon|} \int_B |\h|^2 \, dx' \notag \\
    & \geq \frac{1}{2\pi r_B} \Big( \sum_{i \in I_B} A_i^{\ep} +
    \sum_{i \not\in I_B} \theta_i A_i^\eps
    - \bar \mu^{\eps}|B|\Big)^2\(1 - \frac{Cr_B}{|\ln \eps|}\)\notag \\
    & \geq \frac{1}{2\pi r_B} \Big( \sum_{i \in I_{B}} A_i^{\ep} +
    \sum_{i \not\in I_B} \theta_i A_i^\eps \Big)^2\(1 - 2\bar
    \mu^{\eps}|B|\Big(\sum_{i \in I_B} A_i^{\ep}\Big)^{-1}- \frac{C r_B}{|\ln
      \eps|}\).\label{lecercle4}
  \end{align} 
  We now use the fact that by construction $B$ covers at least one
  $\Omega_{i,\eps}'$ with $A_i^{\ep} \geq \beta$. This leads us to
  \begin{align}
    \nonumber \int_{\partial B}|\nabla \h|^2 & d \mathcal H^1(x') +
    \frac{\kappa^2}{4|\ln \varepsilon|} \int_B |\h|^2
    \, dx' \notag \\
    & \geq \frac{1}{2\pi r_B} \Big(\sum_{i \in I_{B}} A_i^{\ep} +
    \sum_{i \not\in I_B} \theta_i A_i^\eps \Big)^2 \(1 - {2 \pi \bar
      \mu^{\eps} r_B^2 \over \beta} - {C r_B \over |\ln
      \eps|}  \) \notag \\
    &\geq \frac{1}{2\pi r_B} \sum_{i \in I_{\beta,B}} |\tilde
    A_i^{\ep}|^2 \(1 - c r_B\), \label{lecercle4filter}
  \end{align}
  for some $c > 0$ depending only on $\kappa$ and $\bar\delta$, where
  in the last line we used that $A_i^{\ep}\ge \tilde A_i^{\ep}$. Hence
  there exists $r_0 \in (0,1)$ depending only on $\kappa$, and
  $\bar\delta$ such that the right-hand side of
  \eqref{lecercle4filter} is positive.

  Finally, let us define $\F(x,r) := \int_{B(x,r)} |\nabla \h|^2 \,
  dx' + {r \kappa^2 \over 4 |\ln \eps|} \int_{B(x,r)} |\h|^2 \, dx'$,
  where $B(x,r)$ is the ball centered at $x$ of radius $r$. The
  relation \eqref{lecercle4filter} then reads for $B(x,r) = B \in
  \B_r$ and a.e. $r \in (r(\B_0), r_0]$:
  \begin{align} 
    \frac{\partial \F}{\partial r} &\geq \frac{1}{2\pi r} \sum_{i \in
      I_{\beta,B}} |\tilde A_i^{\ep}|^2 (1-c r),
  \end{align}
  with $c$ as before.  Then using \cite[Proposition 4.1]{SS3}, for
  every $B\in\B(s) := \B_r$ with $r = e^{s} r(\mathcal{B}_0)$ (using
  the notation of \cite[Theorem 4.2]{SS3}) we have
  \begin{align}
    \nonumber \int_{B\sm\B_0}|\nabla \h|^2 \, dx' + \frac{r_B
      \kappa^2}{4|\ln \eps|} \int_{B} |\h|^2 \, dx' &\ge \int_0^s
    \sum_{\substack{B'\in \B(t)\\B'\subset B}} \frac{ 1}{2 \pi }
    \sum_{i \in I_{\beta,B'}}
    |\tilde A_i^{\ep}|^2\(1- c  r(\B(t))\)  dt\\
    &\nonumber =\int_0^s \sum_{\substack{B'\in \B(t)\\B'\subset B}}
    \frac{ 1}{2 \pi } \sum_{i \in I_{\beta,B'}}
    |\tilde A_i^{\ep}|^2\(1- c  e^t r(\mathcal{B}_0)\)  dt\\
    &\geq \frac{1}{2\pi}\sum_{i \in I_{\beta,B}} |\tilde A_i^{\ep}|^2
    \(\ln \frac{r}{r(\B_0)}- c r\), \label{BB0}
\end{align}
where we observed that the double summation appearing in the first and
second lines is simply the summation over $I_{\beta,B}$. Once again,
in view of the fact that $r_B \leq 1$ and that both terms in the
integrand of the left-hand side of \eqref{BB0} are non-negative, this
completes the proof of the second item.
\end{proof}

\begin{proof}[Proof of the third item] This follows \cite{SS2}. Let
  $\chi$ be a non-negative Lipschitz function with support in $U$.  By
  the ``layer-cake'' theorem \cite{lieb-loss}, for any $B \in \B_r$ we
  have
  \begin{align}\label{machin}
    \int_{B} \chi \left( |\nab \h|^2 + \frac{\kappa^2}{4 |\ln \eps|}
      |\h|^2 \right) dx' &= \int_0^{+\infty} \int_{E_t \cap B} \(
    |\nabla h_{\varepsilon}'|^2 + \frac{\kappa^2 }{4 \lep}
    |h_{\varepsilon}'|^2 \) dx' \,dt,
  \end{align}
  where $E_t:=\{\chi >t\}$.  If $ i \in I_{\beta, B}$, then by
  construction for any $s\in[r(\B_0),r]$ there exists a unique closed
  ball $B_{i,s}\in\B_s$ containing $\Omega_{i,\eps}'$.  Therefore, for
  $t>0$ we can define
  $$ 
  s(i,t) := \sup \left\{ s \in [r(\B_0), r] \, : \, B_{i,s}\subset E_t
  \right\},
  $$ 
  with the convention that $s(i,t) = r(\B_0)$ if the set is empty. We
  also let $B_i^t := B_{i,s(i,t)}$ whenever $s(i,t) > r(\B_0)$. Note
  that for each $i \in I_{\beta,B}$ we have that $t \mapsto s(i,t)$ is
  a non-increasing function. In particular, we can define $t_i \geq 0$
  to be the supremum of the set of $t$'s at which $s(i,t) = r$ (or
  zero, if this set is empty).

  If $t > t_i$ and $s(i,t) > r(\B_0)$, then for any $x \in
  \Omega_{i,\eps}'$ and any $y \in B_i^t \sm E_t$ (which is not empty)
  we have
  \begin{equation}\label{boundr} 
    \chi(x) - t\le \chi(x) - \chi(y)\le 2
    s(i,t)\|\nab\chi\|_\infty.
  \end{equation}
  Averaging over all $x \in \Omega_{i,\eps}'$, we hence deduce
 \begin{equation}\label{boundr1} 
   \chi_i -  t \leq  2 s(i,t) \|\nabla
   \chi\|_{\infty}.
 \end{equation}
 % \cm{the following sentence can be deleted}
 % In addition, in view of the definition of $t_i$ we also have
 % \begin{equation}\label{boundr2}
 %   \chi_i - t_i \leq  2 r \|\nabla
 %   \chi\|_{\infty}.
 % \end{equation}
 Now, for any $t\ge 0$ the collection $\{B_i^t\}_{i \in
   I_{\beta,B,t}}$, where $I_{\beta,B,t} := \{i \in I_{\beta,B} :
 s(i,t) > r(\B_0) \}$ is disjoint. Indeed if $i,j\in I_{\beta,B,t}$
 and $s(i,t)\ge s(j,t)$ then, since $\B_{s(i,t)}$ is disjoint, the
 balls $B_{i,s(i,t)}$ and $B_{j,s(i,t)}$ are either equal or
 disjoint. If they are disjoint we note that $s(i,t)\ge s(j,t)$
 implies that $B_{j, s(j,t)} \ \subseteq \ B_{j,s(i,t)}$, and,
 therefore, $B_j^t = B_{j,s(j,t)}$ and $B_i^t = B_{i,s(i,t)}$ are
 disjoint. If they are equal and $s(i,t) > s(j,t)$, then
 $B_{j,s(j,t)}\subset E_t$, contradicting the definition of
 $s(j,t)$. So $s(j,t) = s(i,t)$ and then $B_j^t = B_i^t$.

 Now assume that $B'\in \{B_i^t\}_{i \in I_{\beta,B,t}}$ and let $s$
 be the common value of $s(i,t)$ for $i$'s in $I_{\beta, B'}$. Then,
 the previous item of the proposition yields 
 \begin{align*}
   \int_{B'} \left( |\nabla \h|^2 + {\kappa^2 \over 4 |\ln \eps| }
     |\h|^2 \right) dx' \nonumber & \ge \frac{1}{2\pi} \( \ln\frac
   {s}{r(\B_0)} - c s\)^+ \sum_{i\in I_{\beta,B',t}} |\tilde
   A_i^{\ep}|^2.
% \\ & \geq \frac{1}{2\pi} \( \ln\frac {s}{r(\B_0)} - c
%    s\)^+ \sum_{i\in I_{\beta,B'} } |\tilde A_i^{\ep}|^2,
 \end{align*} 
 Summing over $B'\in\{B_i^t\}_{i \in I_{\beta,B,t}}$, we deduce
 \begin{align}
   \int_{B \cap E_t } \left( |\nab \h|^2 + {\kappa^2 \over 4 |\ln
       \eps|} |\h|^2 \right) dx' & \ge \frac{1}{2\pi} \sum_{i\in
     I_{\beta,B,t}} |\tilde A_i^{\ep}|^2 \( \ln\frac
   {s(i,t)}{r(\B_0)} - c s(i,t)\)^+\nonumber \\
   & = \frac{1}{2\pi} \sum_{i\in I_{\beta,B}} |\tilde A_i^{\ep}|^2 \(
   \ln\frac {s(i,t)}{r(\B_0)} - c s(i,t)\)^+,
 \end{align} 
 where in the last inequality we took into consideration that all the
 terms corresponding to $i \in I_{\beta,B} \sm I_{\beta,B,t}$ give no
 contribution to the sum in the right-hand side.  Integrating the
 above expression over $t$ and using the fact that $r_0(\B_0) \le
 s(i,t)\le r$ yields
\begin{multline}\label{lastguy}
  \int_{0}^{+\infty}\int_{E_t \cap B} \left( |\nab \h|^2+
    \frac{\kappa^2}{ 4 |\ln \eps|} |\h|^2 \right) dx'\, dt \ge
  \frac{1}{2\pi} \sum_{i \in I_{\beta,B}} |\tilde A_i^{\ep}|^2
  \int_0^{\chi_i}
  \( \ln\frac {s(i,t)}{r(\B_0)} - cr\)^+ dt \\
  \geq \frac{1}{2\pi} \sum_{i \in I_{\beta,B}} \chi_i |\tilde
  A_i^{\ep}|^2 \left( \ln { r \over r(\B_0)} - c r \right)^+ +
  \frac{1}{2\pi} \sum_{i \in I_{\beta,B}} |\tilde A_i^{\ep}|^2
  \int_0^{\chi_i} \ln\frac {s(i,t)}{r} dt.
\end{multline}

We now concentrate on the last term in \eqref{lastguy}. Using the
estimate in \eqref{boundr1} and the definition of $t_i$, we can bound
the integral in this term as follows
\begin{align}
  \label{intremain}
  \int_0^{\chi_i} \ln\frac {s(i,t)}{r} dt \geq \int_{t_i}^{\chi_i} \ln
  \left( \frac {\chi_i - t}{2 r \| \nabla \chi \|_\infty} \right) dt
  \geq -C \| \nabla \chi \|_\infty,
\end{align}
for some universal $C > 0$, which is obtained by an explicit
computation and the fact that $r \leq r_0 < 1$. Finally, combining
\eqref{intremain} with \eqref{lastguy}, the statement follows from
\eqref{machin}.  
\end{proof}

\begin{remark}
  \label{moinsb0} Inspecting the proof, we note that the statements of
  the proposition are still true with the left-hand sides replaced by
  $ \int_{B\backslash \B_0}\chi |\nab h_\ep'|^2 \, dx'+
  \frac{\kappa^2}{4 \lep}\int_B \chi |h_\ep'|^2 dx'$ (with $\chi
  \equiv 1$ or $\chi$ Lipschitz, respectively).
\end{remark}

\section{Energy displacement} \label{sec5}

In this section, we follow the idea of \cite{SS2} of localizing the
ball construction and combine it with a ``energy displacement'' which
allows to reduce to the situation where the energy density in
\eqref{aprouver} is bounded below.  For the proposition below we
define for all $x' \in \TTT$:
\begin{equation}\label{defnue}
  \nu^{\varepsilon}(x') := \sum_{i \in I_\beta} |\tilde
  A_i^{\ep}|^2 \tilde \delta_i^\eps (x'), 
\end{equation}
where $\tilde \delta_i^\eps(x')$ is given by (\ref{fakedirac}). We
also recall that $\rho_{\eps}$ defined in \eqref{newRdef1} is the
expected radius of droplets in a minimizing configuration in the blown
up coordinates.

We cover $\mathbb{T}_{\ell^{\eps}}^2$ by the balls of radius $\frac14
r_0$ whose centers are in $\frac{r_0}{8} \mz^2$.  We call this cover
$\{U_\alpha\}_\alpha$ and $ \{x_\alpha\}_\alpha$ the centers. We also
introduce $D_\alpha := B(x_\alpha, \frac{3r_0}{4})
$. %nd $n_\alpha = \# \{ i , A_i^\ep \ge \beta\}$.

\begin{pro}\label{wspread} 
  Let $\h$ satisfy \eqref{heqn1}, assume \eqref{ZO.1} holds, and set
\begin{equation}\label{deff}
  f_{\varepsilon}: = |\nab \h|^2 + \frac{\kappa^2}{ 2 \lep} |\h|^2  -
  \frac{1}{2 \pi}|\ln \rho_{\eps}| \, \nu^{\varepsilon}.
% \sum_{i \in I_\beta} |\tilde A_i^{\ep}|^2 \tilde \delta_i^\eps.
\end{equation}
Then there exist $\eps_0 > 0$ as in Proposition \ref{boulesren} and
constants $c,C > 0$ depending only on $\bar{\delta}$ and $\kappa$ such
that for all $\eps < \eps_0$, there exists a family of integers
$\{n_\alpha\}_{\alpha}$ and a density $g_{\varepsilon}$ on
$\mathbb{T}_{\ell^{\eps}}^2$ with the following properties.
\begin{itemize}
\item[-] $g_\eps$ is bounded below:
$$g_{\varepsilon} \ge -c
\ln^2 (M_{\varepsilon} + 2) \quad \text{on} \
\mathbb{T}_{\ell^{\eps}}^2.$$
\item[-] For any $\alpha$, 
$$n_{\alpha}^2 \le C \(  g_\ep(D_\alpha) + c \ln^2 (M_\ep+2)\).
$$
\item[-] For any Lipschitz function $\chi$ on
  $\mathbb{T}_{\ell^{\eps}}^2$ we have
\begin{align}\label{wg}
  \left|\int_\TTT \chi (f_{\varepsilon}-g_{\varepsilon}) dx'\right|
  \le C \sum_{\alpha}\(\nu^\ep(U_\alpha)+ (n_\alpha +M_\ep) \ln
  (n_\alpha + M_\ep +2) \)\|\nab \chi\|_{L^\infty(D_\alpha)}.
  \end{align}

%\item[-] For any  $E\subset U$,
%\begin{equation}\label{bgnalpha}\#(\Lambda\cap E) \le
%C\(1+\|m\|_\infty^2|\widehat E| + g(\widehat E)\),\end{equation} where $C$ is
%universal.

\end{itemize}
\end{pro}

\begin{proof} The proof follows the method of \cite{compagnon},
  involving a localization of the ball construction followed by energy
  displacement.  Here we follow \cite[Proposition 4.9]{SS2}.  One key
  difference is the restriction to $I_\beta$ which means we cover only
  those $\Omega_{i,\eps}'$ satisfying $A_i^{\ep} \geq \beta$ as in
  Proposition (\ref{boulesren}).

\bigskip

\noindent {\it - Step 1: Localization of the ball construction.}\\
%We cover $\mathbb{T}_{\ell^{\eps}}^2$ by the balls of radius $\frac14
%r_0$ whose centers are in $\frac{r_0}{8} \mz^2$.  We call this cover
%$\{U_\alpha\}_\alpha$ and $ \{x_\alpha\}_\alpha$ the centers.  
We use $U_\alpha$ defined above as the cover on $\TTT$. For each
$U_\alpha$ covering at least one droplet whose volume is greater or
equal than $\beta$ and for any $r\in(r(\B_0), \frac14 r_0)$ we
construct disjoint balls $\B_r^\alpha$ covering all $\Omega_{i,\eps}'$
with $i \in I_{\beta,U_\alpha}$, using
Proposition~\ref{boulesren}. Then choosing a small enough
$\rho\in(r(\B_0), \frac14 r_0)$ independent of $\eps$ (to be specified
below), we may extract from $\cup_\alpha \B_\ro^\alpha$ a disjoint
family which covers $\cup_{i \in I_\beta} \Omega_{i,\eps}'$ as
follows: Denoting by $\mathcal{C}$ a connected component of
$\cup_\alpha \B_\rho^\alpha$, we claim that there exists $\alpha_0$
such that $\mathcal{C}\subset U_{\alpha_0}$. Indeed if $x\in
\mathcal{C}$ and letting $\lambda$ be a Lebesgue number\footnote{A
  Lebesgue number of a covering of a compact set is a number
  $\lambda>0$ such that every subset of diameter less than $\lambda$
  is contained in some element of the covering.}  of the covering of
$\TTT$ by $\{U_\alpha\}_\alpha$ (it is easy to see that in our case
$\frac14 r_0 < \lambda < \frac12 r_0$), there exists $\alpha_0$ such
that $B(x,\lambda)\subset U_{\alpha_0}$. If $\mathcal{C}$ intersected
the complement of $U_{\alpha_0}$, there would exist a chain of balls
connecting $x$ to $(U_{\alpha_0})^c$, each of which would intersect
$U_{\alpha_0}$. Each of the balls in the chain would belong to some $
\B_\rho^{\alpha'}$ with $\alpha'$ such that $\dist(U_{\alpha'},
U_{\alpha_0})\le 2\rho< \frac12 r_0$. Thus, calling $k$ the universal
maximum number of $\alpha'$'s such that $\dist(U_{\alpha'},
U_{\alpha_0})< \frac12 r_0$, the length of the chain is at most $2 k
\rho$ and thus $\lambda \le 2 k \rho$. If we choose $\rho <
\lambda/(2k)$, this is impossible and the claim is proved. Let us then
choose $\rho = \lambda/(4k)$. By the above, each $\mathcal{C}$ is
included in some $U_\alpha$.

We next obtain a disjoint cover of $\cup_{ i \in I_\beta}
\Omega_{i,\eps}'$ from $\cup_\alpha \B_\rho^\alpha$. Let $\mathcal{C}$
be a connected component of $\cup_\alpha \B_\rho^\alpha$. By the
discussion of the preceding paragraph, there exists an index
$\alpha_0$ such that $\mathcal{C}\subset U_{\alpha_0}$.  We then
remove from $\mathcal{C}$ all the balls which do not belong to
$\B_\rho^{\alpha_0}$ and still denote by $\B^{\alpha_0}_\rho$ the
obtained collection. We repeat this process for all the connected
components and obtain a disjoint cover $\B_\rho =
\cup_\alpha\B^\alpha_\rho$ of $\cup_{ i \in I_\beta}
\Omega_{i,\eps}'$. Note that this procedure uniquely associates an
$\alpha$ to a given $B \in \B_\rho$, as well as to each
$\Omega_{i,\eps}'$ for a given $i \in I_\beta$ by assigning to it the
ball in $\B_\rho$ that covers it, and then the $\alpha$ of this
ball. We will use this repeatedly below. We also slightly abuse the
notation by sometimes using $\B_\rho^\alpha$ to denote the union of
the balls in the family $\B_\rho^\alpha$.

We now proceed to the energy displacement.
\medskip

\noindent {\it - Step 2: Energy displacement in the balls.}\\
\noindent Note that by construction every ball in $\B_\rho^\alpha$ is
included in $U_\alpha$. From the last item of
Proposition~\ref{boulesren} applied to a ball $B\in\B_\rho^\alpha$, if
$\varepsilon$ is small enough then, for any Lipschitz non-negative
$\chi$ we have for some $c > 0$ depending only on $\kappa$ and
$\bar\delta$ and a universal $C > 0$
\begin{multline*} \int_{B} \chi \(|\nabla \h|^2+\frac{\kappa^2}{4 |\ln
    \varepsilon|} |\h|^2\) dx' - \frac{1}{2\pi}\(\ln\frac
  \rho{r(\B_0^{\alpha})} - c \rho\)^+ \sum_{i\in I_{\beta,B}} \chi_i
  |\tilde A_i^{\ep}|^2\\ \ge - C \nu^{\varepsilon}(B)
  \|\nab\chi\|_{L^\infty(B)},\end{multline*} where $\nu^{\eps}$ is
defined by \eqref{defnue}. Rewriting the above, recalling the
definition \eqref{newRdef1} and defining $n_{\alpha} \geq 1$ to be the
number of droplets included in $U_{\alpha} \supset B$ and satisfying
$A_i^{\ep} \geq \beta$, we have
\begin{multline*}
  \int_{B} \chi \(|\nabla \h|^2+\frac{\kappa^2}{4 |\ln \varepsilon|}
  |\h|^2\) dx' - \frac{1}{2\pi}\(\ln\frac \rho{n_{\alpha} \rho_{\eps}}
  - c \)\sum_{i\in I_{\beta,B}} \chi_i |\tilde A_i^{\ep}|^2 + \int_{B}
  \chi \omega_{\varepsilon}
  dx' \\
  \ge - C \nu^{\varepsilon}(B) \|\nab\chi\|_{L^\infty(B)} ,
\end{multline*} 
where we set $\bar{r}_\alpha:= \frac{r({\B_0^{\alpha}})}{\ro_\ep}$ and
define (recall that $\alpha$ implicitly depends on $i \in I_\beta$)
\begin{equation}\label{omega} 
  \omega_{\varepsilon}(x')
  := \frac{1}{2\pi}\sum_{i \in I_{\beta}} |\tilde
  A_i^{\ep}|^2 \ln \left( \frac{r(\B_0^{\alpha})}{n_{\alpha}
      \rho_{\eps}} \right) \tilde \delta_i^\eps(x')  =
  \frac{1}{2\pi}\sum_{i \in 
    I_{\beta} } |\tilde A_i^{\ep}|^2 \ln \left(
    \frac{\bar{r}_\alpha}{n_{\alpha}} \right) \tilde \delta_i^\eps(x').
  \end{equation}
  The quantity $\omega_{\ep}$ in some sense measures the discrepancy
  between the droplets $\Omega_{i,\eps}'$ and balls of radius
  $\rho_{\eps}$. We will thus naturally use $M_\ep$ in \eqref{defM} to
  control it. Note also that it is only supported in the droplets,
  hence in the balls of $\B_\rho$.

  Applying Lemma~3.1 of \cite{compagnon} to 
 $$ 
 f_{B,\eps} = \(|\nabla \h|^2 + \frac{\kappa^2}{4 |\ln \varepsilon|}
 |\h|^2- \frac{1}{2\pi} \(\ln\frac \rho{\rho_{\eps} n_{\alpha}} -
 c\)\sum_{i \in I_{\beta,B}} |\tilde A_i^{\ep}|^2\tilde \delta_i^\eps
 \nonumber + \omega_{\varepsilon} \) \indic_{B}
 $$ 
 we deduce the existence of a positive measure $g_{B,\eps}$ such
 that \begin{equation} \|f_{B,\eps} - g_{B,\eps}\|_{\textrm{Lip}^*}
   \le C\nu^{\varepsilon}(B) \label{ballest},\end{equation} where
 $\textrm{Lip}^*$ denotes the dual norm to the space of Lipschitz
 functions and $C > 0$ is universal.
\smallskip

 \noindent {\it - Step 3: Energy displacement on annuli and definition
   of $g_\ep$. }
\\
\noindent We define a set $C_\alpha$ as follows: recall that $\rho$
was assumed equal to $\lambda/(4k)$, where $\lambda \le \frac14 r_0$
and $k$ bounds the number of $\alpha'$'s such that $\dist(U_{\alpha'},
U_\alpha)< \frac12 r_0$ for any given $\alpha$. Therefore the total
radius of the balls in $\B_\rho$ which are at distance less than $r_0$
from $U_\alpha$ is at most $k\rho = \frac{1}{16} r_0$. In particular,
letting $T_\alpha$ denote the set of $t\in
(\frac{r_0}{2},\frac{3r_0}{4})$ such that the circle of center
$x_\alpha$ (where we recall $x_\alpha$ is the center of $U_\alpha$)
and radius $t$ does not intersect $\B_\rho^\alpha$, we have
$|T_\alpha|\ge \frac{3}{16} r_0$.
%If there is
%some $t_0 \in (1/4,3/4)$ such that $M \pi t_0^2=1$ then redefine $T_\alpha = T_\alpha \cap (t_0-1/16,t_0+1/16)^c$
%and proceed.
We let $C_\alpha = \{x\mid |x-x_\alpha|\in T_\alpha\}$ and recall that
$D_\a=B(x_\alpha, \frac{3r_0}{4})$.
%If $U_\alpha\cap U\neq\varnothing$ then $d(x_\alpha,\TTT)\le r_0/4$ hence
%$B(x_\alpha,3r_0/4)\subset\widehat \TTT$. In particular,   we have $C_\alpha\subset D_\a\subset \widehat \TTT$.  

Let $t \in T_\alpha$.  Arguing exactly as in the proof of
  \eqref{lecerclelast}, we find that
$$
\int_{\partial B(x_\alpha, t)} |\nab h_\ep'|^2 \, d\mathcal{H}^1(x')
+\frac{\kappa^2}{4 \lep}\int_{B(x_\alpha, t)} |h_\ep'|^2 \, dx'\ge
\frac{ m_{\ep,t}^2 }{2\pi t}\( 1- \frac{\kappa^2 t}{4 \lep}\)
$$
with $m_{\ep, t}:=\int_{B(x_\alpha, t)} (\mu_\ep'(x')
-\bar{\mu}^{\eps} )\, dx'.$ Arguing as in \eqref{lecercle4filter}
and using the fact that $B(x_\alpha, \frac12 r_0)$ contains all the
droplets with $i\in I_{\beta, U_\alpha}$, we find that we can take
$\ep$ sufficiently small depending on $\kappa$, and $r_0$ sufficiently
small depending on $\kappa $ and $\bar{\delta}$ such that for all $t
\in T_\alpha$,
$$
\int_{\partial B(x_\alpha, t)} |\nab h_\ep'|^2 \, d\mathcal{H}^1(x')
+\frac{\kappa^2}{4 \lep}\int_{B(x_\alpha, t)} |h_\ep'|^2 \, dx'\ge
\frac{1}{4\pi t} \Bigg( \sum_{i \in I_{\beta, U_\alpha}} A_i^\ep
\Bigg)^2 .
$$ 
Integrating this over $t\in T_\alpha$, using that $|T_\alpha| \ge
\frac{3}{16} r_0$, we obtain that
\begin{equation}\label{lacouronne}
  \int_{C_\alpha}|\nabla \h|^2 dx' +  \frac{\kappa^2}{ 4 |\ln
    \varepsilon|} 
  \int_{D_\alpha} |\h|^2 dx' \ge c \Bigg(\sum_{i \in
    I_{\beta, U_\alpha}} A_i^{\ep}\Bigg)^2,
\end{equation}  
with $c > 0$ depending only on $r_0$, hence on $\kappa $ and
$\bar{\delta}$. 

We now trivially extend the estimate in \eqref{lacouronne} to all
$\alpha$'s, including those $U_\alpha$ that contain no droplets of
size greater or equal than $\beta$. The overlap number of the sets
$\{C_\alpha\}_\alpha$, defined as the maximum number of sets to which
a given $x' \in \TTT$ belongs is bounded above by the overlap number
of the sets $\{D_\alpha\}_\alpha$, call it $k'$. Since the latter
collection of balls covers the entire $\TTT$, we have $k' \geq 1$.
Then, letting
\begin{align}\nonumber
  f_{\varepsilon}^{\prime} := f_{\varepsilon} - \sum_{B \in \B_{\rho}}
  f_{B,\eps} &= \( |\nabla \h|^2 + \frac{\kappa^2}{2 |\ln
    \varepsilon|} |\h|^2\)\mathbf{1}_{\TTT \backslash \B_{\rho}} +
  \frac{\kappa^2}{4 |\ln \varepsilon|} |\h|^2 \mathbf{1}_{\B_{\rho}}
  \\\label{cwg}&+ \noindent \frac{1}{2\pi} \sum_{i \in I_\beta} \left(
    \ln \frac{\rho}{n_{\alpha}} - c\right) |\tilde A_i^{\ep}|^2 \tilde
  \delta_i^\eps - \omega_{\varepsilon},
\end{align}
and
\begin{align}
  \label{cwga}
  f_{\alpha,\varepsilon} := \frac{1}{2k'} \(|\nabla \h|^2 +
  \frac{\kappa^2}{4 |\ln \varepsilon|} |\h|^2\) \indic_{C_\alpha} +
  \frac{1}{2\pi}\sum_{i\in I_{\beta,
      \B_\rho^\alpha}}|\tilde{A}_i^{\ep}|^2 \(\ln\frac
  \rho{n_{\alpha}} - c \) \tilde \delta_i^\eps -
  \omega_{\varepsilon}\indic_{\B_\rho^\alpha},
\end{align}
we have 
\begin{align}\nonumber
  f_{\varepsilon}' - \sum_\alpha f_{\alpha,\varepsilon} & \ge
  \(|\nabla \h|^2 + \frac{\kappa^2}{2 |\ln \varepsilon|}
  |\h|^2\)\indic_{\TTT \sm\B_\rho} \\ &- \frac{1}{2k'} \sum_\alpha
  \(|\nabla \h|^2 + \frac{\kappa^2}{4 |\ln \varepsilon|} |\h|^2\)
  \indic_{C_\alpha} + \frac{\kappa^2}{4 |\ln \varepsilon|} |\h|^2
  \mathbf{1}_{\B_{\rho}} \nonumber \\ &\ge \hal \(|\nabla \h|^2 +
  \frac{\kappa^2}{2 |\ln \varepsilon|} |\h|^2\) \indic_{\TTT
    \sm\B_\rho} + \frac{\kappa^2}{4 |\ln \varepsilon|} |\h|^2
  \mathbf{1}_{\B_{\rho}} \ge 0 \label{denfp}
\end{align} 
and from \eqref{lacouronne}
\begin{align}
  f_{\alpha,\varepsilon}(D_\alpha) & = \frac{1}{2k'} \int_{C_\alpha}
  \left( |\nabla \h|^2 + \frac{\kappa^2}{4 |\ln \varepsilon|} |\h|^2
  \right) dx' + \frac{1}{2\pi} \(\ln\frac{\rho}{n_{\alpha}} - c\)
  \sum_{i \in I_{\beta, \B_\rho^\alpha}} |\tilde A_i^{\ep}|^2 -
  \omega_{\eps} (D_\alpha) \nonumber \\
  \label{annulibound} &\ge c \Bigg( \sum_{i \in I_{\beta, U_\alpha}
  }A_i^{\ep}\Bigg)^2 - \frac{1}{2\pi}\ln n_{\alpha} \sum_{i \in
    I_{\beta, \B_\rho^\alpha}} |\tilde A_i^{\ep}|^2 -
  \omega_{\eps}(D_{\alpha}) - C \sum_{i \in I_{\beta, \B_\rho^\alpha}}
  |\tilde A_i^{\ep}|^2,
\end{align} 
for some $C, c > 0$ depending only on $\kappa$ and $\bar\delta$.  Now
we combine the middle two terms, using the definition of
$\omega_{\eps,\alpha}$ in \eqref{omega}, to obtain
\begin{align}\label{annulibound2}
  f_{\alpha,\varepsilon}(D_\alpha)\ge c\Bigg( \sum_{i \in I_{\beta,
        U_\alpha} }A_i^{\ep}\Bigg)^2 - \frac{1}{2\pi} \ln
  \bar{r}_\alpha \sum_{i \in I_{\beta, \B_\rho^\alpha}} |\tilde
  A_i^{\ep}|^2- C \sum_{i \in I_{\beta, \B_\rho^\alpha}} |\tilde
  A_i^{\ep}|^2.
\end{align}

The next step is to bound $\bar{r}_\alpha.$ We separate those
$\Omega_{i,\eps}'$ with $A_i^{\ep} \geq 3^{2/3} \pi \gamma^{-1}$ and
those with $A_i^{\ep} < 3^{2/3} \pi \gamma^{-1}$.  We denote (with $s$
for ``small'' and $b$ for ``big'')
\begin{align*} 
  I_{\beta,\alpha}^s &= \left\{i \in I_{\beta,U_\alpha} :
    A_i^{\ep} \leq 3^{2/3} \pi \gamma^{-1}\right\},\\
  I_{\beta,\alpha}^b &= I_{\beta, U_\alpha} \backslash I_{\beta,\alpha}^s,\\
  n_{\alpha_s} &= \# I_{\beta,\alpha}^s.
\end{align*}
For the small droplets, we use the obvious bound 
\begin{equation}\label{volboundsmall} 
  \sum_{i \in
    I_{\beta,\alpha}^s} |A_i^{\ep}|^{1/2} \leq \ c
  n_{\a_s}, 
\end{equation}
with a universal $c > 0$, while for the large droplets we use that in
view of the definition of $M_\eps$ in \eqref{defM} we have
\begin{equation}\label{volbound} 
  \sum_{i \in
    I_{\beta,\alpha}^b} |A_i^{\ep}|^{1/2} \leq C 
  \sum_{ i \in I_{\beta,\alpha}^b} A_i^{\ep} \le C'
  M_{\varepsilon},
\end{equation}  
for some universal $C,C' > 0$. We can now proceed to controlling
$\bar{r}_\alpha$.  By \eqref{newRdef1} and \eqref{rB00}, for
universally small $\eps$ we have
\begin{equation}\label{rbound2} \bar{r}_\alpha \leq C \sum_{i \in 
    I_{\beta, U_\alpha}} P_i^\ep,
\end{equation}
for some universal $C > 0$.  In view of \eqref{defM}, \eqref{volbound}
and \eqref{volboundsmall}, we deduce from Remark \ref{rem-Meps} that
for universally small $\eps$ we have
\begin{align} 
  \nonumber \bar{r}_\alpha & \le C\Big( {M_{\varepsilon}}
  + \sqrt{4 \pi} \sum_{i \in I_{\beta, U_\alpha}} |A_i^{\ep}|^{1/2} \Big) \\
  & \le C\Big(M_{\varepsilon} + c n_{\alpha_s} + C'
  M_{\varepsilon}\Big) \leq C'' (n_{\alpha_s} + M_\ep) \leq C'' (1 +
  n_{\alpha_s} + M_\eps), \label{radiibound}
\end{align}
where $c, C, C', C'' > 0$ are universal.  Therefore,
(\ref{annulibound2}) becomes
\begin{multline} \label{annulibound4} f_{\alpha, \ep} (D_\alpha)\ge
  c\Bigg( \sum_{i \in I_{\beta, \alpha}^s}A_i^{\ep}\Bigg)^2 + c\Bigg(
  \sum_{i \in I_{\beta,\alpha}^b}A_i^{\ep}\Bigg)^2 - C \ln
  \bar{r}_\alpha \sum_{i \in I_{\beta, \B_\rho^\alpha}} |\tilde
  A_i^{\ep}|^2- C''' \sum_{i \in I_{\beta, \B_\rho^\alpha}}
  |\tilde A_i^{\ep}|^2,\\
  \geq c\beta^2 n_{\alpha_s}^2 + c\Bigg( \sum_{i \in I_{\beta
      ,\alpha}^b}A_i^{\ep}\Bigg)^2 - C' \ln (C''(1 + n_{\alpha_s} +
  M_{\varepsilon})) \Bigg( n_{\alpha_s} + \ \sum_{i \in I_{\beta,
      \alpha}^b} A_i^{\ep} \Bigg) ,
\end{multline}
where $C, C' > 0$ are universal, $c, C'', C''' > 0$ depend only on
$\kappa$ and $\bar\delta$, and $C'$ was chosen so that $C |\tilde
A_i^{\eps}|^2\le C'(A_i^{\eps}+1)$.

We now claim that this implies that
\begin{equation}\label{caf2}
  f_{\alpha, \ep}(D_\alpha) \ge
  \frac{c}{2}\beta^2 n_{\alpha_s}^2 + \frac{c}{2}\Bigg( \sum_{i \in
    I_{\beta,\alpha}^b}A_i^{\ep} \Bigg)^2-  C'''\ln^2(
  M_{\varepsilon} +2), 
\end{equation}
where $C''' > 0$ depends only on $\kappa$ and $\bar{\delta}$.  This is
seen by minimization of the right-hand side, as we now detail. For the
rest of the proof, all constants will depend only on $\kappa$ and
$\bar{\delta}$.  For shortness, we will set $X :=\sum_{i \in
  I_{\beta,\alpha}^b}A_i^{\ep}$.

First assume $n_{\alpha_s} =0$. Then \eqref{annulibound4} can be
rewritten
$$
f_{\alpha, \ep} (D_\alpha) \ge c X^2 - C' \ln (C'' (1 +
M_{\varepsilon}))) X,
$$ 
By minimization of the quadratic polynomial in the right-hand side, we
easily see that an inequality of the form \eqref{caf2} holds.  Second,
let us consider the case $n_{\alpha_s}\ge 1$. We may use the obvious
inequality $\log (1 + x+y) \le \log (1 + x)+\log (1 + y)$ that holds
for all $x \geq 0$ and $y \geq 0$ to bound from below
\begin{multline}\label{bfb}
  \frac{c}{2}\beta^2 n_{\alpha_s}^2 + \frac{c}{2} X^2- C' \ln (C''(1 +
  n_{\alpha_s} + M_{\varepsilon})) (n_{\alpha_s} + X) \geq
  \frac{c}{2}\beta^2 n_{\alpha_s}^2 + \frac{c}{2}X^2 \\
  - C (n_{\alpha_s} + X) - Cn_{\alpha_s} \ln ( n_{\alpha_s} +1)- C X
  \ln ( n_{\alpha_s}+1) - C \ln (M_\ep+1) (n_{\alpha_s} + X).
\end{multline}
It is clear that the first three negative terms on the right-hand
side can be absorbed into the first two positive terms, at the expense
of a possible additive constant, which yields
\begin{multline}
  \label{bfb2}
  \frac{c}{2}\beta^2 n_{\alpha_s}^2 + \frac{c}{2} X^2- C' \ln
  (C''(n_{\alpha_s} + M_{\varepsilon})) (n_{\alpha_s} + X) \\ \ge
  \frac{c}{4}\beta^2 n_{\alpha_s}^2 + \frac{c}{4}X^2 - C \ln (M_\ep+1)
  (n_{\alpha_s} + X) - C.
\end{multline}
Then by quadratic optimization the right hand side of \eqref{bfb2} is
bounded below by $- C \ln^2 (M_\ep+2)$ (after possibly changing the
constant). Inserting this into \eqref{annulibound4}, we obtain
\eqref{caf2}.

We then apply \cite[Lemma~3.2]{compagnon} over $D_\alpha $ to
$f_{\alpha,\varepsilon} + C''' |D_\alpha|^{-1} \ln^2 (M_{\varepsilon}
+2) $, where $C'''$ is the constant in the right-hand side of
(\ref{caf2}).  We then deduce the existence of a measure
$g_{\alpha,\varepsilon}$ on $\TTT$ supported in $D_\alpha$ such that
$g_{\alpha,\varepsilon}\ge - C''' |D_\alpha|^{-1} \ln^2
(M_{\varepsilon} +2)$ and such that for every Lipschitz function
$\chi$
\begin{multline}\label{611b}\left|
  \int_{D_\alpha} \chi (f_{\alpha,\varepsilon} -
  g_{\alpha,\varepsilon}) \, dx' \right|\le 2 \, \diam
  (D_\alpha)\|\nab\chi\|_{L^\infty(D_\alpha)} f_{\alpha,\varepsilon}^-
  (D_\alpha)\\ \le C \ln (n_{\alpha_s} +M_\ep + 2)
  \|\nab\chi\|_{L^\infty(D_\alpha)}\sum_{i \in I_{\beta,
      \B_\rho^\alpha}} |\tilde A_i^{\ep}|^2,
\end{multline} 
and we have used the observation that
\begin{equation}\label{rewritf}
  f_{\alpha,\varepsilon} = \frac{1}{2k'} \(|\nabla \h|^2 +
  \frac{\kappa^2}{4 |\ln \varepsilon|} |\h|^2\) \indic_{C_\alpha} +
  \frac{1}{2\pi} \sum_{i\in I_{\beta, \B_\rho^\alpha}}\(\ln\frac
  \rho{\bar{r}_\alpha } - C \) |\tilde{A}_i^{\ep}|^2 \delta_i^\eps,
\end{equation} 
and \eqref{radiibound} to bound the negative part of $f_{\alpha,
  \ep}$.  In particular, taking $\chi = 1$, we deduce, in view of
\eqref{caf2}, that
\begin{equation}\label{galphan} 
  g_{\alpha,\varepsilon}(D_\alpha) =
  f_{\alpha,\varepsilon}(D_\alpha) \ge \frac{c}{2} \beta^2 n_{\alpha_s}^2 +
  \frac{c}{2}\Bigg( \sum_{i \in I_{\beta,\alpha}^b}A_i^{\ep}\Bigg)^2  -
  C''' \ln^2( M_{\varepsilon} + 2 ),
\end{equation}  
from which it follows that 
\begin{equation}\label{gea}
  g_{\alpha, \ep}(D_\alpha) \ge c' \(n_{\alpha_s}^2 + (\#I_{\beta,
    \alpha})^2\) - C''' \ln^2( M_{\varepsilon} + 2 ) 
  \ge \hal c' n_{\alpha}^2
  - C''' \ln^2( M_{\varepsilon} + 2 ).
\end{equation} 
Recalling the positivity of $g_{B,\eps}$ introduced in Step 2, we now
let
\begin{equation}\label{defg}
  g_{\varepsilon} := \sum_{B\in\B_\rho} g_{B,\varepsilon} +
  \sum_{\alpha}g_{\alpha,\varepsilon} + \(f_{\varepsilon}' -
  \sum_{\alpha} f_{\alpha,\varepsilon}\),
\end{equation} 
and observe that since $f_{\varepsilon}^{\prime} - \sum_\alpha
f_{\alpha,\varepsilon}$ is also non-negative by \eqref{denfp}, and
since $\sum_\alpha g_{\alpha,\varepsilon}$ is bounded below by $-k'
C''' |D_\alpha|^{-1} \ln^2( M_{\varepsilon} + 2 )$, where, as before,
$k'$ is the overlap number of $\{D_\alpha\}_\alpha$, we have $g_\ep
\ge - c \ln^2 (M_{\varepsilon} + 2 )$ for some $c > 0$ depending only
on $\kappa$ and $\bar \delta$, which proves the first item. The second
item follows from \eqref{gea}, \eqref{defg} and the positiveness of
$g_{B, \varepsilon}$ and $\(f_{\varepsilon}' - \sum_{\alpha}
f_{\alpha,\varepsilon}\)$.

  \medskip

  \noindent{- \it  Step 4: Proof of the last item}.\\
\noindent Using the definition of $g_{\varepsilon}$ in \eqref{defg},
for any Lipschitz $\chi$ we have
$$
\int_\TTT \chi g_{\varepsilon} dx' = \sum_{B\in\B_\rho} \int_\TTT \chi
g_{B,\varepsilon} dx' + \sum_{\alpha}\int_\TTT \chi
(g_{\a,\varepsilon}-f_{\a,\varepsilon}) dx' + \int_\TTT \chi
f_{\varepsilon}' dx'.
$$
Hence, in view of \eqref{ballest}, \eqref{cwg} and \eqref{611b} we
obtain for some $C > 0$
\begin{multline}\label{gchi}
 \left| \int_\TTT\chi (f_{\varepsilon}-g_{\varepsilon}) dx'\right| \le
  \sum_{B\in\B_\rho}\left| \(\int_\TTT \chi (g_{B,\varepsilon} -
  f_{B,\varepsilon}) dx' \) \right|+ \sum_{\alpha}\left| \int_\TTT \chi
  (g_{\alpha,\varepsilon} - f_{\alpha,\varepsilon}) dx'\right| \\
  \leq C
  \sum_{B\in\B_\rho}\nu^{\varepsilon}(B)\|\nabla\chi\|_{L^\infty(B)}+
  C \sum_\alpha \ \ln (n_{\alpha_s} +M_\ep + 2
  )\|\nabla\chi\|_{L^\infty(D_\alpha)}\sum_{i\in I_{\beta,
      \B_\rho^\alpha} } |\tilde A_i^{\ep}|^2.
\end{multline}
Using that $|\tilde A_i^{\ep}|^2\le C(A_i^{\ep}+1)$ for a universal $C
> 0$ and  \eqref{volbound}, we have 
$$\sum_{i\in I_{\beta,
    \B_\rho^\alpha} } |\tilde A_i^{\ep}|^2 \le C(n_{\alpha_s}+
M_\ep).$$ Since $n_{\alpha_s}\le n_{\alpha}$, the third item follows
from \eqref{gchi}.
%we have
%\begin{multline*}
  %\ln (n_{\alpha_s} +M_\ep +2 ) \sum_{i\in I_{\beta, \B_\rho^\alpha} }
  %|\tilde A_i^{\ep}|^2 \le C \big( \ln (n_{\alpha_s} + 1) + \ln
 % (M_\eps + 2) \big)
 % \Big( n_{\a_s} + \sum_{i \in I_{ \beta, \alpha}^b} A_i^{\ep} \Big)\\
%  \le C n_{\alpha_s} \ln (n_{\alpha_s} +1)+ C\Bigg( \sum_{i \in I_{
  %    \beta, \alpha}^b} A_i^{\ep}\Bigg) \big(\ln (n_{\alpha_s}+1) +
%  \ln  (M_\ep+2) \big) + C n_{\alpha_s}\ln (M_\ep+2)\\%
%  \le C n_{\alpha_s} \ln (n_{\alpha_s} +1)+ C\( \Big( \sum_{i \in I_{
    %  \beta, \alpha}^b} A_i^{\ep}\Big)^2 + \ln^2 (n_{\alpha_s}+1) +
  %\ln^2 (M_\ep+2) + n_{\alpha_s}^2\).
%\end{multline*}
%Using then that $\ln (n_{\alpha_s}+1)\le n_{\alpha_s}$, and that
%$(\sum_{i\in I^b_{\beta, \alpha}} A_i^\ep)^2$ and $n_{\alpha_s}^2 $
%are both controlled by \eqref{galphan}, we then find that
 %\begin{equation*}
   %\ln (n_{\alpha_s} +M_\ep +2 ) \sum_{i\in I_{\beta, 
     %  \B_\rho^\alpha} } |\tilde A_i^{\ep}|^2 \le C' \( g_{\alpha,
     %\ep}(D_\alpha) + \ln^2 (M_\ep +2) \) .
 %\end{equation*} 
% Combining this with \eqref{gchi} and using the fact that $\rho <
 %\frac14 r_0 < 1$, we are led to the result, provided that $\eps$ is
% small enough.
\end{proof}

We now apply Proposition \ref{wspread} to establish uniform bounds on
$M_\eps$, which characterizes the deviation of the droplets from the
optimal shape.
\begin{pro} \label{Mbound} If \eqref{ZO.1} holds, then $M_\ep $ is
  bounded by a constant depending only on $\sup_{\ep > 0} F^\ep
  [u^\ep]$, $\kappa$, $\bar{\delta}$ and $\ell$.
\end{pro}

\begin{proof}From the last item of Proposition \ref{wspread} applied
  with $\chi \equiv 1$ together with the first item, we
  have $$\int_{\TTT} f_\ep dx'= \int_{\TTT} g_\ep dx' \ge - C\lep
  \ln^2 (M_\ep +2) ,
$$ 
with some $C > 0$ depending only on $\kappa$, $\bar \delta$ and
$\ell$, while from \eqref{ZO.1}, \eqref{minoF} and \eqref{deff}, we
have
$$
C' \ge \ell^2 F^\ep [u^\ep] \ge M_\ep +\frac{2}{\lep} \int_{\TTT}
f_\ep dx' +o(1) \ge M_\ep- C \ln^2 (M_\ep +2) +o_\eps(1),
$$
for some $C' > 0$ depending only on $\sup_{\ep > 0} F^\ep [u^\ep]$,
$\kappa$, $\bar{\delta}$ and $\ell$. The claimed result easily
follows.
\end{proof}

With the help of Proposition \ref{Mbound}, an immediate consequence of
Proposition \ref{wspread} is the following conclusion.

\begin{coro}
  \label{c-gepsC}
  There exists $C > 0$ depending only on $\kappa$, $\bar \delta$,
  $\ell$ and $\sup_{\ep > 0} F^\ep [u^\ep]$ such that if $g_\eps$ is
  as in Proposition \ref{wspread} and \eqref{ZO.1} holds, then $g_\eps
  \geq -C$.
\end{coro}

In the following, we also define the modified energy density $\bar
g_\ep$, in which we include back the positive terms of $M_\ep$ and a
half of $ \frac{\kappa^2}{\lep} |h_\ep'|^2$ that had been ``kept
aside'' instead of being included in $f_\ep$:
\begin{multline} 
  \label{barg} \bar g_\ep := g_\ep + \frac{\kappa^2}{2\lep} |h_\ep'|^2
  + \lep \Bigg\{ \sum_i (P_i^\ep- \sqrt{4 \pi A_i^{\ep}} \,
  \tilde{\delta}_i) + c_1 \sum_{A_i^{\ep} > \pi 3^{2/3} \gamma^{-1}
  } A_i^{\ep}\tilde{\delta}_i \\
  + c_2 \sum_{\beta \le A_i^{\ep} \le \pi 3^{2/3} \gamma^{-1}}
  (A_i^{\ep} - \pi \bar r_{\eps}^2)^2 \tilde{\delta}_i+ c_3
  \sum_{A_i^{\ep} <\beta} A_i^{\ep}\tilde{\delta}_i \Biggr\}
\end{multline} 
where we recall $\bar r_{\varepsilon} = \(\frac{|\ln
  \varepsilon|}{|\ln \rho_{\eps}|}\)^{1/3}$ and $\tilde \delta_i^\eps$
is defined by \eqref{fakedirac}.  These extra terms will be used to
control the shapes and sizes of the droplets as well as to control
$h_\ep'$.  We also point out that in view of \eqref{deff}, \eqref{wg}
and \eqref{minoF}, we have
\begin{equation}\label{Fg}
  \ell^2F^\ep[u^\ep] \ge \frac{2}{\lep} \int_{\TTT} \, \bar{g}_\ep dx'
  + o_\eps(1).
\end{equation}

\section{Convergence}\label{sec6}

In this section we study the consequences of the hypothesis
\begin{equation}
  \forall R > 0, \ {\cal C}_{R} := \limsup_{\varepsilon \to 0}
  \int_{\UR} \bar{g}_{\varepsilon}(x + x^0_\eps )dx < +\infty, 
\end{equation}
where $K_R = [-R,R]^2$ and $(x_\eps^0)$ is such that $x_\eps^0 + K_R
\subset \TTT$. This corresponds to ``good'' blow up centers
$x_\eps^0$, and will be satisfied for most of them.

In order to obtain $o_{\eps}(1)$ estimates on the energetic cost of
each droplet under this assumption, we need good quantitative
estimates for the deviations of the shape of the droplets from balls
of the same volume. A convenient quantity that can be used to
characterize these deviations is the {\em isoperimetric deficit},
defined as (in two space dimensions)
\begin{align}
  \label{isodef}
  D(\Omega_{i,\eps}') := { | \partial \Omega_{i,\eps}'| \over \sqrt{4
      \pi |\Omega_{i,\eps}'|}} - 1.
\end{align}
The isoperimetric deficit may be used to bound several types of
geometric characteristics of $\Omega_{i,\eps}'$ that measure their
deviations from balls. The quantitative isoperimetric inequality,
which holds for any set of finite perimeter, may be used to estimate
the measure of the symmetric difference between $\Omega_{i,\eps}'$ and
a ball. More precisely, we have \cite{fusco08}
\begin{equation}\label{strongiso2} 
  \alpha(\Omega_{i,\eps}')\le C \sqrt{D(\Omega_{i,\eps}')}, 
\end{equation} 
where $C > 0$ is a universal constant and $\alpha(\Omega_{i,\eps}')$
is the Fraenkel asymmetry defined as
\begin{align}
  \label{fraen}
  \alpha(\Omega_{i,\eps}') := \min_{B} \frac{|\Omega_{i,\eps}'
    \triangle B|}{|\Omega_{i,\eps}'|},
\end{align}
where $\triangle$ denotes the symmetric difference between the two
sets, and the infimum is taken over balls $B$ with
$|B|=|\Omega_{i,\eps}'|$. In the following, we will use the notation
$r_i^\eps$ and $a_i^\eps$ for the radii and the centers of the balls
that minimize $\alpha(\Omega_{i,\eps}')$, respectively.

On the other hand, in two space dimensions the following inequality
due originally to Bonnesen \cite{bonnesen24} (for a review, see
\cite{osserman79}) is applicable to $\Omega_{i,\eps}'$:
\begin{align}
  \label{bonnes}
  R_i^\eps \leq r_i^\eps \left( 1 + c \sqrt{D(\Omega_{i,\eps}')}
  \right).
\end{align}
Here $R_i^\eps$ is the radius of the circumscribed circle of the
measure theoretic interior of $\Omega_{i,\eps}'$ and $c > 0$ is
universal.  Indeed, apply Bonnesen inequality to the saturation of
$\Omega_{i,\eps}'$ (i.e., the set with no holes) for each
droplet. Then since the set $\Omega_{i,\eps}'$ is connected and,
therefore, its saturation has, up to negligible sets, a Jordan
boundary \cite{ambrosio2}, Bonnesen inequality applies to it.

\subsection{Main result}

We will obtain local lower bounds in terms of the renormalized energy
for a finite number of Dirac masses in the manner of \cite{BBH}:
\begin{definition}\label{WR}
  For any function $\chi$ and $\varphi \in \mathcal{A}_m$
  (cf. Definition \ref{defam}), we denote
  \begin{equation}W(\varphi, \chi) = \lim_{\eta\to 0} \(
    \hal\int_{\mr^2 \backslash \cup_{p\in\Lambda} B(p,\eta) }\chi
    |\nabla \varphi|^2 dx + \pi \ln \eta \sum_{p\in\Lambda} \chi (p)
    \).
\end{equation}
\end{definition}

We now state the main result of this section and postpone its proof to
Section \ref{LboundW}. Throughout the section, we use the notation of
Sec. \ref{sec5}. To further simplify the notation, we periodically
extend all the measures defined on $\TTT$ to the whole of $\mathbb
R^2$, without relabeling them. We also periodically extend the ball
constructions to the whole of $\mathbb R^2$. This allows us to set,
without loss of generality, all $x_\eps^0 = 0$.

\begin{theorem}\label{Wlbound} Under assumption \eqref{ZO.1}, the
  following holds.
 \begin{itemize}
  \item[ 1.] Assume that for any $R>0$ we have
\begin{equation} \label{gbound}
 \limsup_{\varepsilon \to 0} \bar{g}_{\varepsilon}(\UR) < +\infty,
\end{equation}
where $K_R=[-R,R]^2$.  Then, up to a subsequence, the measures
$\mu_{\varepsilon}'$, defined in \eqref{rescaledmu}, converge in
$(C_0(\mathbb{R}^2))^*$ to a measure of the form $\nu = \ 3^{2/3} \pi
\sum_{a \in \Lambda} \delta_a$ where $\Lambda$ is a discrete subset of
$\mathbb{R}^2$, and $\{\varphi^{\eps}\}_{\varepsilon}$ defined in
\eqref{phiepsdef} converge weakly in $\dot
W_{loc}^{1,p}(\mathbb{R}^2)$ for any $p \in (1,2)$ to $\varphi$ which
satisfies
\[-\Delta \varphi = 2\pi \sum_{a \in \Lambda} \delta_{a} - m \textrm{
  in } \mathbb{R}^2,\] in the distributional sense, with
$m=3^{-2/3}(\bar \delta - \bar \delta_c)$. Moreover, for any sequence
$\{\Omega_{i_\ep, \ep}\}_{\ep} $ which remains in $K_R$, up to a
subsequence, the following two alternatives hold:
 \begin{itemize}
 \item[i.]  Either $A_{i_\ep}^\ep \le \frac{C_R}{\lep}$ and
   $P_{i_\ep}^\ep \le \frac{C_R}{\sqrt{\lep}} $ as $\ep \to 0$,
 % \begin{equation}\label{restexact2} r_{in}((\Omega_{i_\ep,
   %  \ep}')^*)\le r_{out}(\Omega_{i_\ep, \ep}' ) =
   %o\(\rho_{\eps}\) \end{equation} where $(\Omega_{i_\ep,\ep}')^*$
 %denotes the smallest simply connected 
 %set containing $\Omega_{i_\ep,\ep}$.
 \item[ii.] Or $A_{i_\ep}^\ep$ is bounded below by a positive constant
   as $\ep \to 0$, and
 $$
 A_{i_\ep}^\ep \to 3^{2/3} \pi \text{ and } P_{i_\ep}^\ep \to 2 \cdot
 3^{1/3} \pi \quad \text{as} \ \ep \to 0,
 $$
 with
 \begin{equation}\label{frankela}
   \a(\Omega_{i_\ep, \ep}') \le \frac{C_R}{\lep^{1/2}} \quad \text{as} \ \ep
   \to 0,
  \end{equation}
 %begin{align}
  %\label{restexact1} r_{out}(\Omega_{i_\ep, \ep}' ) =
  %r_{in}((\Omega_{i_\ep, \ep}')^* ) & + o\(\rho_{\eps} \)= \rho_{\eps}
  %(1+ o\(1\)).\end{align} Moreover there exist balls $B_{in}$ and
%$B_{out}$ with $B_{in} \subset (\Omega_{i_\ep,\ep}')^* \subset
%B_{out}$ and $a \in \Lambda$ such that
%\begin{align}
  %\label{Hdistest} d_H(B_{in},B_{out}) = o(\rho_{\eps}) \qquad
  %d_H(\Omega_{i_\ep, \ep}', B(a,\rho_{\eps})) &= o\(1\),
%\end{align}
%where $d_H$ denotes the Hausdorff distance between sets in
%$\mathbb{R}^2$.
\end{itemize}
for some $C_R > 0$ independent of $\eps$.
% there exists a ball $B^{\eps} \subset \TTT$
%with $|B^\eps|= |\Omega_{i_\ep, \ep}'|$ and such that
%\begin{equation}\label{Fraenkelsmall}
 % |\Omega_{i_{\ep},\ep}' \triangle B^{\eps}| = o(\rho_{\eps}^2)  \quad
  %\text{as} \ \ep \to 0. 
%\end{equation}

\item[ 2.] If we replace \eqref{gbound} by the stronger assumption
\begin{equation}
  \label{gbound2} \limsup_{\varepsilon \to 0} \bar g_{\varepsilon}(\UR)
  < CR^2,
\end{equation}
where $C > 0$ is independent of $R$, then we have for any $p \in
(1,2)$,
\begin{equation}
  \limsup_{R \to +\infty} \left( {1 \over |K_R|}  \int_{\UR} |\nabla
    \varphi|^p dx \right) < + \infty.
\end{equation}
Moreover, for every family $\{\chi_R\}_{R>0}$ defined in Definition
\ref{Wvdef} we have 
\begin{equation} 
  \label{Wlbound2} 
  \liminf_{\varepsilon \to 0} \int_{\mathbb{R}^2} \chi_R
  \bar g_{\varepsilon} dx \geq {3^{4/3} \over 2} W(\varphi,\chi_R) +
  \frac{3^{4/3} \pi }{8} \sum_{a \in \Lambda} \chi_R(a) + o(|K_R|).
\end{equation}
% as $R \to +\infty$. 
\end{itemize}

\end{theorem}
\begin{remark}
  We point out that it is included in Part 1 of Theorem \ref{Wlbound}
  that at most one droplet $\Omega_{i_{\varepsilon},\varepsilon}'$
  with $A_{i_\ep,\ep}$ bounded from below converges to $a \in
  \Lambda$. Indeed otherwise in the first item we would have
  $\mu_{\eps}' \to 3^{2/3} \pi n_a \sum_{a \in \Lambda} \delta_a$
  where $n_a > 1$ is the number of non-vanishing droplets converging
  to the point $a$.
\end{remark}

Theorem \ref{Wlbound} relies crucially on the following proposition
which establishes bounds needed for compactness. Each of the bounds
relies on \eqref{gbound}. Throughout the rest of this section, all
constants are assumed to implicitly depend on $\kappa$, $\bar \delta$,
$\ell$ and $\sup_{\ep > 0} F^\ep [u^\ep]$. 

%\cm{Should it be better if this proposition be based on \eqref{wg}?}

\begin{lem}\label{Rbounds}
  Let $\bar g_{\varepsilon}$ be as above, assume \eqref{gbound} holds
  and denote $\mathcal{C}_R= \limsup_{\ep \to 0} \bar g_\ep(\UR)$.
  Then for any $R$ and $\varepsilon$ small enough depending on $R$ we
  have
  \begin{equation} \label{Rbounds.11} \sum_{\alpha |_{U_{\alpha}
        \subset \UR} } n_{\alpha}^2 \leq C(\mathcal{C}_{R+C} + R^2),
\end{equation}
\begin{equation}\label{Rbounds.21}
  \sum_{i \in I_{\beta,K_R}} A_i^\eps  \leq
  C(\mathcal{C}_{R+C}+R^2), 
\end{equation}
\begin{align} \label{Rbounds.31}\left| \int_{K_R} \chi_R
    (f_{\varepsilon}-g_{\varepsilon}) dx \right|&\leq C \sum_{\alpha
    |_{U_{\alpha} \subset \URC \backslash \URc}} (n_{\alpha}+1) \ln
  (n_{\alpha}+2) \leq C(\mathcal{C}_{R+C} + R^2),
\end{align}
where $\{\chi_R\}$ is as in Definition \ref{Wvdef} and $n_{\alpha} =
\# I_{\beta,U_{\alpha}}$, with $U_\alpha$ as in the proof of
Proposition \ref{wspread}, for some $C>0$ independent of $\eps$ or
$R$.  Furthermore, for any $p \in (1,2)$ there exists a $C_p > 0$
depending on $p$ such that for any $R>0$ and $\varepsilon$ small
enough
\begin{equation} \label{Rbounds.41} \int_{\UR} |\nabla h_\ep'|^p dx
  \leq C_p(\mathcal{C}_{R+C}+R^2).
\end{equation}
\end{lem}

\begin{proof}
  First observe that the rescaled droplet volumes and perimeters
  $A_i^\ep$ and $P_i^\ep$ are bounded independently of $\ep$, as
  follows from Proposition \ref{Mbound} and the definition of
  $M_\ep$. Then, \eqref{Rbounds.11} and \eqref{Rbounds.21} are a
  consequence of \eqref{gbound}, the second item in Proposition
  \ref{wspread} together with the upper bound on $M_\ep$.  The first
  inequality appearing in \eqref{Rbounds.31} follows from item 3 of
  Proposition \ref{wspread} with the bound on $M_\ep$, where we took
  into consideration that only those $D_\alpha$ that are in the $O(1)$
  neighborhood of the support of $|\nabla \chi_R|$ contribute to the
  sum, along with the observation that the mass of $\nu^\ep$ (of
  \eqref{defnue}) is now controlled by $n_\alpha$ (a consequence of
  the above fact that all droplet volumes are uniformly bounded).  The
  second inequality in \eqref{Rbounds.31} follows from
  \eqref{Rbounds.11}.  The bound \eqref{Rbounds.41} is a consequence
  of Proposition \ref{boulesren} and follows as in \cite{compagnon}
  and \cite{SS2}. We refer the reader to \cite{compagnon}, Lemma 4.6
  or \cite{SS2} Lemma 4.6 for the proof in a slightly simpler setting.
\end{proof}

\subsection{Lower bound by the renormalized energy (Proof of Theorem  
  \ref{Wlbound}) }\label{LboundW}
We start by proving the first assertions of the theorem.\\

\noindent {\it - Step 1: All limit droplets have optimal sizes.}
>From \eqref{barg}, \eqref{gbound} and Corollary \ref{c-gepsC}, for all
$\eps$ sufficiently small depending on $R$ we have
\begin{multline}\label{Mbound5} 
  \int_{\UR} \Bigg( \sum_i \left(P_i - \sqrt{4 \pi |A_i^{\ep}|}
  \right) \tilde \delta_i^\eps + c_1 \sum_{A_i^{\ep} > 3^{2/3} \pi
    \gamma^{-1}}
  A_i^{\ep}\tilde \delta_i^\eps  \\
  + c_2 \sum_{\beta \le A_i^{\ep} \le \pi 3^{2/3} \gamma^{-1}}
  (A_i^{\ep} - \pi \bar r_{\eps}^2)^2\tilde \delta_i^\eps + c_3
  \sum_{A_i^{\ep} <\beta} A_i^{\ep}\tilde \delta_i^\eps \Bigg) dx \le
  \frac{C_{R}}{\lep},
\end{multline} 
where we recall that all the terms in the sums are nonnegative. It
then easily follows that for all $i \in I_{K_R}$ the droplets with
$A_i^\ep> 3^{2/3} \pi \gamma^{-1}$ do not exist when $\ep$ is small
enough depending on $R$, and those with $A_i^\ep<\beta$ satisfy
$A_i^\ep =C_R \lep^{-1}$ and $P_i^\ep \le C_R \lep^{-1/2}$, for some
$C_R > 0$ independent of $\eps$.  This establishes item (i) of Part 1
of the theorem.

It remains to treat the case of $A_i^\ep \in [ \beta, 3^{2/3} \pi
\gamma^{-1} ]$ when $\ep $ is small enough.  It follows from
\eqref{Mbound5} that
\begin{align}
  \label{DOmiest}
  D(\Omega_{i,\eps}') \leq {C_R  \over |\ln \eps|},
\end{align}
for some $C_R > 0$ independent of $\eps$, and since $\bar{r}_\ep =
3^{1/3}+o_\ep(1)$, for all these droplets (or equivalently for all
droplets with $A_i^\ep \geq \beta$) we must have
\begin{equation}\label{Aiconv} 
  A_i^{\ep} \to 3^{2/3} \pi \quad \text{and} \quad P_i^\ep \to 2 \cdot
  3^{1/3} \pi \ \text{ 
    as} \  \ep \to 0.
\end{equation} 
Using \eqref{strongiso2}, \eqref{frankela} easily follows from
\eqref{Aiconv} and \eqref{Mbound5}.

\noindent {\it - Step 2: Convergence results.}\\
>From boundedness of $A_i^\eps$, \eqref{Rbounds.21} and \eqref{Mbound5}
we know that $\# I_{\beta,\UR}$ and $\mu'_\eps(K_R)$ are both bounded
independently of $\varepsilon$ as $\eps \to 0$.  We easily deduce from
this, the previous step and the definition of $\mu_\ep'$ that up to
extraction, $\mu_\ep' $ converges in each $K_R$ to at most finitely
many point masses which are integer multiples of $3^{2/3} \pi$ and,
hence, to a measure of the form $\nu=3^{2/3} \pi \sum_{a \in \Lambda }
d_a \delta_a$, where $d_a\in \mathbb{N}$ and $\Lambda $ is a discrete
set in the whole of $\mathbb R^2$.  In view of \eqref{Rbounds.41}, we
also have $h_\eps' \rightharpoonup h \in \dot W^{1,p}_{loc}(\mathbb
R^2)$ as $\ep \to 0$, up to extraction (recall that we work with
equivalence classes from \eqref{vequiv}).  Finally, from the
definition of $\bar{g}_\ep$ in \eqref{barg} and the bound
\eqref{gbound} we deduce that
$$\frac{\kappa^2}{\lep}\int_{K_R} |h_\ep'|^2 \le C_R$$
from which it follows that $\lep^{-1} h_\ep' $ tends to $0$ in
$L^2_{loc}(\mr^2)$ as $\ep \to 0$.  Passing to the limit in the sense
of distributions in \eqref{heqn1}, we then deduce from the above
convergences that we must have 
\begin{equation}\label{dehnu}
  - \Delta h = 3^{2/3} \pi \sum_{a \in \Lambda }
    d_a \delta_a - \bar \mu \quad \text{ on } \mr^2. 
\end{equation}
We will show below that $d_a=1$ for every $a \in \Lambda$, and when
this is done, this will complete the proof of the first item after
recalling $\varphi^{\eps} = 2 \cdot 3^{-2/3} \h$ and $m = 2 \cdot
3^{-2/3} \bar \mu$.  \medskip

\noindent {\it - Step 3: There is only one droplet converging to any
  limit point $a$.}\\
\noindent In order to prove this statement, we examine lower bounds
for the energy.  Fix $R > 1$ such that $\partial K_R \cap \Lambda =
\varnothing$ and consider $a \in \Lambda \cap K_R$. From Step 1,
\eqref{AP} and Lemma \ref{l-diamP}, for any $\eta \in (0, \frac12)$
such that $\eta < \frac12 \min_{b \in \Lambda \cap K_R \backslash
  \{a\}} |a-b|$ and for all $r<\eta$, all the droplets converging to
$a$ are covered by $B(a, r)$, and $B(a, \eta)$ contains no other
droplets with $A_i^\eps \geq \beta$, for $\ep $ small enough. There
are $d_a \ \geq 1$ droplets in $B(a,r)$ such that $A_i^{\ep} \to
3^{2/3} \pi$ as $\varepsilon \to 0$, let us relabel them as
$\Omega_{1,\eps}', \dots, \Omega_{d_a,\eps}'$.
% Since these sets are disjoint, the balls $B_{int, 1} , \dots,
% B_{int, d_a}$ given by Lemma \ref{lemInt} are also disjoint.

Let $U= B(a, \eta)$. Arguing as in the proof of the first item of
Proposition \ref{boulesren}, by \eqref{Aiconv}, we may construct a
collection $\B_0$ of disjoint closed balls covering $\bigcup_{i \in
  I_{\beta, U} } \Omega_{i,\eps}'$ and satisfying
\begin{equation}\label{rb0}
  r(\B_0)\le C d_a\ro_\ep  < \eta,
\end{equation}
for some universal $C > 0$, provided $\eps$ is small enough, and a
collection of disjoint balls $\mathcal{B}_r$ covering $\B_0$ of total
radius $r \in [r(\B_0), \eta]$. Choosing $r = \eta^3$, which is always
possible for small enough $\eps$, it is clear that
$\mathcal{B}_{\eta^3}$ consists of only a single ball contained in
$B(a,\frac{3}{2}\eta^3)$ for $\varepsilon$ small enough. Applying the
second item of Proposition \ref{boulesren} to that ball, we then
obtain 
\begin{align} 
  \label{dabound1} \int_{\mathcal{B}_{\eta^3} } \left( |\nab h_\ep'|^2
    + \frac{\kappa^2}{4 \lep} |h_\ep'|^2 \right) dx' \geq \frac{1}{2
    \pi} \(\ln\frac{\eta^3}{r(\mathcal{B}_0)} - c\eta^3 \)
  \sum_{i=1}^{d_a} |\tilde A_i^{\ep}|^2.
\end{align}
Therefore, we have
\begin{align} 
  \label{dabound11} 
  \int_{\mathcal{B}_{\eta^3} } \chi_R \left( |\nab h_\ep'|^2 +
    \frac{\kappa^2}{4 \lep} |h_\ep'|^2 \right) dx' \geq \frac{1}{2
    \pi} \(\ln\frac{\eta^3}{r(\mathcal{B}_0)} - c\eta^3 \)
  \(\min_{B(a,\eta)} \chi_R\) \sum_{i=1}^{d_a} |\tilde A_i^{\ep}|^2.
\end{align}
 On the other hand, we can estimate the contribution of the remaining
part of $B(a, \eta)$ as
  \begin{multline}
    \int_{B(a, \eta) \sm \mathcal{B}_{\eta^3}} \chi_R|\nab h_\ep'|^2
    \,dx'+ \frac{\kappa^2}{4 \lep}\int_{B(a, \eta)} \chi_R |h_\ep'|^2
    \, dx'\\ \geq \(\min_{B(a,\eta)} \chi_R\) \left( \int_{B(a, \eta)
        \sm B(a, 2 \eta^3)} |\nab h_\ep'|^2 \,dx' + \frac{\kappa^2}{4
        \lep}\int_{B(a, \eta)} |h_\ep'|^2 \, dx'
    \right) \\
    \geq \(\min_{B(a,\eta)} \chi_R\) \int_{2 \eta^3}^\eta \left(
      \int_{\partial B(a, r_B)} |\nab h_\ep'|^2 \, d \mathcal H^1(x')
      + \frac{\kappa^2}{4 \lep}\int_{B(a, r_ B)} |h_\ep'|^2 \, dx'
    \right) dr_B.
    \label{minchiR}
  \end{multline}
 Arguing as in \eqref{lecercle4filter} and using the fact that
  $\eta < \frac12$, we obtain
\begin{multline} 
  \int_{ B(a,\eta) \backslash \mathcal{B}_{\eta^3} }\chi_R |\nab
  h_\ep'|^2 dx' + \frac{\kappa^2}{4 \lep} \int_{B(a, \eta)} \chi_R
  |h_\ep'|^2 dx' \\ \geq \frac{1}{2 \pi} \(\min_{B(a,\eta)} \chi_R\)
  \ln \frac{1}{2\eta^{2}} \(\sum_{i=1}^{d_a}
  A_i^{\ep}\)^2(1-C\eta), \label{dabound2}
\end{multline}
where $C > 0$ is independent of $\eta$ and $\varepsilon$, for small
enough $\eps$.

We will now use crucially the fact shown in Step 1 that all $A_i^{\ep}
\geq \beta$ approach the same limit as $\eps \to 0$. We begin by
adding \eqref{dabound1} and \eqref{dabound2} and subtracting
$\frac{1}{2 \pi} | \ln \rho_{\eps} | \sum_{i=1}^{d_a} |\tilde
A_i^{\ep}|^2 \chi_R^i$ from both sides. With the help of \eqref{rb0}
we can cancel out the leading order $O(|\ln \rho_\eps|)$ term in the
right-hand side of the obtained inequality. Replacing $\tilde A_i^\ep$
and $A_i^\ep$ with $3^{2/3} \pi +o_\ep(1)$ in the remaining terms and
using the fact that $\min_{B(a,\eta)}\chi_R \geq \chi_R(a) - 2
\eta \|\nabla \chi_R\|_{\infty}$ on $B(a,\eta)$, we then find
\begin{multline} 
  \label{dabound3} 
  \int_{B(a,\eta) } \chi_R \left( |\nab h_\ep'|^2 + \frac{\kappa^2}{2
      \lep} |h_\ep'|^2 - \frac{1}{2 \pi} |\ln \rho_{\eps}|
    \nu^{\varepsilon} \right) \, dx' \\ \geq {3^{4/3} \pi \over 2}
  \chi_R(a) \( d_a^2\ln \frac{1}{2\eta^2} + d_a\ln \frac{\eta^3}{2} \)
  - C,
\end{multline}
where $C > 0$ is independent of $\eps$ or $\eta$.

Now, adding up the contributions of all $a \in \Lambda \cap K_R$ and
recalling the definition of $f_\eps$ in \eqref{deff}, we conclude that
on the considered sequence
\begin{multline}
  \label{fepsadded}
  \limsup_{\eps \to 0} \int_{K_R} \chi_R f_\eps \, dx' \geq
  \limsup_{\eps \to 0} \sum_{a \in \Lambda \cap K_R} \int_{B(a, \eta)}
  \chi_R f_\eps \, dx' \\
  \geq {3^{4/3} \pi \over 2} |\ln \eta| \sum_{a \in \Lambda \cap K_R}
  (2 d_a^2 - 3 d_a) \chi_R(a) - C,
\end{multline}
for some $C > 0$ is independent of $\eps$ or $\eta$. In particular,
since $\chi_R(a) > 0$ for all $a \in \Lambda \cap K_R$, the right-hand
side of \eqref{fepsadded} goes to plus infinity as $\eta \to 0$,
unless all $d_a = 1$. But by the estimate \eqref{Rbounds.31} of
Proposition \ref{Rbounds}, Corollary \ref{c-gepsC} and our assumption
in \eqref{gbound} together with \eqref{barg}, the left-hand side of
\eqref{fepsadded} is bounded independently of $\eta$, which yields the
conclusion.

\medskip
\noindent {\it - Step 4: Energy of each droplet}.  Now that we know
that for each $a_i \in \Lambda \cap K_R$ there exists exactly one
droplet $\Omega_{i,\eps}'$ such that $a_i^\eps \to a_i$ and $A_i^\eps
\to 3^{2/3} \pi$, we can extract more precisely the part of
energy that concentrates in a small ball around each such
droplet.  Let $B_i$ be a ball that minimizes Fraenkel asymmetry
defined in \eqref{fraen}, i.e., let $B_i = B(a_i^\eps, r_i^\eps)$, and
let $B$ be a ball of radius $r_B$ centered at $a_i^\eps$. Arguing as
in \eqref{lecercle4filter} in the proof of the second item of
Proposition \ref{boulesren}, we can write
\begin{multline}
  \label{lecerclelastlast}
  \int_{\partial B} |\nabla \h|^2 d \mathcal H^1(x') + {\kappa^2 \over
    4|\ln \eps|} \int_B |\h|^2 dx' \geq {\eps^{-4/3} |\ln \eps|^{-2/3}
    | \Omega_{i,\eps}' \cap B|^2 \over 2 \pi r_B} \left( 1 - c
      r_i^\eps \right).
\end{multline}
Observe that by the definition of Fraenkel asymmetry 
% we have $| \Omega_{i,\eps}' \cap B| \geq | \Omega_{i,\eps}' \cap
% B_i| = |\Omega_{i,\eps}'| \left( 1 - \frac12
%   \alpha(\Omega_{i,\eps}') \right)$ for all $r_B > r_i^\eps$. Taking
% into consideration that $|\Omega_{i,\eps}'| = \eps^{2/3} |\ln
% \eps|^{1/3} A_i^\eps$, for all $r_B \in (r_i^\eps, \eta')$ with
% $\eta' < \eta$ sufficiently close to $\eta$ specified in Step 3, we
% then find that
% \begin{align}
%   \label{lecerclelastlast2}
%   \int_{\partial B} |\nabla \h|^2 d \mathcal H^1(x') + {\kappa^2
%   \over 2 |\ln \eps|} \int_B |\h|^2 dx' \geq {1 \over 2 \pi r_B}
%   |A_i^\eps|^2 \left(1 - \tfrac12 \alpha(\Omega_{i,\eps}')
%   \right)^2
%   \left( 1 - {C \eta \over |\ln \eps|} \right),
% \end{align}
% for some $C > 0$ and all $\eps$ sufficiently small. Integrating the
% expression in \eqref{lecerclelastlast2} over this interval of $r_B$
% and using \eqref{frankela}, \eqref{Aiconv} and the fact that
% $r_i^\eps / \rho_\eps \to 1$ as $\eps \to 0$, we then find
% \begin{multline}
%   \label{lecerclelastlast3}
%   \int_{B(a_i^\eps, \eta') \sm B_i} |\nabla \h|^2 \, dx' + {\kappa^2
%   \over 2 |\ln \eps|} \int_{B(a_i^\eps, \eta')} |\h|^2 \, dx' \\
%   \geq {3^{4/3} \pi \over 2} \ln \left({ \eta' \over \rho_\eps}
%   \right) (1 + o_\eps(1)).
% \end{multline}
% It remains to integrate \eqref{lecerclelastlast} over all $r_B \leq
% r_i^\eps$. To this end, we use the fact that
we have $|\Omega_{i,\eps}' \cap B| \geq |B| - \frac12
\alpha(\Omega_{i,\eps}') |B_i|$ for all $r_B < r_i^\eps$. Hence,
denoting by $\tilde r_i^\eps$ the smallest value of $r_B$ for which
the right-hand side of this inequality is non-negative and integrating
from $\tilde r_i^\eps$ to $r_i^\eps$, we find
\begin{multline}
  \label{lecerclelastlast4}
  \int_{B_i} \left( |\nabla \h|^2 + {\kappa^2 \over 4 |\ln \eps|}
    |\h|^2 \right) dx' \\
  \geq {\pi \over 2} (1 + o_\eps(1)) \eps^{-4/3} |\ln \eps|^{-2/3}
  \int_{\tilde r_i^\eps}^{r_i^\eps} r_B^{-1} (r_B^2 - |\tilde
  r_i^\eps|^2)^2 d r_B.
\end{multline}
Since by \eqref{frankela} and \eqref{Aiconv} we have $\tilde r_i^\eps
/ r_i^\eps \to 0$ and $\eps^{-1/3} |\ln \eps|^{-1/6} r_i^\eps \to
3^{1/3}$ as $\eps \to 0$, after an elementary computation we find
\begin{align}
  \label{lecerclelastlast5}
  \int_{\Omega_{i,\eps}'} \left( |\nabla \h|^2 + {\kappa^2 \over 4
      |\ln \eps|} |\h|^2 \right) dx' \geq {3^{4/3} \pi \over 8} +
  o_\eps(1).
\end{align}

On the other hand, by \eqref{bonnes} and \eqref{DOmiest} it is
  possible to choose a collection $\B_0 \subset B(a_i, \eta)$,
  actually consisting of only a single ball $B(\tilde a_i^\eps,
  R_i^\eps)$ circumscribing $\Omega_{i,\eps}'$, so that
\begin{align}
  \label{rBobonn}
  r(\B_0) = R_i^\eps \leq r_i^\eps \left( 1 + C_R |\ln \eps|^{-1/2}
  \right) = \rho_\eps + o_\eps(\rho_\eps).
\end{align}
The corresponding ball construction $\B_r$ of the first item of
Proposition \ref{boulesren}, with $U = B(a_i, \eta)$ and $\eta$ as in
Step 3 of the proof (again, just a single ball $B(\tilde a_i^\eps,
r)$), exists and is contained in $U$ for all $r \in [r(\B_0), \eta']$,
for any $\eta' \in (r(\B_0), \eta)$, provided $\eps$ is sufficiently
small depending on $\eta'$.  In view of the fact that for small enough
$\eta'$ and small enough $\eps$ depending on $\eta'$ we have $\chi_R(x)
\geq \chi_R(\tilde a_i^\eps) - c |x - \tilde a_i^\eps| > 0$, with $c >
0$ independent of $\eps$, $\eta'$ or $R$, we obtain that
\begin{multline}
  \label{lecercleannulus}
  \int_{B(\tilde a_i^\eps, \eta') \sm \B_0} \chi_R |\nabla \h|^2 dx' +
  {\kappa^2 \over 4 |\ln \eps|} \int_{B(\tilde a_i^\eps, \eta')}
  \chi_R
  |\h|^2 dx' \\
  \geq \int_{r(\B_0)}^{\eta'} (\chi_R(\tilde a_i^\eps) - c r) \left(
    \int_{\partial \B_r} |\nabla \h|^2 d \mathcal H^1(x) \right) dr +
  {\kappa^2 \chi_R(\tilde a_i^\eps) \over 8 |\ln \eps|} \int_{B(\tilde
    a_i^\eps,
    \eta')} |\h|^2 dx' \\
  \geq \int_{r(\B_0)}^{\eta'} (\chi_R(\tilde a_i^\eps) - c r) \left(
    \int_{\partial \B_r} |\nabla \h|^2 d \mathcal H^1(x) + {\kappa^2
      \over 8 \eta' |\ln \eps|} \int_{B(\tilde
      a_i^\eps, \eta')} |\h|^2 dx' \right) dr \\
  \geq \int_{r(\B_0)}^{\eta'} (\chi_R(\tilde a_i^\eps) - c r) \left(
    \int_{\partial \B_r} |\nabla \h|^2 d \mathcal H^1(x) + {\kappa^2
      \over 4 |\ln \eps|} \int_{\B_r} |\h|^2 dx' \right) dr  \\
  \geq {1 \over 2 \pi} |\tilde A_i^\eps|^2 \int_{r(\B_0)}^{\eta'}
  (\chi_R(\tilde a_i^\eps) - c r) (1 - C r) {dr \over r},
\end{multline}
for $\eta'$ and $\eps$ sufficiently small, arguing as in
\eqref{lecercle4filter} in the proof of Proposition \ref{boulesren}
and taking into account Remark \ref{moinsb0} in deducing the last
line. Performing integration in \eqref{lecercleannulus} and using
\eqref{rBobonn}, we then conclude
\begin{multline}
  \label{lecercleannulus2}
  \int_{B(\tilde a_i^\eps, \eta') \sm \B_0} \chi_R |\nabla \h|^2 dx' +
  {\kappa^2 \over 4 |\ln \eps|} \int_{B(\tilde a_i^\eps, \eta')}
  \chi_R
  |\h|^2 dx' \\
  \geq {1 \over 2 \pi} |\tilde A_i^\eps|^2 \chi_R(\tilde a_i^\eps) \ln
  \left( {\eta' \over \rho_\eps} \right) - C \eta',
\end{multline}
for $\eps$ sufficiently small.

\noindent
- {\it Step 5: Convergence}.  Using the fact, seen in Step 2,
that $\h \rightharpoonup h$ in $\dot W^{1,p}_{loc}(\mathbb R^2)$, we
have, by lower semi-continuity,
\begin{equation}\label{partieh}
  \liminf_{\eps \to 0}
  \int_{\mr^2 \backslash \cup_{a \in \Lambda } B(a, \eta)} \chi_R
   |\nab h_\ep'|^2\, dx'
  \ge \int_{\mr^2 \backslash \cup_{a \in \Lambda }B(a,
    \eta)}\chi_R |\nab h|^2  dx'.
\end{equation} 
On the other hand, in view of $\chi_R(\tilde a_i^\eps) = \chi_R^i +
O(\rho_\eps)$ by \eqref{rBobonn}, from \eqref{lecercleannulus2} we
obtain
\begin{multline}
  \label{lecerclelastlast55}
  \liminf_{\eps \to 0} \int_{B(a_i^\eps, \eta) \sm \B_0} \chi_R
  |\nabla \h|^2 dx' + \int_{B(a_i^\eps, \eta)} \chi_R \left( {\kappa^2
      \over 4 |\ln \eps|}
    |\h|^2 - \frac{1}{2 \pi} |\ln \rho_\eps| \nu^\eps \right) dx' \\
  \geq \liminf_{\eps \to 0} \int_{B(\tilde a_i^\eps, \eta') \sm \B_0}
  \chi_R |\nabla \h|^2 dx' + \int_{B(\tilde a_i^\eps, \eta')} \chi_R
  \left( {\kappa^2 \over 4 |\ln \eps|} |\h|^2 - \frac{1}{2 \pi} |\ln
    \rho_\eps| \nu^\eps \right)
  dx' \\
  \geq {3^{4/3} \pi \over 2} \chi_R(a_i) \ln \eta' - C \eta',
\end{multline}
where we also used that $\chi_R^i \to \chi_R(a)$ as $\eps \to 0$.

We now convert the estimate in \eqref{lecerclelastlast5} to one over
$\B_0$ and involving $\chi_R$ as well. Observing that
$\Omega_{i,\eps}' \subseteq \B_0$ and that $\chi_R(x') \geq \chi_i^R -
4 \rho_\eps \| \nabla \chi_R \|_\infty$ for all $x' \in
\Omega_{i,\eps}'$ and $\eps$ small enough by \eqref{rBobonn}, from
\eqref{lecerclelastlast5} and \eqref{newRdef1} we obtain
\begin{align}
  \label{lecerclelastlast6}
  \liminf_{\eps \to 0} \int_{\B_0} \chi_R \left( |\nabla \h|^2 +
    {\kappa^2 \over 4 |\ln \eps|} |\h|^2 \right) dx' \geq {3^{4/3} \pi
    \over 8} \chi_R(a_i),
\end{align}
where we used the fact that by \eqref{P1.N34}, \eqref{minoF} and
\eqref{Pibound} the integral in the left-hand side of
\eqref{lecerclelastlast5} may be bounded by $C |\ln \eps|$, for some
$C > 0$ independent of $\eps$ and $R$.  Adding up \eqref{partieh}
  with \eqref{lecerclelastlast55} and \eqref{lecerclelastlast6} summed
  over all $a_i \in K_R$, in view of  the arbitrariness of $\eta' < \eta$
  we then obtain 
\begin{multline} 
  \label{dabound32} 
  {\liminf_{\eps \to 0}} \int_{\mathbb{R}^2}\chi_R\( |\nab
  h_\ep'|^2 + \frac{\kappa^2 }{2\lep} |h_\ep'|^2 -
  \frac{1}{2 \pi} |\ln
  \rho_{\eps}| \nu^{\varepsilon}\) dx' \\
  \geq \int_{\mr^2 \backslash \cup_{a \in \Lambda }B(a, \eta)}\chi_R
  |\nab h|^2 dx' + \frac{3^{4/3} \pi}{2} \sum_{a \in \Lambda}
  \chi_R(a) \( \ln \eta + \frac{1}{4}\) - C \eta.
%  - C\Delta(R) +
  % o_{\varepsilon}(1),
\end{multline}
Letting now $\eta \to 0$ in \eqref{dabound32}, and recalling that
$\varphi= 2 \cdot 3^{-2/3}h$ and that the definition of
$W(\varphi,\chi)$ is given by Definition \ref{WR}, we obtain
\begin{multline}\label{prodfin}
  \liminf_{\varepsilon \to 0} \int_{\mathbb R^2} \chi_R \(|\nab
  h_\ep'|^2 + \frac{\kappa^2}{2 \lep} |h_\ep'|^2 -
  \frac{1}{2\pi}|\ln \ro_{\eps} |\nu^{\varepsilon}\) dx'
  \\
  \geq \frac{3^{4/3}}{2} W(\varphi,\chi_R) + \frac{3^{4/3}
    \pi}{8}\sum_{a \in \Lambda} \chi_R(a).
\end{multline}

>From \eqref{Rbounds.31}  we may replace $f_\eps =
|\nab h_\ep'|^2 + \frac{\kappa^2}{2 \lep} |h_\ep'|^2 -
\frac{1}{2\pi} |\ln \rho_{\eps}| \nu^{\varepsilon}$ by
$g_{\varepsilon}$ in \eqref{prodfin} with an  additional error term:  
% \ed{which can be absorbed into $\Delta(R)$ provided we redefine it as}:
\begin{align}\label{prodfin2}
  \liminf_{\varepsilon \to 0} \int_{\mathbb R^2} \chi_R
  g_{\varepsilon} dx' \geq \frac{3^{4/3}}{2} W(\varphi,
  \chi_R) + \frac{3^{4/3} \pi}{8} \sum_{a \in \Lambda} \chi_R(a)
  - c \Delta(R),
\end{align}
where 
\[
\Delta(R) = \limsup_{\varepsilon \to 0} \sum_{\alpha|_{K_{R-C} \subset
    U_{\alpha} \subset \URC}} (n_{\alpha}+1) \ln (n_{\alpha}+2),
\]
for some $c, C > 0$ independent of $R$.  Under hypothesis
(\ref{gbound2}), from \eqref{Rbounds.11} we have
\[\limsup_{\varepsilon \to 0} \sum_{\alpha|_{U_{\alpha} \subset \UR}}
n_{\alpha}^2 \leq CR^2,\] and thus, using H\"older inequality and
bounding the number of $\alpha$'s involved in the sum by $CR$ we find
\begin{multline*}
 {\Delta(R)} \leq C \ {\limsup_{\eps \to 0}}
\sum_{\alpha|_{U_{\alpha} \subset \URC \backslash \URc}}
(n_{\alpha}^{3/2} +1) \\ \leq {C'} R^{1/4} \ {\limsup_{\eps \to 0}}
\left( \sum_{\alpha|_{{U}_{\alpha} \subset \URC}}
  n_{\alpha}^2\right)^{3/4} + CR \leq C'' R^{7/4},
\end{multline*}
for some $C, C', C'' > 0$ independent of $R$. Hence 
\[
\limsup_{R \to \infty} \limsup_{\varepsilon \to 0}
\frac{\Delta(R)}{R^2} = 0,
 \]
 which together with \eqref{prodfin2} and the fact that $\bar g_\eps
 \geq g_\eps$ establishes (\ref{Wlbound2}). \qed

 \subsection{Local to Global bounds via the Ergodic Theorem: proof of
   Theorem \ref{main}, item i.}

 The proof follows the procedure outlined in \cite{SS2}.  We refer the
 reader to Sections 4 and 6 of \cite{SS2} for the proof adapted to the
 case of the magnetic Ginzburg-Landau energy, which is essentially
 identical to the present one, with some simplifications due to the
 fact that we work on the torus. As in \cite{SS2}, we say that $\mu
 \in \mathcal{M}_0(\mathbb R^2)$, if the measure $d\mu + Cdx$ is a
 positive locally bounded measure on $\mathbb R^2$, where $C$ is the
 constant appearing in Corollary \ref{c-gepsC}. The measures $d \bar
 g_\ep$ and the functions $\vp_\ep$ will be alternatively seen as
 functions on $\TTT$ or as periodically extended to the whole of
 $\mr^2$, which will be clear from the context. We let $\chi$ be a
 smooth non-negative function on $\mathbb R^2$ with support in
 $B(0,1)$ and with $\int_{\mathbb R^2} \chi(x) dx =1$.  We set $X=\dot
 W_{loc}^{1,p}(\mathbb{R}^2) \times \mathcal{M}_0 (\mathbb R^2)$, and
 define for every $\mathbf x = (\varphi, g) \in X$ the following
 functional
\begin{equation} 
  \label{P1.N35}
  \mathbf{f}(\mathbf{x}) :=  2 \int_{\mathbb R^2} \chi(y) dg(y). 
    % \left\{\begin{array}{ll} \int_{\TTT} \chi(y) d\bar{
    %     g}_{\varepsilon}(y) & \mbox{ if $\exists x \in \TTT$
    %     s.t. $\mathbf x =
    %     (\varphi^{\eps}(x+\cdot),\bar{
    %     g}_{\varepsilon}(x+\cdot))$},\\
    %     +\infty & \mbox{otherwise}.\end{array} \right.
\end{equation} 
We note that from \eqref{Fg} we have for $\eps$ sufficiently small
\begin{equation}\label{Ergodic1}
  F^\ep[u^\ep] + o_\eps(1) \ge \frac{2}{\ell^2\lep}
  \int_{\TTT}d\bar{g}_\ep = 
  \dashint_{\TTT}  \mathbf f(\theta_\lambda \mathbf x_\eps) d\lambda,  
\end{equation}
%  \tilde F_{\varepsilon}(\tilde \varphi^{\eps},\tilde  g_{\varepsilon}) := \frac{1}{\ell^2} \int_{K_{\varepsilon}} \mathbf f_{\varepsilon}(\theta_{\lambda} \tilde \varphi^{\eps},\theta_{\lambda}\tilde g_{\varepsilon})d\lambda = \int \chi * 1_{K_{\varepsilon}} d\tilde  g_{\varepsilon},
%\end{equation}
where $\mathbf x_\eps := ( \varphi^{\eps}, \bar g_{\varepsilon})$,
  $\theta_\lambda$ denotes the translation operator by $\lambda \in
  \mathbb R^2$, i.e., $\theta_\lambda f(x) := f(x + \lambda)$, and
  $\dashint$ stands for the average. Here the last equality follows
by an application of Fubini's theorem and the fact that $\int_{\mathbb
  R^2} \chi (x) dx = 1$. 

% It can be easily shown as in \cite{SS2} that
% $\textbf{f}_{\varepsilon}$ satisfies the coercivity and
% $\Gamma$-liminf properties required for the application of Theorem
% \ed{3} in \cite{SS2}, with $\textbf{f}$ defined over $X$ by
% \begin{equation} \label{P1.N36}
%  \textbf{f}(\varphi,g) = \int \chi dg.
% \end{equation}

It can be easily shown as in \cite{SS2} that $\mathbf f_\eps =
  \mathbf f$ satisfies the coercivity and $\Gamma$-liminf properties
required for the application of Theorem 3 in \cite{SS2} on
  sequences consisting of $\mathbf x_\eps = (\varphi^\eps, \bar
  g_\eps)$ obtained from $(u^\eps)$ obeying \eqref{ZO.1}. This is
done by starting with a sequence
$\{\mathbf{x}_{\varepsilon}\}_{\varepsilon}$ in $X$ such that
\begin{equation}\label{ErgodicBound} 
  \limsup_{\varepsilon \to 0} \int_{K_{R}}
  \mathbf{f}(\theta_{\lambda}\mathbf{x}_{\varepsilon}) 
  d\lambda < +\infty,
\end{equation} 
for every $R>0$, which implies that the integral is finite whenever
$\varepsilon$ is small enough. Consequently
$\mathbf{f}_{\varepsilon}(\theta_{\lambda} \mathbf{x}_{\varepsilon}) <
+\infty$ for almost every $\lambda \in K_R$.  Applying Fubini's
theorem again, \eqref{ErgodicBound} becomes
\[
\limsup_{\varepsilon \to 0} \int_{\mathbb R^2} \chi_R(y)
d\bar{g}_{\varepsilon}(y) < +\infty,
\]
where $\chi_R = \chi * \mathbf{1}_{K_R}$, and ``$*$'' denotes
convolution.  Then since $\chi_R = 1$ in $K_{R-1}$ and $\bar{g}_\ep$
is bounded below by a constant, the assumption \eqref{gbound} in Part
1 of Theorem \ref{Wlbound} is satisfied, and we deduce from that
  theorem that $\varphi^{\eps}$ and $\bar g_{\varepsilon}$ converge,
upon extraction of a subsequence, weakly in $\dot
W^{1,p}_{loc}(\mathbb R^2)$ and weakly in the sense of measures,
respectively. Furthermore, if $\mathbf x_\eps \to \mathbf x =
(\varphi, g)$ on this subsequence, we have $2 \int_{\mathbb R^2}
\chi(y) d\bar g_\eps(y) = \mathbf f(\mathbf x_\eps) \to \mathbf
f(\mathbf x) = 2 \int_{\mathbb R^2} \chi(y) d\bar g(y) $.

We may then apply Theorem 3 of \cite{SS2} to $\mathbf f$ on $\TTT$ and
conclude that the measure $\{\widetilde
P^{\varepsilon}\}_{\varepsilon}$ defined as the push-forward of the
normalized uniform measure on $\mathbb{T}_{\ell^{\eps}}^2$ by
\[
{\lambda} \mapsto (\theta_\lambda \varphi^{\eps},
  \theta_\lambda \bar g_{\varepsilon}),
\]
converges to a translation-invariant probability measure
$\widetilde P$ on $X$ with 
\begin{equation}\label{Ergodic4}
  \liminf_{\varepsilon \to 0}F^\ep[u^\ep] \geq \int \mathbf
    f(\mathbf{x})  d \widetilde P (\mathbf x) =\int \ \mathbf
  f^*(\mathbf{x})  d \widetilde P (\mathbf x), 
\end{equation} 
where 
\begin{align}
  \label{Ergodic44}
  \mathbf f^*(\varphi,g)= \lim_{R\to \infty} \dashint_{\UR}
  \mathbf f(\theta_{\lambda} \mathbf{x})d\lambda = \lim_{R \to
      +\infty} \left( {2 \over |K_R|} \int_{\mr^2} \chi_{R}(y)
      dg(y) \right),
\end{align}
provided that $\mathbf x$ is in the support of $\widetilde P$.

The next step is to show that for $\widetilde{P}$-a.e. $\mathbf{x}$
we have $\varphi \in \mathcal{A}_m$ with $m = 3^{-2/3} (\bar
  \delta - \bar \delta_c)$, and $\mathbf{f^*}$ can be
computed. 
By \cite[Remark 1.6]{SS2}, we have that for  $\widetilde P$-a.e $\mathbf{x}$,  there exists a sequence
$\{{\lambda}_{\varepsilon}\}_{\varepsilon}$ such that $\mathbf
  {x_\eps} = (\theta_{\lambda_\eps} \varphi^{\eps},
  \theta_{\lambda_\eps} \bar g_{\varepsilon})$ converges to $\mathbf
x$ in $X$. In addition, from \eqref{Ergodic4}--\eqref{Ergodic44}, for $\widetilde{P}$-a.e. $\mathbf{x}$, we have 
\[
\lim_{R \to +\infty} \dashint_{\UR} \mathbf f(\theta_{\lambda}
\mathbf{x})d\lambda \ < +\infty,
\]
for $\widetilde P$-almost every $\mathbf{x}$. Using Fubini's
theorem again, together with the definition of $\mathbf{f}$, we then
find $$ \lim_{R \to +\infty} \left( {1 \over |K_R|} \int_{\mr^2} \chi_
  R(y) dg(y) \right) < +\infty.
$$
Therefore, since 
\begin{align}
  \label{gbeg}
  \int_{\mathbb R^2} \chi_R(y) d \bar g_\eps(y) \to \int_{\mathbb
      R^2} \chi_R(y) d g(y) \qquad \text{as } \eps \to 0,
\end{align}
a bound of the type \eqref{gbound2} holds, and the results of Part
  2 of Theorem \ref{Wlbound} hold for $\mathbf
    {x}_\eps$.  In particular, we find that
\begin{equation}\label{dpm}
-\Delta \varphi = 2\pi \sum_{a \in \Lambda} \delta_a - m,
\end{equation}
with $m = 3^{-2/3} (\bar \delta - \bar \delta_c)$, and that
\begin{equation}\label{Ergodic5} 
  \mathbf f^*(\varphi,g) = \lim_{R \to \infty} \left( {2 \over
      |K_R|} \lim_{\eps \to 
      0} \int_{\mathbb R^2} \chi_R 
    \bar g_\eps dx \right) \geq 
  3^{4/3} W(\varphi) + \frac{
    3^{4/3} }{8} m.  
\end{equation} 
The result in \eqref{Ergodic5} follows from the definition of
$\mathbf{f}^*$, \eqref{gbeg}, \eqref{Wlbound2}, the definition of $W$,
provided we can show that
\begin{equation}\label{numpoints} 
  \lim_{R \to +\infty} \frac{1}{|\UR|}
  \sum_{a \in \Lambda} \chi_R(a) =  
  \lim_{R\to + \infty}\frac{\nu(\UR)}{2\pi |\UR|} = \frac{m}{2\pi}.
\end{equation}
The latter can be obtained from \eqref{Rbounds.41}, exactly as in
Lemma 4.11 of \cite{SS2}, so we omit the proof. Note that with
\eqref{dpm}, it proves that $\varphi \in \mathcal{A}_m$, and we thus
have the claimed result.  Combining \eqref{Ergodic4} and
\eqref{Ergodic5}, we obtain
$$\liminf_{\varepsilon \to 0}F^\ep[u^\ep] \geq  \int \( 3^{4/3}
W(\varphi) + \frac{3^{2/3} }{8} (\bar \delta - \bar\delta_c)
\) d\widetilde P(\varphi, g).$$ Letting now $P^\ep$ and $P$ be the
first marginals of $\widetilde P^\ep$ and $\widetilde P$ respectively,
this proves \eqref{Lbound1} and the fact that $P$-almost every
$\varphi $ is in $\mathcal{A}_{m}$ with $m = 3^{-2/3} (\bar \delta
  - \bar \delta_c)$. \qed

%\footnote{for this the reference is
%(6.16) and its proof, in \cite{SS2}, more details should be given here}
%Finally we have
%\begin{align}
% \lim_{R \to +\infty} \frac{1}{|\UR|} \int \chi_R(x)dg(x) = \lim_{R \to + \infty} &\lim_{\varepsilon \to 0} \frac{1}{|\UR|}\int \chi_R(x-x_{\varepsilon})dg_{\varepsilon}(x)\\
%&\stackrel{(\ref{Wlbound2})}{\geq} \limsup_{R \to +\infty} \frac{W(j,\chi_R)}{|\UR|} +
%\frac{\mu 3^{2/3}}{4} = W(j)+ \frac{\mu 3^{2/3}}{4},
%\end{align}

\section{Upper bound construction: proof of Part ii) of Theorem~\ref{main}}\label{sec11}

%\subsection{Definition of the test profile}\label{deftestcurr}

We follow closely the construction performed for the magnetic
Ginzburg-Landau energy in \cite{SS2}, but our situation is somewhat
simpler, since we work on a torus (instead of a domain bounded by a
free boundary). The construction given in \cite{SS2} relies on a
result stated as Corollary 4.5 in \cite{SS2}, which we repeat below
with slight modifications to adapt it to our setting.  These results
imply, in particular, that the minimum of $W$ may be approximated by
sequences of periodic configurations of larger and larger
period. Below for any discrete set of points $\Lambda$, $|\Lambda|$
will denote its cardinal. 

\begin{pro}[Corollary 4.5 in \cite{SS2}]
  \label{mesyoung} Let $p \in (1,2)$ and let $P$ be a probability
  measure on $\dot W^{1,p}_{loc}(\mathbb{R}^2)$ which is invariant
  under the action of translations and concentrated on
  $\mathcal{A}_{1}$.  Let $Q$ be the push-forward of $P$ under
  $-\Delta$.  Then there exists a sequence $R \to \infty$ with $R^2
  \in 2\pi \mn$ and a sequence $\{b_R\}_{R}$ of $2R$-periodic
  vector fields  such that:
\begin{itemize}
\item[-] There exists a finite subset $\Lambda_R$ of
   the interior of $K_R$ such that
$$\left\{\begin{array}{ll}
- \div b_R =
 2\pi\sum_{a\in\Lambda_R}\delta_a - 1 & \text{ in }\  K_R\\
 b_R \cdot \nu=0 & \text{on} \ \p K_R.\end{array}\right.
 $$

\item[-] 
Letting $Q_R$ be the probability measure on
  $W^{-1,p}_{loc}(\mathbb{R}^2)$, which is defined as the image
  of the normalized Lebesgue measure on $K_R$ by $x\mapsto - \div
  b_R(x + \cdot )$,  we have $Q_R \to Q$ weakly as $R \to
  \infty$.
\item[-] $\D\limsup_{R\to\infty} \frac{1}{|K_R|} \lim_{\eta\to 0} \(
  \hal \int_{K_R \backslash \cup_a\in \Lambda_R B(a, \eta)} |b_R|^2 dx
  + \pi |\Lambda_R| \ln \eta\) \le\int W(\varphi) \,
  dP(\varphi)$. 
\end{itemize}
\end{pro}

\begin{remark}\label{rem72} We would like to make the following 
    observations concerning the vector field $b_R$ constructed in
    Proposition \ref{mesyoung}.
\begin{enumerate}\item
  By construction, the vector fields $b_R$ has no distributional
  divergence concentrating on $\partial K_R$ and its translated copied
  since $b_R \cdot \nu$ is continuous across $\partial K_R$. However,
  $b_R \cdot \tau$ may not be, and this may create a singular part of
  the distributional $\curl b_R$.  This is the difficulty that
  prevents us from stating the convergence result for $P$ directly in
  Theorem \ref{main}, Part ii).
  \item
We also note that an inspection of the construction in \cite{SS2} shows that $b_R $ is curl-free in a neighborhood of each point $a\in \Lambda_R$ and  that $\curl b_R$ belongs to   $W^{-1,p}_{\loc}(\mr^2)$ for $p<\infty$.
\end{enumerate}\end{remark}

\subsection{Definition of the  test configuration}

We take $R$ the sequence given by Proposition  \ref{mesyoung}. The first
thing to do is to change the density $1$ into a suitably chosen
density $m_{\ep,R}$, in order to ensure the compatibility of the
functions with the torus volume. Recalling that $\bar \mu^\eps >
  0$ for $\bar \delta > \bar \delta_c$ and $\eps$ small enough, we
set
\begin{equation}
  \label{mep}
  m_{\ep,R}= \frac{ 4 R^2}{|\ell^\ep|^2 }\left\lfloor
    \frac{\ell^\ep \sqrt{2 \bar{\mu}^\ep}  }  { 2  R
      \bar{r}_\ep}\right\rfloor^2  
\end{equation}
where, as usual, $\lfloor x \rfloor$ denotes the integer part of a
$x$.  We note for later that \begin{equation}
\label{mem}
\left| m_{\ep,R} - {2 \bar{\mu}^\ep \over \bar r_{\eps}^2 } \right| \le
\frac{CR}{\ell^\ep}
% + \frac{CR^2}{(\ell^\ep)^2} 
=o_\ep(1). 
\end{equation}
Recalling also that $\bar r_\eps = 3^{1/3} + O\(\frac{\ln |\ln
  \eps|}{|\ln \eps|}\)$ and $\bar{\mu}^\ep - \bar{\mu} = O\(\frac{\ln
  \lep}{\lep}\)$, we deduce that $m_{\ep, R} \to m$,  where $m :=
  2\cdot 3^{-2/3} \bar{\mu}$, as $\ep \to 0$, for each $R$.  In
  particular, $m_{\ep, R}$ is bounded above and below by
constants independent of $\ep $ and $R$.  The choice of $m_{\ep, R}$
ensures that we can split the torus into an integer number of
translates of the square $K_{R'}$ with $R' :=
\frac{R}{\sqrt{m_{\ep, R}} }$, each of which containing an
identical configuration of $2R^2\over  \pi$ points.
 
Let $P\in \mathcal{P}$ be given as in the assumption of Part 2 of
Theorem \ref{main}, i.e., let $P$ be a probability measure
concentrated on $\mathcal{A}_m$.  Letting $\bar{P}$ be the
push-forward of $P$ by $\vp \mapsto \vp(\frac{\cdot}{\sqrt{m}})$, it
is clear that $\bar{P}$ is concentrated on $\mathcal{A}_1$, and by the
change of scales formula \eqref{scalingW} we have
\begin{equation}\label{wpp}
  \int W(\vp) \, d\bar{P}(\vp) = \frac{1}{m}\int W(\vp) \, dP(\vp)  +
  \frac{1}{4} \ln m.
\end{equation} 
We may then apply Proposition  \ref{mesyoung} to $\bar{P}$.  It yields a
vector field $\bar{b}_R$. We may then rescale it by setting
$${b}_{ \ep, R}(x)= \sqrt{m_{\ep,R}} \, \bar{b}_R (\sqrt{m_{\ep,R}} x).$$
We note that $b_{\ep, R}$ is a well-defined periodic vector-field on
$\TTT$ because $\frac{\ell^\ep \sqrt{m_{\ep,R}}}{2 R} $ is an
integer.  This new vector field satisfies
\begin{equation}\label{eqsurb}
  - \div b_{\ep, R}= 2\pi \sum_{a \in \Lambda_{\ep, R} } \delta_a - m_{\ep,R} \quad
   \text{in} \ \TTT
\end{equation} for some set of points that we denote $\Lambda_{\ep,R}$, 
and
\begin{multline*}
  \frac{1}{|K_R|}\lim_{\eta\to 0} \left( \frac12
    \int_{K_{\frac{R}{\sqrt{m_{\ep,R}}}} \backslash \cup_{a \in 
        \Lambda_{\ep, R}} B(a, \eta) }|b_{\ep,R} |^2 dx + \pi
    |\Lambda_{\ep,R}\cap K_{R/\sqrt{m_{\ep,R}}}   |\ln (\eta \sqrt{m_{\ep,R}}) \right) \\ \le
  \int W(\vp) \, d\bar{P}(\vp) +o_R(1)\quad \text{as} \ R\to \infty.
\end{multline*} 
Using \eqref{wpp} and $| \Lambda_{\ep, R}\cap K_{R/\sqrt{m_{\ep, R}}
}|= \frac{2 R^2}{\pi}$, this can be rewritten as
\begin{multline*}
  \frac{m_{\ep,R}}{|K_R|}\lim_{\eta\to 0} \( \frac12
  \int_{K_{\frac{R}{\sqrt{m_{\ep,R}}}} \backslash \cup_{a \in
      \Lambda_{\ep,R}} B(a, \eta) } | b_{\ep,R} |^2 dx + \pi
  |\Lambda_{\ep,R} \cap K_{R / \sqrt{m_{\ep,R}}} |\ln \eta \)
  +\frac{m_{\ep,R}}{4} \ln \frac{ m_{\ep,R}}{m}\\
  \le\frac{m_{\ep,R}}{m} \int W(\vp) \, dP(\vp) +o_R(1).
\end{multline*} 
But we saw that $m_{\ep,R} \to m$ as $\ep \to 0$ hence $\ln \left(
  \frac{m_{\ep,R}}{m} \right) \to 0$. Therefore, recalling the
  definition of $R'$ we have
\begin{multline}\label{bbr}
  \frac{1}{|K_{R'}|} \lim_{\eta\to 0} \( \frac12 \int_{K_{R'}
    \backslash \cup_{a \in \Lambda_{\ep,R}} B(a, \eta) } |b_{\ep,R}|^2
  dx + \pi | \Lambda_{\ep,R}\cap K_{R/\sqrt{m_{\ep, R}} }|\ln \eta \)
  \\ \le \int W(\vp) \, dP(\vp) +o_R(1)+o_\ep(1).\end{multline} It
thus follows that
\begin{equation}\label{nrbcarre0}
  \frac{1}{(\ell^\eps)^2}\lim_{\eta\to
    0}\( \frac12 \int_{\TTT\backslash\cup_{a \in \Lambda_{\ep, R}}  B(a, \eta)
  } |b_{\ep, R}|^2  dx' + \pi |\Lambda_{\ep, R} | \ln \eta\)\le  \int W(\vp) \,
  dP(\vp) +o_R(1)+o_\ep(1) .
\end{equation} 

Note that $\Lambda_{\ep, R}$ is a dilation by the factor $1 /
\sqrt{m_{\ep, R}}$, uniformly bounded above and below, of the set of
points $\Lambda_R$, hence the minimal distance between the points in
$\Lambda_{\ep, R}$ is bounded below by a constant which may depend on
$R$ but does not depend on $\ep$.  For the same reason, estimates on
$b_{R, \ep}$ are uniform with respect to $\ep$.

In addition, we have that $\bar Q_{\ep, R}$, the push-forward of the
normalized Lebesgue measure on $\TTT$ by $x \mapsto - \div b_{\ep,R}
(x+\cdot)$ converges to $Q$, the push-forward of $P$ by $-\Delta$, as
$\eps \to 0$ and $R \to \infty$.  The final step is to replace the
Dirac masses appearing above by their non-singular approximations:
\begin{align} \label{reps} \tilde \delta_a:=
  \frac{\chi_{B(a,r_{\varepsilon}')}}{\pi |r_{\varepsilon}'|^2}\qquad
  r_{\varepsilon}':= \eps^{1/3} |\ln \eps|^{1/6} \bar r_{\eps},
\end{align} 
where $\bar r_\eps$ was defined in \eqref{newRdef1}. Note also that in
view of the discussion of Section \ref{setup} it is crucial to use
droplets with the corrected radius $\ep^{1/3} \lep^{1/6} \bar r_\ep$
instead of its leading order value $\ro_\ep = 3^{1/3}
\ep^{1/3} \lep^{1/6}$.
 
Once the set $\Lambda_{\ep, R} $ has been defined, the definition of
the test function $u^\ep \in \mathcal{A}$ follows: it suffices to take
$$
u^\ep ( x )= -1 + 2\sum_{a \in \Lambda_{\ep, R}}\chi_{B(a, r_\ep')}
\left( x |\ln \eps|^{1/2} \right) ,
$$ 
which means (after blow up) that all droplets are round of identical
radii $r_\ep'$ and centered at the points of $\Lambda_{\ep, R}$. We
now need to compute $F^\ep[u^\ep]$ and check that all the desired
properties are satisfied. This is done by working with the
associated function $h_\ep'$ defined in \eqref{heqn1}, i.e.  the
solution in $\TTT$ to
\begin{equation}\label{vpconstr} 
  -\Delta h_\ep' +
  \frac{\kappa^2}{\lep} h_\ep'= \pi \bar{r}_\ep^2\sum_{a \in
    \Lambda_{\ep, R}} \tilde{\delta}_a - \bar{\mu}^\ep,
\end{equation} 
obtained from \eqref{heqn1} by explicitly setting all $A_i^\eps =
  \pi \bar r_\eps^2$.
 
\subsection{Reduction to   auxiliary functions}

Let us introduce $\phi_\ep$, which is the solution with mean zero of
\begin{equation}\label{phiepsi}
  -\Delta \phi_{\varepsilon} =
  2\pi\sum_{a \in\Lambda_{\eps,R}} \tilde \delta_a - m_{\eps,R}
 \  \text{ in }\   \TTT,
\end{equation} 
where $m_{\ep,R}$ is as in \eqref{mep}, and $f_\ep$ the solution with
mean zero of
\begin{equation}\label{efff}
-\Delta f_\ep = 2\pi \sum_{a\in\Lambda_{\ep, R}} \delta_a - m_{\ep, R}
\ \text{ in }\   \TTT.
\end{equation}
We note that $f_\ep$ is a rescaling by the factor $m_{\ep, R} \to m$
of a function independent of $\ep$, so all estimates on $f_\ep$ can be
made uniform with respect to $\ep$.
\begin{lem}\label{lemev}
  Let $h_\ep'$ and $\phi_\ep$ be as above. We have as $\ep \to 0$
 \begin{equation}\label{phil2}
 \int_{\TTT} |h_\ep' |^2 dx' \le C_R  \lep
 \end{equation} 
 and for any $1\le q<\infty$ 
 \begin{equation}\label{vpp}
   \left\| \nab\(h_\ep' - \frac{\bar{r}_\ep^2
     }{2}\phi_\ep\) \right\|_{L^{q}(\TTT)}
   \le C_{R,q},
 \end{equation}  for some constant $C_{R,q}>0$ independent of $\ep$.
\end{lem}
\begin{proof}
  Since $\Lambda_{\ep, R} $ is $2R'$-periodic, $h_\ep'$ is too,
  and thus
$$\int_{\TTT} |h_\ep' |^2 dx' = \ell^2 \lep \dashint_{K_{R'}}
|h_\ep'|^2 dx' \le C_R \lep.$$

For the second assertion, let $$h_\ep(x)= h_\ep'(x \sqrt{\lep} )
\qquad \hat{\phi}_\ep (x)= \phi_\ep(x \sqrt{\lep}) $$ be the
rescalings of $h_\ep'$ and $\phi_\ep$ onto the torus $\TT$. Rescaling
\eqref{phil2} gives
\begin{equation}\label{pjil}
\|h_\ep\|_{L^2(\TT)} \le C_R.
\end{equation}
Furthermore, the function $w_\ep {:= h_\ep - \frac12
  \bar{r}_\ep^2 \hat \phi_\ep}$ is easily seen to solve
\begin{align*}
  - \Delta w_\ep = - \kappa^2 \( h_\ep    - \dashint_{\TT} h_\eps
      dx \right)  \quad \text{in} \ \TT.
\end{align*} 
But from elliptic regularity, Cauchy-Schwarz inequality and
\eqref{pjil}, we must have
$$
\|\nab w_\ep\|_{L^{q}(\TT)} \le C \left\|h_\ep - \dashint_{\TT} h_\eps
\right\|_{L^2(\TT)} \le C_{R,q},$$ which yields \eqref{vpp}.
\end{proof}

The next lemma  consists in comparing $\phi_\ep$ and $f_\ep$.
\begin{lem}\label{lemfphi}
We have 
$$  \|\nab (f_\ep- \phi_\ep)\|_{L^\infty\(\TTT   \sm \cup_a B(a,
  r_\ep') \)}\le C_R \ep^{1/4}.
$$
\end{lem}

\begin{proof}
  We observe that $f_\ep $ and $\phi_\ep$ are both $2 R'$-periodic. We
  may thus write
 $$
 \phi_\ep (x) -f(x)= 2\pi \int_{\mathbb{T}^2_{2 R'} }G_{2
   R'}(x-y) \sum_{a \in \Lambda_{\ep, R} }d(\tilde{\delta}_a
 -\delta_a)(y) ,
  $$
  where $G_{2 R'}$ is the zero mean Green's function for the
  Laplace's operator on the square torus of size $2 R'$ with
  periodic boundary conditions, i.e. the solution to
  \begin{equation} \label{715} -\Delta G_{2 R'}= \delta_0-
    \frac{1}{|\mathbb{T}^2_{2R'} | } \quad \text{in} \
    \mathbb{T}^2_{2 R'}
  \end{equation} 
  which we may be split as $G_{2R'}(x)= -\frac{1}{2\pi} \log |x| +S_{2
    R'}(x)$ with $S_{2 R'}$ a smooth function.  By Newton's theorem
  (or equivalently by the mean value theorem for harmonic functions
  applied to the function $\log |\cdot |$ away from the origin), the
  contribution due to the logarithmic part is zero outside of $\cup_{a
    \in \Lambda_{\ep, R}} B(a, r_\ep')$. Differentiating the above we
  may thus write that for all $x \notin \cup_{a \in \Lambda_{\ep, R}}
  B(a , r_\ep')$,
  \begin{equation}\label{phif}
    \nab  (\phi_\ep -f)(x) =2\pi  \int_{\mathbb{T}^2_{2 R'} }
    \nab S_{2R'}  (x-y) 
    \sum_{a \in\Lambda_{\ep, R}} d(\tilde{\delta}_a -\delta_a)(y) .
  \end{equation}
  Using the $C^2$ character of $S_{2 R'}$ we deduce that
 $$  \|\nab (f_\ep- \phi_\ep)\|_{L^\infty\(\TTT   \sm \cup_a B(a,
   r_\ep') \)}\le C_{R'} |\Lambda_{\ep, R}\cap K_{R'}| r_\ep'$$
 and the result follows in view of \eqref{reps}.
\end{proof}

The next step involves a comparison of the energy of $\phi_\ep$ and
that of $b_{\ep, R}$ and leads to the following conclusion.

\begin{lem}\label{lemU1.1} Given $\Lambda_{\ep, R}$ as constructed
  above, and $h_\ep'$ the solution to \eqref{vpconstr}, we have
  \begin{equation*}\frac{1}{(\ell^\eps)^2} \lim_{\eta \to 0} \(\int_{\TTT
      \backslash \bigcup_{a \in \Lambda_{\ep, R} } B(a,\eta)}
    \frac{2}{\bar{r}_\ep^4}{|\nabla h_\ep 
      '|^2} dx' + \pi |\Lambda_{\ep, R}| \ln \eta \) \leq \int
    W(\vp) \, dP(\vp) +o_\ep(1)+ o_R(1).
\end{equation*}
\end{lem}

\begin{proof} In view of Lemmas \ref{lemev} and \ref{lemfphi}, it
  suffices to show the corresponding result for $ \int_{\TTT
    \backslash \bigcup_{a \in \Lambda_{\ep, R} } B(a,\eta)}
  \frac{1}{2}|\nabla f_\ep|^2\, dx' $ instead of the one for $h_\ep'$.
  From \eqref{efff} and \eqref{eqsurb}, we have $\div (b_{\ep, R} -
  \nab f_\ep) = 0 $ hence by Poincar\'e's lemma we may write $ \nab
  f_\ep= b_{\ep, R}+ \np \xi_\ep$. We note that $-\Delta \xi_\ep=
  \curl b_{\ep,R}$, which is in $W^{-1,p}_{loc}$ for any
  $p<+\infty$ as mentioned in Remark \ref{rem72}.  By elliptic
  regularity we find that $\nab \xi_\ep \in L^p_{loc}(\mr^2)$ for
  all $1\le p<+\infty$, uniformly with respect to $\ep$.  We may
  thus write
\begin{multline}
  \label{fxieps}
  \int_{\TTT \backslash \bigcup_{a \in \Lambda_{\ep, R}} B(a,\eta)}
  \frac{1}{2}|b_{\ep, R}|^2 dx' \\
  = \int_{\TTT \backslash \bigcup_{a \in \Lambda_{\ep, R}} B(a,\eta)}
  \left( \hal |\nab f_\ep|^2 + \hal |\nab \xi_\ep|^2 - \nab f_\ep\cdot
    \np \xi_\ep \right) dx' ,
\end{multline}
where $\nab f_\ep \cdot \np\xi_\ep$ makes sense in the duality $\nab
\xi_\ep \in L^p$, $p>2$, $\nab f_\ep \in L^q$, $q<2$. 
In addition, by the same duality, we have  for any $a \in \Lambda_{\ep, R}$, 
$$\lim_{\eta \to 0} \int_{B(a,\eta)}  \nab f_\ep\cdot
\np \xi_\ep=0$$ uniformly with respect to $\ep$. 
%Indeed, $\nab f_\ep+
%\frac{x-a}{|x-a|^2}$ is smooth in a neighborhood of $a$ \cm{needed? if
  %$\nabla \xi_\eps \in L^p$ with $p$ bigger than the conjugate of $q$
  %from $\nabla f_\eps \in L^q$, isn't that all you need?}, while
%$\ed{\nabla} \xi_{\ed \eps} \in L^p$ with $p>2$, and all these
%estimates can be made uniform in $\ep$, as mentioned previously.
Therefore, we may extend the domain of integration in the last
integral in \eqref{fxieps} to the whole of $\TTT$ at the expense of an
error $o_\eta(1)$ multiplied by the number of points, and obtain
\begin{multline}
  \int_{\TTT \backslash \bigcup_{a \in \Lambda_{\ep, R}} B(a,\eta)}
  \frac{1}{2}|\nabla f_\ep|^2 dx' \le \int_{\TTT \backslash \bigcup_{a
      \in \Lambda_{\ep, R}} B(a,\eta)} \hal |b_{\ep, R}|^2 dx' \\
  + \int_{\TTT} \nab f_\ep \cdot \np \xi dx' + o_\eta(\lep).
\end{multline}
  Noting that the last integral on the right-hand side vanishes by Stokes' theorem (and
    by approximating $\nab f_\ep $ and $\np \xi_\ep$ by smooth
    functions),  adding $\pi|\Lambda_{\ep, R}| \ln
  \eta$ to both sides, and combining with \eqref{nrbcarre0} we obtain
  the result.
\end{proof}

In view of \eqref{eqsurb} and \eqref{phiepsi} we have that $-\div
b_{\ep, R} + \Delta \phi_\ep = 2\pi \sum_{a \in \Lambda_{\ep, R}
}(\delta_a - \tilde \delta_a) \to 0$ in $W^{-1,p}_{loc} (\mr^2)$, so
we deduce, since the push-forward of the normalized Lebesgue measure
on $\TTT$ by $x \mapsto -\div b_{\ep, R} (x+\cdot)$ converges to $Q$,
that the push-forward of it by $ x \mapsto - \Delta \vp^\ep(x+ \cdot
)$ also converges to $Q$. Thus, part ii) of Theorem \ref{main} is
established modulo \eqref{Ubound1}, which remains to be proved.

\subsection{Calculating the energy}

We begin by calculating the exact amount of energy contained in a ball
of radius $\eta$.
%\ed{
\begin{lem}\label{energyballs}   Let $h_\ep' $ be as above. Then we
  have for any $a \in \Lambda_{\ep, R}$,
\begin{equation}\label{eqint}
  \int_{B(a,r_{\varepsilon}')} |\nabla h_\ep' |^2  dx' =
  \frac{3^{4/3} \pi }{8} + o_{\varepsilon}(1)
\end{equation} 
and 
\begin{equation}\label{eqann}
  \int_{B(a,\eta)\backslash B(a,r_{\varepsilon}')} |\nabla
  h_\ep' |^2  dx' \leq \frac{\pi}{2} \bar{r}_\ep^4  \ln \frac{\eta}{\rho_{\eps}} +
  o_{\varepsilon}(1)+o_\eta(1).\end{equation} 
\end{lem}
\begin{proof} 
  In view of \eqref{vpp} applied with $q>2$ and using H\"older's
  inequality, we have that for all $a \in \Lambda_{\ep, R}$,
\begin{equation}\label{beeta}
\int_{B(a, \eta)}\left|\nab (h_\ep'- \frac{\bar{r}_\ep^2}{2}
  \phi_\ep)\right|^2 \, dx' \le o_\eta(1).
\end{equation}
Thus it suffices to compute the corresponding integrals for
$\phi_\ep$.  Using again the $2R'$-periodicity of $\phi_\ep$, we
may write, with the same notation as in the proof of Lemma
\ref{lemfphi}
$$\phi_\ep(x)=  \int_{\mathbb{T}^2_{2 R'} } G_{2 R'} (x-y)\(
2\pi  \sum_{a \in \bar\Lambda_{\ep, R}} \tilde{\delta}_a (y)- m_{\ep,
  R}  \)\, dy. 
$$
Since the distances between the points in $\bar\Lambda_{\ep, R}$ are
bounded below independently of $\ep$, and the number of points is
bounded as well, we may write $\phi_\ep$ in $B(a, \eta)$ as
\begin{equation}\label{logterm}
  \phi_\ep  (x)=\psi_\ep(x) -\int_{\mathbb{T}^2_{2R'} } \ln |x-
  y| \, \tilde{\delta}_a(y)\, dy
\end{equation}
where $\psi_\ep(x)$ is smooth and its derivative is bounded
independently of $\ep$ (but depending on $R$).

Thus the contribution of $\psi_\ep$ to the integrals $\int_{B(a,\eta)}
|\nab \phi_\ep|^2 $ is $o_\eta(1)$, and its contribution to
$\int_{B(a,r_\ep')} |\nab \phi_\ep|^2$ is $o_\ep(1)$.  There remains
to compute the contribution of the logarithmic term in
\eqref{logterm}. But this is almost exactly the same computation as in
\eqref{lecerclelastlast}--\eqref{lecerclelastlast5}, and with
\eqref{beeta} it yields \eqref{eqint}, while it yields as well that
\begin{equation}\label{UBound11} \int_{B(a,\eta)\backslash
    B(a,r_{\varepsilon}')} |\nabla \phi_\ep |^2  dx' \leq 2 \pi
  \ln \frac{\eta}{r_{\varepsilon}'} +
  o_\eta(1).
\end{equation} 
Now
\[ \frac{r_{\varepsilon}' }{\rho_{\eps}} = \frac{1}{3^{1/3}}
\(\frac{|\ln \varepsilon|}{|\ln \rho_{\eps}|}\)^{1/3} = \(1 +
O\(\frac{\ln|\ln \varepsilon|}{|\ln \varepsilon|}\)\)^{1/3}.\]
Consequently $\ln \frac{r_{\varepsilon}' }{\rho_{\eps}} =
o_{\varepsilon}(1)$, and so we may replace $r_{\varepsilon}'$ with
$\rho_{\eps}$ at an extra cost of $o_{\varepsilon}(1)$ in
\eqref{UBound11}, and the result follows with \eqref{beeta}.
\end{proof}

%The last thing we must prove is that we were justified in ignoring the screening term in the previous propositions. Fortunately
%this is an easy consequence of periodicity.
%\begin{lem} (Decay of screening)
%Let $\phi_{\varepsilon}$ be as in Lemma \ref{lemU1.1}. Then we have
%\[ \frac{1}{|\ln \varepsilon|^2} \int_{\mathbb{T}_{\ell^{\eps}}^2} |\phi_{\varepsilon}|^2 \to 0,\]
%as $\varepsilon \to 0$. \end{lem}
%\begin{proof}
%By periodicity we have for fixed $R>0$
%\[\frac{1}{|\ln \varepsilon|^2} \int_{\mathbb{T}_{\ell^{\eps}}^2} |\phi_{\varepsilon}|^2 = \frac{1}{|\ln \varepsilon|} \int_{K_R} %|\phi_{\varepsilon}|^2 \leq \frac{C(R)}{|\ln \varepsilon|},\]
%which gives the result.
%\end{proof}

We can now combine all the previous results to compute the energy of
the test-function $u^\ep$.  By following the lower bounds of
Proposition \ref{lem1}, it is easy to see that in our case (all the
droplets being balls of radius $r_\ep')$ all the inequalities in that
proof become equalities, and thus recalling \eqref{newRdef1}:
\begin{align*} F^{\eps}[u^{\eps}] = \frac{1}{
    |\ell^{\eps}|^2}\(2\int_{\mathbb{T}_{\ell^{\eps}}^2} \( |\nabla
  \h|^2 + \frac{\kappa^2}{|\ln \varepsilon|} |\h|^2 \) dx' + \pi \bar
  r_{\eps}^4 |\Lambda_{\ep, R}| \ln \rho_{\eps}\) +
  o_{\eps}(1),
\end{align*} 
with the help of Lemma \ref{energyballs} we have for every $R$
\begin{align*} F^{\eps}[u^{\eps}] & \le
  \frac{1}{|\ell^{\eps}|^2}\(2\int_{\mathbb{T}_{\ell^{\eps} }^2\sm
    \cup_{a\in \Lambda_{\ep , R} } B(a, \eta) } |\nabla h_\ep' |^2\,
  dx' + \pi \bar r_{\eps}^4 |\Lambda_{\ep, R}| \ln \eta +
  \frac{3^{4/3} \pi}{ 4} |\Lambda_{\ep, R}| \) +
  o_{\eps}(1)+o_\eta(1).
\end{align*}
In view of Lemma \ref{lemU1.1}, letting $\eta \to 0$, we obtain
  $$
  F^\ep[u^\ep]\le \bar {r}_{\eps}^4\( \int W(\vp) \,
  dP(\vp)+o_\ep(1)+o_R(1)\) +\frac{ 3^{4/3} \pi}{ 4 |\ell^\ep|^2}
  |\Lambda_{\ep, R}| + o_\ep(1).
  $$ 
  Letting $\ep \to 0$, using that $\bar{r}_\ep \to 3^{1/3}$ and the
  fact that $|\Lambda_{\ep, R}|= \frac{1}{2\pi} m_{\ep,R}
  |\ell^\ep|^2$ with $m_{\ep, R} \to m$, and then finally letting $R
  \to \infty$, we conclude that
$$
\limsup_{\ep \to 0} F^\ep[u^\ep] \le 3^{4/3}\int W(\vp) \, dP(\vp) +
\frac{3^{4/3} m }{8} .
$$
Since $\frac18 3^{2/3}m= \frac18 (\bar{\delta}-\bar{\delta_c})$,
this  completes  the proof of part ii) of Theorem \ref{main}. \qed \\

\subsection{Proof of Theorem \ref{th2}}
\label{sec:proof-theorem-refth2}

In order to prove Theorem \ref{th2}, it suffices to show that 
\begin{equation}\label{claim4}
  \min_{P \in \mathcal{P}} F^0 [P]= 3^{4/3}  \min_{\mathcal{A}_m}
  W+\frac{3^{2/3}(\bar{\delta}- \bar{\delta}_c)}{8}.
\end{equation} 
For the proof, we use the following result, adapted from  Corollary 4.4 in \cite{SS2}.

\begin{proposition}[Corollary 4.4 in \cite{SS2}]\label{pw2}  Let $\varphi\in
  \mathcal{A}_{1} $ be given, such that $W(\varphi) <\infty$.  For any
  $R$ such that $R^{ 2} \in 2\pi \mn,$ there exists a  $2R$-periodic $\vp_R$ such that $$
  \left\{ \begin{array}{ll} - \Delta \vp_R =
      2\pi\sum_{a\in\Lambda_R}\delta_a -1\ & \text{ in} \ K_{R },\\ [1mm]
      \D \frac{\p \vp_R}{\p \nu} = 0 & \text{on} \ \p K_R,
 \end{array}\right.$$
where $\Lambda_R$ is a finite subset of the interior of $K_R$, and
such that $$\D\limsup_{R\to\infty} \frac{W(\vp_R,\indic_{K_R})}{|K_R|}
\le W(\varphi).$$
\end{proposition}
Let us take $\vp$ to be a minimizer of $W$ over $\mathcal{A}_{m}$
(which exists from \cite{SS2}).  We may rescale it to be an element of
$\mathcal{A}_1$. Then Proposition \ref{pw2} yields a $\vp_R$, which
can be extended periodically. We can then repeat the same construction
as in the beginning of this section, starting from $\nab \vp_R$
instead of $b_R$, and in the end it yields a $u^\ep$ with
 $$\limsup_{\ep \to 0} F^\ep[u^\ep] \le   3^{4/3}
 \min_{\mathcal{A}_m} W+\frac{3^{2/3}(\bar{\delta}-
   \bar{\delta}_c)}{8}.$$ It follows that $$\limsup_{\ep \to 0}
 \min_{\mathcal{A}} F^\ep \le 3^{4/3} \min_{\mathcal{A}_m}
 W+\frac{3^{2/3}(\bar{\delta}- \bar{\delta}_c)}{8}.$$ But by part i)
 of Theorem \ref{main} applied to a sequence of minimizers of $F^\ep$,
 we also have
$$\liminf_{\ep \to 0} \min_{\mathcal{A}}F^\ep  \ge \inf_{\mathcal{P}}
F^0 \ge 3^{4/3}  \min_{\mathcal{A}_m} W+\frac{3^{2/3}(\bar{\delta}-
  \bar{\delta}_c)}{8}$$ where the last inequality is an immediate
consequence of the definition of $F^0$. Comparing the inequalities
yields that there must be equality and \eqref{claim4} is proved, which
completes the proof of Theorem \ref{th2}. \qed

\vskip 1cm

\noindent \textbf{Acknowledgments} The research of D. G.  was
partially supported by an NSERC PGS D award. The work of C. B. M. was
supported, in part, by NSF via grants DMS-0718027 and DMS-0908279. The
work of S. S. was supported by a EURYI award. C. B. M. would like to
express his gratitude to M. Novaga and G. Orlandi for valuable
discussions.

% \bibliography{../nonlin}

\bibliographystyle{plain}

\noindent {\sc Dorian Goldman\\
  Courant Institute of Mathematical
  Sciences, New York, NY 10012, USA,\\
  \& UPMC Univ  Paris 06, UMR 7598 Laboratoire Jacques-Louis Lions,\\
  Paris, F-75005 France ;\\
  CNRS, UMR 7598 LJLL, Paris, F-75005 France \\
  {\tt      dgoldman@cims.nyu.edu} \\
  \\
  {\sc Cyrill B. Muratov} \\
  Department of Mathematical Sciences, \\ New Jersey Institute of
  Technology, \\ Newark,
  NJ 07102, USA \\
  {\tt muratov@njit.edu}
  \\  \\
  {\sc Sylvia Serfaty}\\
  UPMC Univ  Paris 06, UMR 7598 Laboratoire Jacques-Louis Lions,\\
  Paris, F-75005 France ;\\
  CNRS, UMR 7598 LJLL, Paris, F-75005 France \\
  \&  Courant Institute, New York University\\
  251 Mercer st, NY NY 10012, USA\\
  {\tt serfaty@ann.jussieu.fr}

\end{document}